\documentclass[10pt]{amsart} 
\usepackage{amscd,amssymb}
\usepackage{enumerate} 
\usepackage[all]{xypic} 
\usepackage{multirow}
\usepackage{hhline}
\usepackage{verbatim} 
\usepackage{mathdots} 

\theoremstyle{plain} 
\newtheorem{Thm}[equation]{Theorem}
\newtheorem*{Thm*}{Theorem} 

\newtheorem{Prop}[equation]{Proposition}
\newtheorem{Lem}[equation]{Lemma} 
\newtheorem{Rmk}[equation]{Remark}

\numberwithin{equation}{section}

\newcommand{\GL}{\operatorname{GL}}

\newcommand{\GSp}{\operatorname{GSp}}

\newcommand{\Ind}{\operatorname{Ind}}
\newcommand{\ind}{\operatorname{ind}}

\newcommand{\supp}{\operatorname{supp}}

\newcommand{\GLt}{\widetilde{\operatorname{GL}}}
\newcommand{\SLt}{\widetilde{\operatorname{SL}}}
 
\newcommand{\Bt}{\widetilde{B}}
\newcommand{\Tt}{\widetilde{T}} 
\newcommand{\Pt}{\widetilde{P}}
\newcommand{\Qt}{\widetilde{Q}} 

\newcommand{\MPt}{\widetilde{M}_P} 
\newcommand{\MQt}{\widetilde{M}_Q}
\newcommand{\e}{\operatorname{e}}

\newcommand{\GLtt}{\widetilde{\operatorname{GL}}^{(2)}}

\newcommand{\Tte}{\widetilde{T}^{\e}}
\newcommand{\Ttt}{\widetilde{T}^{(2)}}
\newcommand{\Ptt}{\widetilde{P}^{(2)}}

\newcommand{\MPtt}{\widetilde{M}_P^{(2)}}

\newcommand{\MQtt}{\widetilde{M}_Q^{(2)}}

\newcommand{\sigGLt}{{^{\sigma}\widetilde{\operatorname{GL}}}}

\newcommand{\GLtwo}{\GL^{(2)}}

\newcommand{\iQ}{{^\iota Q}}

\newcommand{\iMQt}{\widetilde{M}_{^\iota Q}}
\newcommand{\itheta}{{^\iota\theta}}

\newcommand{\Mt}{\widetilde{M}}

\newcommand{\At}{\widetilde{\mathbb{A}}}
\newcommand{\Ht}{\widetilde{H}}

\newcommand{\Tm}{T^{\operatorname{m}}}
\newcommand{\Tmt}{\widetilde{T}^{\operatorname{m}}}
\newcommand{\m}{{\operatorname{m}}}

\newcommand{\Mtt}{\widetilde{M}^{(2)}}

\newcommand{\timest}{\widetilde{\times}}
\newcommand{\otimest}{\,\widetilde{\otimes}\,}
\newcommand{\otimestt}{\widetilde{\otimes}}

\newcommand{\pin}{\pi^{(2)}}

\newcommand{\etat}{\tilde{\eta}}
\newcommand{\chit}{\tilde{\chi}}
\newcommand{\iotat}{\tilde{\iota}}

\newcommand{\rr}{\mathbf{r}}

\newcommand{\OF}{\mathcal{O}_F}

\newcommand{\Res}{\operatorname{Res}}

\newcommand{\rec}{\operatorname{rec}}
\renewcommand{\Re}{\operatorname{Re}}

\newcommand{\C}{\mathbb C}
\newcommand{\A}{\mathbb{A}}

\newcommand{\Z}{\mathbb{Z}}
\newcommand{\R}{\mathbb{R}}

\newcommand{\Th}{{\text{th}}}

\newcommand{\Id}{\operatorname{Id}}

\newcommand{\s}{\mathbf{s}}

\newcommand{\ie}{{\em i.e. }}

\title[Metaplectic Eisenstein series and symmetric square]{On a certain metaplectic
  Eisenstein series and the twisted symmetric square $L$-function} \author{Shuichiro Takeda}
\address{Department of Mathematics, University of Missouri,
  Columbia, MO 65211}

\begin{document}

\maketitle

\begin{abstract} 
In our earlier paper, based on a paper by Bump and Ginzburg, we used
an Eisenstein series on the double cover of $\GL(r)$ to obtain the
integral representation of the twisted symmetric square $L$-function
of $\GL(r)$. Using that, we showed that the (incomplete) twisted
symmetric square $L$-function of $\GL(r)$ is holomorphic for
$\Re(s)> 1$. In this paper, we will determine the possible poles of
this Eisenstein series more precisely and show that the (incomplete) twisted
symmetric square $L$-function is entire except possible simple poles at $s=0$ and $s=1$.
\end{abstract}

%%%%%%%%%%%%%%%%%%%%%%%%%%%%%%%%%%%%%%%%%%%%%%%%%%%%%%%%%%%%%%%%%%%

\section{\bf Introduction}

%%%%%%%%%%%%%%%%%%%%%%%%%%%%%%%%%%%%%%%%%%%%%%%%%%%%%%%%%%%%%%%%%%%

Let $\pi\cong\otimes'_v\pi_v$ be an irreducible cuspidal automorphic
representation of $\GL_r(\A)$ and $\chi$ a unitary Hecke character on
$\A^\times$, where $\A$ is the ring of adeles over a number field
$F$. By the local Langlands correspondence by Harris-Taylor \cite{HT}
and Henniart \cite{He}, each $\pi_v$ corresponds to an $r$-dimensional
representation $\rec(\pi_v)$ of the Weil-Deligne group $WD_{F_v}$ of
$F_v$. We can also consider the twist of $\rec(\pi_v)$ by $\chi_v$, namely,
\[ \rec(\pi_v)\otimes\chi_v:WD_{F_v}\rightarrow \GL_r(\C),
\] where $\chi_v$ is viewed as a character of $WD_{F_v}$ via local
class field theory. Now for each homomorphism
\[ \rho:\GL_r(\C)\rightarrow\GL_N(\C),
\] one can associate the local $L$-factor
$L_v(s,\pi_v,\rho\circ\rec(\pi_v)\otimes\chi_v)$ of Artin type. Then
one can define the automorphic $L$-function by
\[
L(s,\pi,\rho\otimes\chi):=\prod_vL_v(s,\pi_v,\rho\circ\rec(\pi_v)\otimes\chi_v).
\] In particular in this paper, we consider the case where $\rho$ is
the symmetric square map
\[ Sym^2:\GL_r(\C)\rightarrow\GL_{\frac{1}{2}r(r+1)}(\C),
\] namely we consider the twisted symmetric square $L$-function
$L(s,\pi,Sym^2\otimes\chi)$. By the Langlands-Shahidi method, it can
be shown that the $L$-function $L(s,\pi,Sym^2\otimes\chi)$ admits
meromorphic continuation and a functional equation. 
(See \cite[Theorem 7.7]{Sh90}.)

The Langlands-Shahidi method, however, is unable to determine the
locations of the possible poles of $L(s,\pi,Sym^2\otimes\chi)$. The
main theme of this paper is to determine them
though we consider only the incomplete $L$-function
$L^S(s,\pi,Sym^2\otimes\chi)$. To be more specific, let $S$ be a
finite set of places that contains all the archimedean places and
non-archimedean places where $\pi$ or $\chi$ ramifies. For $v\notin
S$, each $\pi_v$ is parameterized by a set of $r$ complex numbers
$\{\alpha_{v,1},\dots,\alpha_{v,r}\}$ known as the Satake
parameters. Then we have
\[ 
L_v(s,\pi_v,Sym^2\otimes\chi_v)=\prod_{i\leq
j}\frac{1}{(1-\chi_v(\varpi_v)\alpha_{v,i}\alpha_{v,j}q_v^{-s})},
\] 
where $\varpi_v$ is the uniformizer of $F_v$ and $q_v$ is the order
of the residue field, and we set
\[ 
L^S(s,\pi,Sym^2\otimes\chi)=\prod_{v\notin
S}L_v(s,\pi_v,Sym^2\otimes\chi_v).
\] 
As our main theorem (Theorem \ref{T:main3}) we will prove

\begin{Thm*}
Let $\pi$ be a cuspidal
automorphic representation of $\GL_r(\A)$ with unitary central character
$\omega_\pi$ and $\chi$ a unitary Hecke character. Then the incomplete
twisted symmetric square $L$-function $L^S(s,\pi,Sym^2\otimes\chi)$ is
holomorphic everywhere except that it has a possible pole at $s=0$ and
$s=1$. Moreover there
is no pole if $\chi^r\omega_\pi^2\neq 1$. (Here the set $S$ can be
taken to be exactly the finite set
of places containing all the archimedean places, places dividing 2, and the
non-archimedean places where $\pi$ or $\chi$ is ramified.)
\end{Thm*}

Indeed, in our previous work (\cite{Takeda1}), which is based on the
work by Bump and Ginzburg (\cite{BG}), which is in turn based on works
of various people such as Patterson and Piatetski-Shapiro (\cite{PP}),
Gelbart and Jacquet (\cite{GJ}) and most originally Shimura
(\cite{Shimura}), we showed the $L$-function $L^S(s,\pi,Sym^2\otimes\chi)$
is holomorphic for $\Re(s)>1$. (Actually what we showed in
\cite{Takeda1} is slightly more than this. See \cite{Takeda1} for more
details.) In \cite{Takeda1}, however, we were unable to show the
holomorphy for $\Re(s)<1$. This was because we were unable to
determine the locations of possible poles of certain Eisenstein series
on the metaplectic double cover $\GLt_r$ of $\GL_r$ for all $s\in\C$. For the sake of
explaining it, let us assume $r$ is odd here. Then in \cite{Takeda1},
the twisted symmetric square $L$-function
$L^S(s,\pi,Sym^2\otimes\chi)$ is represented by Rankin-Selberg
integrals of the form
\[
Z(\phi, \Theta,
f^s)=\int_{Z(\A)\GL_r(F)\backslash\GL_r(\A)}\phi(g)\theta(\kappa(g))E(\kappa(g),s;
f^s)\,dg,
\]
where $\phi$ is a cusp form in $\pi$, $\Theta$ is an automorphic form
on the twisted exceptional representation of $\GLt_r(\A)$, and
$E(-,s;f^s)$ is the Eisenstein series on $\GLt_r(\A)$ associated with
the section $f^s$ in the global induced space
$\Ind_{\Qt(\A)}^{\GLt_r(\A)}\theta\otimes\delta_Q^s$, where $Q$ is
the $(r-1,1)$-parabolic of $\GL_r$ and $\theta$ is the exceptional
representation of the Levi part
$\GLt_{r-1}(\A)\timest\GLt_1(\A)$. (Those exceptional representations
will be recalled in later sections.) Then the holomorphy of the
twisted symmetric square $L$-function can essentially be reduced to
the holomorphy of the normalized Eisenstein series
\[
E^{\ast}(-,s;f^s)=L^S(r(2s+\frac{1}{2})\chi^r\omega_\pi^2) E(-,s;f^s).
\]
Indeed, the bulk of this paper is devoted to showing the following
result on the normalized Eisenstein series, which is  Theorem
\ref{T:main2} with the notation adjusted. 
\begin{Thm*}
The normalized Eisenstein series above is holomorphic for all $s\in\C$
except that if $\chi^r\omega_\pi^2=1$ it has a possible simple pole at
$s=\frac{1}{4}$ and $-\frac{1}{4}$.
\end{Thm*}

Let us note that the possible pole at $s=\frac{1}{4}$
(resp. $s=-\frac{1}{4}$) for the normalized Eisenstein series gives
the one at $s=1$ (resp. $s=0$) for the $L$-function. 

\quad\\

Determination of the
location of possible poles of  (normalized) Eisenstein series
(especially degenerate Eisenstein series for classical groups) has
been done in various places such as \cite{PSR, GPSR, KR, Ikeda,
Jiang}, and we essentially follow their approach, in which we
determine possible poles of the Eisenstein series by computing
a constant term of the Eisenstein series and poles of intertwining
operators. Our Eisenstein series, however, is on the metaplectic group
$\GLt_r(\A)$, which requires extra care, and for this reason we have
developed the theory of metaplectic tensor products for automorphic
representations in our earlier paper \cite{Takeda2}.

Even though the theory of metaplectic groups is an important
subject in representation theory and automorphic forms, it has an
unfortunate history of numerous technical errors and as a result published
literatures in this area are often marred by those
errors which compromise their reliability. For this reason, we try to
make this paper as self-contained as possible and supply as detailed
proofs as possible. In particular, we will not use any of the results
in \cite{BG} (though many of the ideas in this paper are borrowed from \cite{BG})
except one proposition (\cite[Proposition 7.3]{BG}) on $\GLt_2$ for
which the proof there is detailed enough to be reliable.\\

The following is the structure of the paper. In the next section, we
will recall the theory of the metaplectic double cover $\GLt_r$ of
$\GL_r$ both locally and globally and quote the results from
\cite{Takeda2} on the metaplectic tensor product, which will be needed
in later sections. In Section 3, we recall the notion of the
exceptional representation on $\GLt_r$, which was originally developed
in \cite{KP} for the non-twisted case, \cite{Banks} for the local
twisted case, and finally in \cite{Takeda1} for the general case. The
exceptional representation is used to define our Eisenstein
series. In Section 4, we define the induced representation that gives
rise to our Eisenstein series, and examine analytic properties of the
intertwining operators on it, and in Section 5 we will determine the
possible poles of the (unnormalized) Eisenstein series for $\Re(s)\geq
0$. Those two sections comprise the main part of the paper. Then in
Section 6, we will determine the possible poles of the normalized
Eisenstein series. Finally in Section 7, we will give the main theorem
on the twisted symmetric square $L$-function.

\quad\\

\begin{center}{\bf Notations}\end{center}

Throughout the paper, $F$ is a local field of characteristic zero or
a number field. If $F$ is a number field, we denote the ring of adeles
by $\A$. Both locally and globally, we denote by $\OF$ the ring of integers of
$F$. For each algebraic group $G$ over a global $F$, and $g\in G(\A)$,
by $g_v$ we mean the $v^{th}$ component of $g$, and so $g_v\in
G(F_v)$.

If $F$ is local, the symbol $(-,-)_F$ denotes the Hilbert
symbol of $F$. If $F$ is global, we let $(-,-)_\A:=\prod_v(-,-)_{F_v}$, where the
product is finite. We sometimes write simply $(-,-)$ for the Hilbert
symbol when there is no danger of confusion.

Throughout the paper we write
\[
r=\begin{cases}2q\\2q+1\end{cases}
\]
depending on the parity of $r$. For a partition $r_1+\cdots+r_k=r$ of $r$, we let
\[
M=\GL_{r_1}\times\cdots\times\GL_{r_k}\subseteq\GL_r
\]
and assume it is embedded diagonally as usual.
Let $P=P_1\times\cdots\times P_k$ be a parabolic
subgroup of $M$ where each $P_i$ is a parabolic subgroup of
$\GL_{r_i}$. Further assume that the Levi factor of $P_i$ is
$\GL_{l^i_1}\times\cdots\times\GL_{l^i_{m_i}}$, where
$l_1^i+\cdots+l_{m_i}^i=r_i$. Then we write
\[
P=P^{r_1,\dots,r_k}_{l^1_1,\dots,l^i_{m_1},\dots,l^k_1,\dots,l^k_{m_k}}
\]
namely the superscript indicates the ambient group $M$, and the
subscript indicates the Levi part. For example, $P^{2, r-3, 1}_{1,1,
  r-3, 1}$ means the parabolic subgroup of
$\GL_2\times\GL_{r-3}\times\GL_1$ whose Levi part is
$\GL_1\times\GL_1\times\GL_{r-3}\times\GL_1$. For the
minimal parabolic of $M$, we write $B^{r_1,\dots,r_k}$, namely
$B^{r_1,\dots,r_k}=P_{1,\dots,1}^{r_1,\dots,r_k}$. Also if $M=\GL_r$,
we usually omit the superscript and simply write $P_{l_1,\dots,l_m}$ for the
$(l_1,\dots,l_m)$-parabolic of $\GL_r$. In particular $B$ denotes the
Borel subgroup of $\GL_r$. For a parabolic subgroup $P$, we denote its
Levi part by $M_P$ and unipotent radical by $N_P$. We use the same
convention for the subscripts and superscripts for the unipotent
radical. For example, $N_{1,1,r-2}^{2,r-2}$ denotes the unipotent
radical of the parabolic $P_{1,1,r-2}^{2,r-2}$.

For any group $G$ and subgroup $H\subseteq G$, and for each $g\in G$,
we define $^gH=gHg^{-1}$. Then for a representation $\pi$ of $H$, we
define $^g\pi$ to be the representation of $^gH$ defined by
$^g\pi(h')=\pi(g^{-1}h'g)$ for $h'\in gHg^{-1}$.

For each $r$, we denote the $r\times r$ identity matrix by $I_r$. We
let $W$ be the set of all $r\times r$ permutation
matrices, so for each element $w\in W$ each row and each column has
exactly one 1 and  all the other entries are 0. The Weyl group of
$\GL_r$ is identified with $W$. Also for a Levi
$M=\GL_{r_1}\times\cdots\times\GL_{r_k}$, we let
$W_M$ be the subset of $W$ that only permutes the
$\GL_{r_i}$-blocks of $M$. Namely  $W_M$ is the collection of block
matrices 
\[
W_M:=\{(\delta_{\sigma(i),j}I_{r_j})\in W : \sigma\in S_k\},
\]
where $S_k$ is the permutation group of $k$ letters. Though $W_M$
is not a group in general, it is in bijection with $S_k$. Accordingly
we sometimes use the permutation notation for the Weyl group element. For example,
$(12\dots k)\in S_k$ corresponds to the longest element in $W_M$.

We use the usual notation for the roots of $\GL_r$. Namely $e_i$ is
the character on the maximum torus defined by $(t_1,\cdots,t_r)\mapsto
t_i$. Then each root is of the form $e_i-e_j$ and each positive root is
of the form $e_i-e_j$ with $i<j$. Let $P=MN$ be a parabolic subgroup
whose Levi is $M$. We let $\Phi_{P}(\C)$ be the
$\C$-vector space spanned by the roots of $M$. So in particular if
$M=\GL_r$, then $\Phi_P(\C)\cong\C^{r-1}$ and each $\nu\in \Phi_P(\C)$
is for the form $\nu=s_1e_1+\cdots+s_re_r$ with $s_1+\cdots+s_r=0$. We
let $\rho_P$ be half the sum of the positive roots of $M$. 

\quad

\begin{center}{\bf Acknowledgements}\end{center}

The author is partially supported by NSF grant DMS-1215419. He
would like to thank the referee for his/her careful reading of the
manuscript. 

\quad

%%%%%%%%%%%%%%%%%%%%%%%%%%%%%%%%%%%%%%%%%%%%%%%%%%%%%%%%%%%%%%%%%%%

\section{\bf The metaplectic double cover $\GLt_r$ of 
$\GL_r$}\label{S:metaplectic_cover}

%%%%%%%%%%%%%%%%%%%%%%%%%%%%%%%%%%%%%%%%%%%%%%%%%%%%%%%%%%%%%%%%%%%

In this section, we review the theory of the metaplectic double cover $\GLt_r$ of
$\GL_r$ for both local and global cases, which was originally
constructed by Kazhdan and Patterson in \cite{KP} and the metaplectic
tensor product for the Levi part developed by Mezo (\cite{Mezo}) and the
author (\cite{Takeda2}).

%%%%%%%%%%%%%%%%%%%%%%%%%%%%%%%%%%%%%%%%%%%%%%%%%%%%%%%%%%%%%%%%%%%

\subsection{\bf The local metaplectic double cover $\GLt_r(F)$}

%%%%%%%%%%%%%%%%%%%%%%%%%%%%%%%%%%%%%%%%%%%%%%%%%%%%%%%%%%%%%%%%%%%

Let $F$ be a (not necessarily non-archimedean) local field of
characteristic $0$. In this paper, by the metaplectic double cover
$\GLt_r(F)$ of $\GL_r(F)$, we mean the central extension of $\GL_r(F)$ by
$\{\pm1\}$ as constructed in \cite{KP} by Kazhdan and
Patterson. (Kazhdan and Patterson considered more general $n^\Th$ covers
$\GLt_r^{(c)}(F)$ with a twist by $c\in\{0,\dots,n-1\}$. But we only consider
the non-twisted double cover, \ie $n=2$ and $c=0$.) Later, Banks, Levy, and Sepanski
(\cite{BLS}) gave an explicit description of a 2-cocycle 
\[
\sigma_r:\GL_r(F)\times\GL_r(F)\rightarrow\{\pm 1\}
\]
which defines $\GLt_r(F)$ and shows that their 2-cocycle is
``block-compatible'', by which we mean the
following property of $\sigma_r$:  For the standard
$(r_1,\dots,r_k)$-parabolic $P$ of $\GL_r$, so that its Levi $M_P$ is
of the form
$\GL_{r_1}\times\cdots\times\GL_{r_k}$ which is embedded diagonally into
$\GL_r$, we have
\begin{equation}\label{E:compatibility}
\sigma_r(\begin{pmatrix}g_1&&\\ &\ddots&\\ &&g_k\end{pmatrix},
\begin{pmatrix}g'_1&&\\ &\ddots&\\ &&g'_k\end{pmatrix})
=\prod_{i=1}^k\sigma_{r_i}(g_i,g_i')\prod_{1\leq i<j\leq k}(\det(g_i), \det(g_j'))_F,
 \end{equation}
for all $g_i, g_i'\in\GL_{r_i}(F)$ (see \cite[Theorem 11, \S3]{BLS}), where
$(-,-)_F$ is the Hilbert symbol for $F$.
The 2-cocycle of \cite{BLS} generalizes the
well-known cocycle given by Kubota \cite{Kubota} for the case
$r=2$. Note that $\GLt_r(F)$ is not the $F$-rational points of an
algebraic group, but this notation seems to be standard. 

We define $\sigGLt_r(F)$ to be the group whose underlying set is
\[ 
\sigGLt_r(F) =\GL_r(F)\times\{\pm 1\}=\{(g,\xi):g\in\GL_r(F), \xi\in\{\pm
1\}\},
\] 
and the group law is defined by
\[
(g_1,\xi_1)\cdot (g_2,\xi_2)=(g_1g_2,\sigma_r(g_1,g_2)\xi_1\xi_2).
\]
Since we would like to emphasize the cocycle being used, we write
$\sigGLt_r(F)$ instead of $\GLt_r(F)$.

To use the block-compatible 2-cocycle of \cite{BLS} has obvious
advantages. In particular, it can be explicitly computed and, of course, it is
block-compatible. However it does not allow us to construct the global
metaplectic cover $\GLt_r(\A)$. Namely one
cannot define an adelic block-combatible 2-cocycle simply by taking
the product of the local block-combatible 2-cocycles over all the
places. This can be already observed for the case $r=2$. (See
\cite[p.125]{F}.)

For this reason, we will use a different 2-cocycle $\tau_r$ which
works nicely with the global metaplectic cover $\GLt_r(\A)$. To construct such
$\tau_r$, first assume $F$ is non-archimedean. It is known that an
open compact subgroup $K$
splits in $\GLt_r(F)$, and moreover if the residue characteristic of $F$ is odd,
$K=\GL_r(\OF)$. (See \cite[Proposition 0.1.2]{KP}.) Also for
$k,k'\in K$, we have $(\det(k),\det(k'))_F=1$.
Hence one has a
continuous map $s_r:\GL_r(F)\rightarrow\{\pm1\}$  such that
$\sigma_r(g,g')s_r(g)s_r(g')=s_r(gg')$ for all $g, g'\in
K$. Then define our 2-cocycle $\tau_r$ by
\begin{equation}\label{E:tau_sigma}
\tau_r(g, g'):=\sigma_r(g, g')s_r(g)s_r(g')/s_r(gg')
\end{equation}
for $g, g'\in\GL_r(F)$. If $F$ is archimedean, we set
$\tau_r=\sigma_r$. 

The choice of $s_r$ and hence $\tau_r$ is not unique. But there is a canonical choice with
respect to the splitting of $K$ in the sense explained in
\cite{Takeda2}. With this choice of $s_r$, the section
$K\rightarrow\sigGLt_r(F)$ defined by $k\mapsto (k,s_r(k))$ is what is
called the canonical lift in \cite{KP} which is denoted by $\kappa^\ast$
there. Also if $r=2$, our choice of $\tau_2$ is equal to the cocycle
denoted by $\beta$ in \cite{F}, which can be shown to be block
compatible. Indeed, the restriction of $\tau_2$ to $B^2\times B^2$
where $B^2$ is the Borel subgroup of $\GL_2$ coincides with $\sigma_2$.

Using $\tau_r$, we realize $\GLt_r(F)$ to be
\[ 
\GLt_r(F)=\GL_r(F)\times\{\pm 1\},
\] 
as a set and the group law is given by
\[ 
(g,\xi)\cdot(g',\xi')=(gg', \tau_r(g,g')\xi\xi').
\] 
Note that we have the exact sequence
\[
\xymatrix{
0\ar[r]&\{\pm1\}\ar[r]&\GLt_r(F)\ar[r]^{p}&\GL_r(F)\ar[r]&
0
}
\]
given by the obvious maps, where we call $p$ the canonical projection.

We define a set theoretic section
\[ 
\kappa:\GL_r(F)\rightarrow\GLt_r(F),\; g\mapsto (g,1).
\] 
Note that $\kappa$ is not a homomorphism. But by our
construction of the
cocycle $\tau_r$, $\kappa|_K$ is a homomorphism if $F$ is
non-archimedean and $K$ is a sufficiently small open compact subgroup,
and moreover if the residue characteristic is odd, one has
$K=\GL_r(\OF)$.

Also we define another set theoretic section
\[
\s:\GL_r(F)\rightarrow\GLt_r(F),\; g\mapsto (g,s_r(g)^{-1})
\]
where $s_r(g)$ is as above. We have the isomorphism
\[
\GLt_r(F)\rightarrow\sigGLt_r(F),\quad (g,\xi)\mapsto (g,s_r(g)\xi), 
\]
which gives rise to the commutative diagram
\[
\xymatrix{
\GLt_r(F)\ar[rr]&&\sigGLt_r(F)\\
&\GL_r(F)\ar[ul]^{\s}\ar[ur]_{g\mapsto (g,1)}&
}
\]
of set theoretic maps, \ie maps which are not necessarily homomorphisms. Also note
that the elements in the image $\s(\GL_r(F))$ ``multiply via
$\sigma_r$'' in the sense that for $g,g'\in\GL_r(F)$, we have
\begin{equation}\label{E:convenient}
(g,s_r(g)^{-1}) (g',s_r(g')^{-1})=(gg', \sigma_r(g,g')s_r(gg')^{-1}).
\end{equation}

For a subgroup
$H\subseteq\GL_r(F)$, whenever the cocycle $\sigma_r$ is trivial on
$H\times H$, the section $\s$ splits 
$H$ by (\ref{E:convenient}). We often denote the image $\s(H)$ by
$H^\ast$ or sometimes simply
by $H$ when it is clear from the context. Particularly important is
that by \cite[Theorem 7 (f), \S3]{BLS}, $\s$ splits $N_B$, the unipotent radical of the Borel
subgroup $B$ of $\GL_r(F)$, and
accordingly we denote $\s(N_B)$ by $N_B^\ast$. 

Assume $F$ is non-archimedean of odd residue characteristic. By
\cite[Proposition 0.I.3]{KP} we have
\begin{equation}\label{E:kappa_and_s}
\kappa|_{T\cap K}=\s|_{T\cap K},\quad \kappa|_{W}=\s|_{W},\quad
\kappa|_{N_B\cap K}=\s|_{N_B\cap K},
\end{equation}
where $W$ is the Weyl group and $K=\GL_r(\OF)$. 
In particular, this implies $s_r|_{T\cap K}=s_r|_{W}=s_r|_{N_B\cap
  K}=1$. Also note that $s_r(1)=1$. 

Now assume $F$ is any local field $F$. For each element $w\in W$, we
denote $\s(w)$ by $w$, which is equal to $(w, 1)$ if the residue
characteristic of $F$ is odd or $F=\C$, when it
 is clear from the context. However it is important to note that
 $\s$ does not split $W$ if the residue characteristic of $F$ is
 even or $F=\R$. Indeed, $\s$ splits $W$ if and only if $(-1,-1)_F=1$.

Note that $\GLt_1=\GL_1(F)\times\{\pm 1\}$, where the product is the
direct product, \ie $\sigma_1$ is trivial. (See \cite[Corollary
8, \S3]{BLS}.) Also we define $\widetilde{F^\times}$ to be
$\widetilde{F^\times}=F^\times\times\{\pm 1\}$ as a set but the
product is given by $(a,\xi)\cdot(a',\xi')=(aa',
(a,a')_F\xi\xi')$. (It is known that $\widetilde{F^\times}$ is
isomorphic to $\GLt_1$ if and only if $(-1,-1)_F=1$. It is our
understanding that this is due to J. Klose (see \cite[p.42]{KP}), though we
do not know where his proof is written. See
\cite{Adams} for a proof for a more general statement.)

For each subgroup $H(F)\subseteq\GL_r(F)$, we denote the preimage
$p^{-1}(H(F))$ of $H(F)$ via the canonical projection $p$ by
$\widetilde{H}(F)$  or sometimes simply by
$\widetilde{H}$ when the base field is clear from the context. We call
it the ``metaplectic preimage'' of $H(F)$. 

If $P$ is a parabolic subgroup of $\GL_r$ whose Levi is
$M_P=\GL_{r_1}\times\cdots\times\GL_{r_k}$, we often write
\[ 
\MPt=\GLt_{r_1}\timest\cdots\timest\GLt_{r_k}
\] 
for the metaplectic preimage of $M_P$. One can check
\[
\Pt=\MPt N_P^\ast
\]
and $N_P^\ast$ is normalized by $\MPt$. Hence if $\pi$
is a representation of $\MPt$, one can consider the parabolically induced
representation $\Ind_{\MPt N_P^\ast}^{\GLt_r}\pi$ as usual by letting
$N_P^\ast$ act trivially. 

Next let
\[ 
\GL_r^{(2)}=\{g\in\GL_r:\det g\in  F^{\times 2}\}, 
\] 
and $\GLtt_r$ its metaplectic preimage. Also we define
\[ 
M_P^{(2)}=\{\begin{pmatrix}g_1&&\\ &\ddots&\\ &&g_k\end{pmatrix}
\in M_P: \det g_i\in  F^{\times 2}\}
\] 
and often denote its preimage by
\[ 
\MPtt=\GLtt_{r_1}\timest\cdots\timest\GLtt_{r_k}.
\] 
We write $P^{(2)}=M_P^{(2)}N_P$ and denote its preimage by
$\Ptt$. Then we have
\[
\Ptt=\MPtt N_P^\ast.
\]

\quad

Let us mention the following important fact. Let $Z_{\GL_r}\subseteq\GL_r$ be
the center of $\GL_r$. Then the preimage $\widetilde{Z_{\GL_r}}$,
though abelian, is not the center of
$\GLt_r$ in general. It is the center only when $r=2q+1$ or
$F=\C$. If $r=2q$, the center $Z_{\GLt_r}$ is 
\[
    Z_{\GLt_r}=\{(aI_r, \xi):a\in  F^{\times 2}, \xi\in\{\pm1\}\}.
\]
From
(\ref{E:compatibility}), one can compute
\[
\sigma_r(aI_r, a'I_r)=\prod_{1\leq i<j\leq r}(a,a')_F=(a,a')_F^{\frac{1}{2}r(r-1)}.
\]
Hence for either $r=2q$
or $r=2q+1$, $\widetilde{Z_{\GL_r}}$ is isomorphic to
$\widetilde{F^\times}$ if $q$ is odd, and isomorphic to $\GLt_1$ if
$q$ is even. Also note that for $r=2q$ we have $\widetilde{Z_{\GL_r}}\subset\GLtt_r$ and
it is the center of $\GLtt_r$.

\quad

Let $\pi$ be an admissible representation of a subgroup
$\widetilde{H}\subseteq \GLt_r$. We say $\pi$ is ``genuine'' if each
element $(1,\xi)\in\widetilde{H}$ acts as multiplication by $\xi$. On
the other hand, if $\pi$ is a
representation of $H$, one can always view it as a (non-genuine)
representation of $\widetilde{H}$ by pulling back $\pi$ via the
canonical projection $\widetilde{H}\rightarrow H$, which we denote by the
same symbol $\pi$. In particular, for a parabolic subgroup $P$, we
view the modular character $\delta_P$ as a character on $\Pt$ in this way.

%%%%%%%%%%%%%%%%%%%%%%%%%%%%%%%%%%%%%%%%%%%%%%%%%%%%%%%%%%%%%%%%%%%

\subsection{\bf The global metaplectic double cover $\GLt_r(\A)$}\label{S:group}

%%%%%%%%%%%%%%%%%%%%%%%%%%%%%%%%%%%%%%%%%%%%%%%%%%%%%%%%%%%%%%%%%%%

In this subsection we consider the global metaplectic group. So we let
$F$ be a number field and $\A$ the ring of adeles. We shall define the
2-fold metaplectic cover $\GLt_r(\A)$ of $\GL_r(\A)$. (Just like the
local case, we write $\GLt_r(\A)$ even though it is not the adelic
points of an algebraic group.) The construction of $\GLt_r(\A)$ has
been done in various places such as \cite{KP, FK}. 

First define the adelic 2-cocycle $\tau_r$ by
\[
    \tau_r(g,g'):=\prod_v\tau_{r,v}({g}_v,g'_v),
\]
for $g, g'\in\GL_r(\A)$, where $\tau_{r,v}$ is the local
cocycle defined in the previous subsection. By definition of $\tau_{r,v}$, we
have $\tau_{r,v}({g}_v,g'_v)=1$ for almost all $v$, and hence the
product is well-defined. 

We define $\GLt_r(\A)$ to be the group whose underlying set is
$\GL_r(\A)\times\{\pm 1\}$ and the group structure is defined as in the
local case, \ie
\[
    (g,\xi)\cdot(g',\xi')=(gg', \tau_r(g,g')\xi\xi'),
\]
for $g, g'\in\GL_r(\A)$, and $\xi,\xi'\in\{\pm 1\}$.  Just as the local case, we have
\[
\xymatrix{
0\ar[r]&\{\pm1\}\ar[r]&\GLt_r(\A)\ar[r]^{p}&\GL_r(\A)\ar[r]&0,
}
\]
where we call $p$ the canonical projection. Define a set theoretic section
$\kappa:\GL_r(\A)\rightarrow\GLt_r(\A)$ by
$g\mapsto(g,1)$. 

It is well-known that $\GL_r(F)$ splits in $\GLt_r(\A)$. However the
splitting is not via $\kappa$ but via the product of all the local
sections $\s_{r,v}$. Namely one can define the map
\[
\s:\GL_r(F)\rightarrow\GLt_r(\A),\quad g\mapsto (g, s_r(g)^{-1}),
\]
where
\[
s_r(g):=\prod_vs_{r,v}(g)
\]
makes sense for all $g\in\GL_r(F)$ and the splitting is implied by the
``product formula'' for the block-compatible 2-cocycle.

Unfortunately, however, the expression $\prod_vs_{r,v}(g_v)$ does
not make sense for every $g=\prod_vg_v\in\GL_r(\A)$ because one does not know
whether $s_{r,v}(g_v)=1$ for almsot all $v$. But whenever the product
$\prod_vs_{r,v}(g_v)$ makes sense we denote the element $(g,
\prod_vs_{r,v}(g_v)^{-1})$ by $\s(g)$. This defines a partial global section
$\s:\GL_r(\A)\rightarrow\GLt_r(\A)$. 
It is shown in \cite{Takeda2} that the section $\s$ is defined
and splits the groups $\GL_r(F)$ and $N_B(\A)$. Also
$\s$ is defined, though not a homomorphism, on $B(\A)$ thanks to
(\ref{E:kappa_and_s}). And the following will be used later
\begin{Lem}\label{L:section_s}
For $g\in\GL_r(F)$ and $n,n'\in N_B(\A)$, the section $\s$ is
defined on $ngn'$, and moreover we have
\[
\s(ngn')=\s(n)\s(g)\s(n').
\]
\end{Lem}
\begin{proof}
Let us first note that in \cite[Lemma
1.9]{Takeda1}, it is shown that both $\s(ng)$ and $\s(gn')$ are
defined and moreover $\s(ng)=\s(n)\s(g)$ and
$\s(gn')=\s(g)\s(n')$. Namely the lemma holds for $n=1$ or
$n'=1$. Hence it suffices to show 
$\s(ngn')$ is defined and $\s(ngn')=\s(n)\s(gn')$. But if $s_r(ngn')$ is defined,
\[
s_r(ngn')=\sigma_r(n,gn')s_r(n)s_r(gn')/\tau_r(n,gn'),
\]
Note that here all of $s_r(n), s_r(gn')$ and
$\tau_r(n,gn')$ are defined. Moreover, locally
$\sigma_r(n_v,gn_v)=1$ for all $v$ by \cite[Theorem 7,
p.153]{BLS}. Hence $s_r(ngn')$ is defined. Thus $\s(ngn')$ is defined.
Moreover, since $\sigma_r(n,gn')=1$, we have $\s(ngn')=\s(n)\s(gn')$.
\end{proof}

Analogously to the local case, if the partial global section $\s$ is
defined on a subgroup $H\subseteq\GL_r(\A)$ and $\s|_H$ is a
homomorphism, we denote the image $\s(H)$ by $H^\ast$ or simply by $H$
when there is no danger of confusion.

We define the groups like $\GLtt_r(\A)$, $\MPtt(\A)$, $\Ptt(\A)$, etc
completely analogously to the local case. Also $\At^\times$ is a group whose
underlying set is $\A^\times\times\{\pm1\}$ and the group structure is given by
the global Hilbert symbol analogously to the local case. Also just like the local
case, the preimage $\widetilde{Z_{\GL_r}}(\A)$ of the center $Z(\A)$ is the center of
$\GLt_r(\A)$ only if $r=2q+1$. If $r=2q$, then the center of
$\GLt_r(\A)$ is the set of elements of the form $(aI_r,\xi)$ where
$a\in\A^{\times 2}$ and $\xi\in\{\pm 1\}$, and
$\widetilde{Z_{\GL_r}}(\A)$ is the center of only
$\GLtt_r(\A)$.

\quad

Let $\pi$ be a representation of $\widetilde{H}(\A)\subseteq \GLt_r(\A)$. Just like the local
case, we call $\pi$ genuine if $(1,\xi)\in\widetilde{H}(\A)$ acts as
multiplication by $\xi$. If $\pi$ is a genuine automorphic
representation of $\GLt_r(\A)$, then for each automorphic form
$f\in\pi$ we have $f(g,\xi)=\xi f(g,1)$ for all
$(g,\xi)\in\GLt_r(\A)$. Also any representation of
$H(\A)$ is viewed as a representation of $\widetilde{H}(\A)$ by
pulling it back by the canonical projection $p$, which we also denote by $\pi$. In
particular, this applies to the modular character $\delta_P$ for each
parabolic $P(\A)$.

\quad

We can also describe $\GLt_r(\A)$ as a quotient of a restricted direct
product of the groups $\GLt_r(F_v)$ as follows. Consider the
restricted direct product $\prod_v'\GLt_r(F_v)$ with respect to the
groups $\kappa(K_v)=\kappa(\GL_r(\mathcal{O}_{F_v}))$ for all $v$ with
$v\nmid 2$ and $v\nmid\infty$. If we denote each element in this
restricted direct product by $\Pi_v(g_v,\xi_v)$ so that $g_v\in
K_v$ and $\xi_v=1$ for almost all $v$, we have the
surjection
\begin{equation}\label{E:surjection}
    \rho:{\prod_v}'\GLt_r(F_v)\rightarrow\GLt_r(\A),\quad
    \Pi_v(g_v,\xi_v)\mapsto (\Pi_vg_v, \Pi_v\xi_v).
\end{equation}
This is a group homomorphism by our definition of $\GLt_r(F_v)$ and
$\GLt_r(\A)$. We have
\[
    {\prod_v}'\GLt_r(F_v)/\ker\rho\cong \GLt_r(\A),
\]
where $\ker\rho$ consists of the elements of the form
$\Pi_v(1,\xi_v)$ with $\xi_v=-1$ at an even number of $v$.

Suppose we are given a collection of irreducible admissible
representations $\pi_v$ of $\GLt_r(F_v)$ such that $\pi_v$ is
$\kappa(K_v)$-spherical for almost all $v$. Then we can form an
irreducible admissible representation of $\prod_v'\GLt_r(F_v)$ by
taking a restricted tensor product $\otimes_v'\pi_v$ as usual. Suppose
further that $\ker\rho$ acts trivially on $\otimes_v'\pi_v$, which is
always the case if each $\pi_v$ is genuine. Then it
descends to an irreducible admissible representation of $\GLt_r(\A)$,
which we denote by $\otimestt'_v\pi_v$, and call it the ``metaplectic
restricted tensor product''. Let us emphasize that the space for
$\otimestt'_v\pi_v$ is the same as that for
$\otimes_v'\pi_v$. Conversely, if $\pi$ is an irreducible admissible
representation of $\GLt_r(\A)$, it is written as
$\otimestt'_v\pi_v$ where $\pi_v$ is an irreducible admissible
representation of $\GLt_r(F_v)$, and for almost all $v$, $\pi_v$ is
$\kappa(K_v)$-spherical. (To see it, view $\pi$ as a representation of
the restricted product $\prod_v'\GLt_r(F_v)$ by pulling it back by
$\rho$ and apply the usual
tensor product theorem for the restricted product, which gives
$\otimes_v'\pi_v$, and it descends to $\otimestt_v'\pi_v$.) 

\quad

%%%%%%%%%%%%%%%%%%%%%%%%%%%%%%%%%%%%%%%%%%%%%%%%%%%%%%%%%%%%%%%%%%%

\subsection{The block-compatibility for $\GLt_r(\A)$}

%%%%%%%%%%%%%%%%%%%%%%%%%%%%%%%%%%%%%%%%%%%%%%%%%%%%%%%%%%%%%%%%%%%

We need to address an issue on the block-compatibility of
the global metaplectic double cover $\GLt_r(\A)$. As we already
mentioned, one cannot define $\GLt_r(\A)$ by using the
block-compatible local cocycles $\sigma_r$, but instead one needs to
introduce the cocycle $\tau_r$ which is not known to be block-compatible. To get
around it, one needs to introduce an intermediate cocycle $\tau_P$ for
each parabolic subgroup $P$.

Let $P(\A)\subseteq\GL_r(\A)$ be a parabolic subgroup whose Levi part
is $M_P(\A)=\GL_{r_1}(\A)\times\cdots\times\GL_{r_k}(\A)$. We define a
2-cocycle $\tau_P$ on $M_P(\A)$ by
\begin{equation}\label{E:compatibility_global}
\tau_P(\begin{pmatrix}g_1&&\\ &\ddots&\\ &&g_k\end{pmatrix},
\begin{pmatrix}g'_1&&\\ &\ddots&\\ &&g'_k\end{pmatrix})
=\prod_{i=1}^k\tau_{r_i}(g_i,g_i')\prod_{1\leq i<j\leq k}(\det(g_i), \det(g_j'))_{\A},
 \end{equation}
where $(-,-)_{\A}$ is the global Hilbert symbol. We define the group
$^c\MPt(\A)$ to be $M_P(\A)\times\{\pm 1\}$ as a set and the group law is given
by this cocycle $\tau_P$. Then it is shown in \cite{Takeda2} that 
\[
^c\MPt\cong \MPt.
\]
Namely the cocycle $\tau_P$ is cohomologous to $\tau_r|_{M_P(\A)\times
M_P(\A)}$.

\quad

%%%%%%%%%%%%%%%%%%%%%%%%%%%%%%%%%%%%%%%%%%%%%%%%%%%%%%%%%%%%%%%%%%%

\subsection{\bf The metaplectic tensor product}\label{S:tensor_product}

%%%%%%%%%%%%%%%%%%%%%%%%%%%%%%%%%%%%%%%%%%%%%%%%%%%%%%%%%%%%%%%%%%%

In this subsection, we assume $F$ is either a number field or a
local field. Let
$P\subseteq\GL_r$ be a parabolic subgroup whose Levi part is
$M_P=\GL_{r_1}\times\cdots\times\GL_{r_k}$. Given irreducible
admissible representations (or automorphic representations)
$\pi_1,\dots,\pi_k$ of $\GLt_{r_1},\dots,\GLt_{r_k}$, we would like to
construct a representation of $\Mt_P$ that can be called the
``metaplectic tensor product'' of $\pi_1,\dots,\pi_k$. However unlike the
non-metaplectic case, the construction is far from trivial, because
$\MPt$ is not the direct product
$\GLt_{r_1}\times\cdots\times\GLt_{r_k}$, and even worse there is no
natural map between them.  The construction of the metaplectic tensor
product for the local case was carried out by Mezo in \cite{Mezo} and the global case was
carried out by the author in \cite{Takeda2}. In what follows, we will
briefly recall this construction.

Assume $F$ is local. Let $\pin_i$ be an irreducible constituent of the
restriction $\pi_i|_{\GLtt_{r_i}(F)}$. Then the (usual) tensor product
$\pin_1\otimes\cdots\otimes\pin_k$, which is a representation of the
direct product $\GLtt_{r_1}(F)\times\cdots\times\GLtt_{r_k}(F)$,
descends to an irreducible admissible representation
$\pin_1\otimest\cdots\otimest\pin_k$ of
$\MPtt(F)=\GLtt_{r_1}(F)\timest\cdots\timest\GLtt_{r_k}(F)$. Let
$\omega$ be a character on the center $Z_{\GLt_r}(F)$ such that
$\omega$ agrees with $\pin_1\otimest\cdots\otimest\pin_k$ on the
overlap $Z_{\GLt_r}(F)\cap\Mtt_P(F)$, so
that we can extend $\pin_1\otimest\cdots\otimest\pin_k$ to a
representation of $Z_{\GLt_r}(F)\MPtt(F)$ by letting
$Z_{\GLt_r}(F)$ act by $\omega$, which we denote by
$\omega(\pin_1\otimest\cdots\otimest\pin_k)$. Now extend it to a
representation of some subgroup $\Ht(F)\subseteq\MPt(F)$, so that the
induced representation
$\Ind_{\Ht(F)}^{\Mt_P(F)}\omega(\pin_1\otimest\cdots\otimest\pin_k)$
is irreducible. Then Mezo has shown that this induced representation
is independent of all the choices made except the character
$\omega$. We denote this induced representation by
\[
\pi_\omega:=(\pi_1\otimest\cdots\otimest\pi_k)_\omega,
\]
and call it the metaplectic tensor product of $\pi_1,\dots,\pi_k$ with
respect to the character $\omega$. Moreover one can show that the induced
representation
$\Ind_{Z_{\GLt_r}(F)\MPtt(F)}^{\Mt_P(F)}\omega(\pin_1\otimest\cdots\otimest\pin_k)$
not only contains $\pi_\omega$ but any of its constituent is
(isomorphic to) $\pi_\omega$. Let us mention that if $r$ is even we have
the situation $Z_{\GLt_r}(F)\subseteq\MPtt(F)$, in which case
  there is no choice for $\omega$ and the metaplectic tensor product
  is canonical and we sometimes write simply
  $\pi_1\otimest\cdots\otimest\pi_k$.

Next assume $F$ is global, and assume all the $\pi_i$ are irreducible unitary
automorphic representations of $\GLt_{r_i}(\A)$. Let $\pin_i$ be an
irreducible constituent of the representation of $\GLtt_{r_i}(\A)$
obtained by restricting the automorphic forms in $\pi_i$ to
$\GLtt_{r_i}(\A)$. One can construct an ``automorphic
representation'' $\pin_1\otimest\cdots\otimest\pin_k$ analogously to
the local case.  Let $\omega$ be a Hecke character of $Z_{\GLt_r}(\A)$  such that
$\omega$ agrees with $\pin_1\otimest\cdots\otimest\pin_k$ on the
overlap $Z_{\GLt_r}(\A)\cap\Mtt_P(\A)$. Then essentially in the
analogous way to the local case, one can construct an automorphic
representation $\pi_\omega$ of $\Mt_P(\A)$, which is independent of
all the choices made except $\omega$, such that
\[
\pi_\omega=\otimestt_v'\pi_{\omega_v},
\]
\ie it is the restricted metaplectic tensor product of the local
metaplectic tensor products $\pi_{\omega_v}$. Just like the local
case we write
$\pi_\omega=(\pi_1\otimest\cdots\otimest\pi_k)_\omega$. 

In \cite{Takeda2} it is shown that the metaplectic tensor product
behaves just like the usual tensor product for the non-metaplectic
case. First of all, the cuspidality and square-integrability are preserved.

\begin{Prop}\label{P:tensor_cuspidal}
Assume $F$ is global. If each $\pi_i$ is square-integrable modulo center
(resp. cuspidal), then the tensor product
$(\pi_1\otimest\cdots\otimest\pi_k)_\omega$ is square-integrable
modulo center.
(resp. cuspidal).
\end{Prop}

The metaplectic tensor product behaves as expected under the action of
the Weyl group element. Namely,
\begin{Prop}\label{P:tensor_Weyl}
Let $w\in W_M$ be a Weyl group
element of $\GL_r$ that only permutes the $\GL_{r_i}$-factors of
$M$. Namely for each
$(g_1,\dots,g_k)\in\GL_{r_1}\times\cdots\times\GL_{r_k}$, we have $w
(g_1,\dots,g_k)w^{-1}=(g_{\sigma(1)},\dots,g_{\sigma(k)})$ for a
permutation $\sigma\in S_k$ of $k$ letters. Then both locally and
globally, we have
\[
^{w}(\pi_1\otimest\cdots\otimest\pi_k)_\omega
\cong(\pi_{\sigma(1)}\otimest\cdots\otimest\pi_{\sigma(k)})_\omega,
\]
where the left hand side is the twist of
$(\pi_1\otimest\cdots\otimest\pi_k)_\omega$ by $w$.
\end{Prop}

The metaplectic tensor product is compatible with parabolic induction.
\begin{Prop}\label{P:tensor_parabolic}
Both locally and globally, let $P=MN\subseteq\GL_r$ be the standard
parabolic subgroup whose Levi
part is $M=\GL_{r_1}\times\cdots\times\GL_{r_k}$. Further for each
$i=1,\dots,k$ let $P_i=M_iN_i\subseteq\GL_{r_i}$ be the standard
parabolic of $\GL_{r_i}$ whose Levi part is
$M_i=\GL_{r^i_1}\times\cdots\times\GL_{r^i_{l_i}}$. For each $i$, we are
given a representation 
\[
\sigma_i:=(\tau_{i,1}\,\otimest\cdots\otimest\,\tau_{i,l_i})_{\omega_i}
\]
of $\Mt_i$, which is given as the metaplectic tensor product of the
representations $\tau_{i,1},\dots,\tau_{i,l_i}$ of
$\GLt_{r^i_1},\dots,\GLt_{r^i_{l_i}}$, respectively. Assume that $\pi_i$ is an
irreducible constituent of the induced representation
$\Ind_{\Pt_i}^{\GLt_{r_i}}\sigma_i$. Then the metaplectic tensor
product
\[
\pi_\omega:=(\pi_1\,\otimest\cdots\otimest\,\pi_k)_\omega
\]
is an irreducible constituent of the induced representation
\[
\Ind_{\Qt}^{\Mt}(\tau_{1,1}\,\otimest\cdots\otimest\,\tau_{1, l_1}\,\otimest\cdots\otimest
\,\tau_{k,1}\,\otimest\cdots\otimest\,\tau_{k,l_k})_\omega,
\]
where $Q$ is the standard parabolic subgroup of $M$ whose Levi part is
$M_1\times\cdots\times M_k$. (Here ``irreducible constituent'' can be
replaced by ``irreducible quotient'' or ``irreducible
subrepresentation'', and the analogous proposition still holds.)
\end{Prop}

The global metaplectic tensor product behaves nicely with
restriction to a smaller Levi in the following sense.
\begin{Prop}\label{P:tensor_restriction}
Assume $F$ is global. 
\begin{enumerate}[(a)]
\item Let
\[
M_2=\GL_{r_2}\times\cdots\times\GL_{r_k}\subseteq 
M=\GL_{r_1}\times\GL_{r_2}\times\cdots\times\GL_{r_k}
\]
be the natural embedding in the lower right corner. Then there exists a realization of the
metaplectic tensor product
$\pi_\omega=(\pi_1\otimest\cdots\otimest\pi_k)_{\omega}$ such that for
each $f\in \pi$ and the restriction $f|_{\Mt_2(\A)}$ we have
\[
f|_{\Mt_2(\A)}\in \bigoplus_{\delta\in \GL_{r_1}(F)}
m_\delta(\pi_2\otimest\cdots\otimest\pi_k)_{\omega_\delta},
\]
where
$(\pi_2\otimest\cdots\otimest\pi_k)_{\omega_\delta}$
is the metaplectic tensor product of $\pi_2,\dots,\pi_k$,
$\omega_\delta$ is a certain character twisted by 
$\delta\in \GL_{r_1}(F)$ and $m_\delta\in\Z^{\geq 0}$ is a multiplicity. 

\item Let
\[
M'_2=\GL_{r_1}\times\cdots\times\GL_{r_{k-1}}\subseteq 
M=\GL_{r_1}\times\cdots\times\GL_{r_{k-1}}\times\GL_{r_k}
\]
be the natural embedding in the upper left corner. Then there exists a
realization (possibly different from the above) of the
metaplectic tensor product
$\pi_\omega=(\pi_1\otimest\cdots\otimest\pi_k)_{\omega}$ such that for
each $f\in \pi$ and the restriction $f|_{\Mt'_2(\A)}$ we have
\[
f|_{\Mt'_2(\A)}\in \bigoplus_{\delta'\in \GL_{r_k}(F)}
m_{\delta'}(\pi_1\otimest\cdots\otimest\pi_{k-1})_{\omega_{\delta'}},
\]
where
$(\pi_1\otimest\cdots\otimest\pi_{k-1})_{\omega_{\delta'}}$
is the metaplectic tensor product of $\pi_1,\dots,\pi_{k-1}$,
$\omega_{\delta'}$ is a certain character twisted by 
$\delta'\in\GL_{r_k}(F)$ and $m_{\delta'}\in\Z^{\geq 0}$ is a multiplicity. 
\end{enumerate}
\end{Prop}

Finally let us mention that the uniqueness of the metaplectic tensor product.
\begin{Prop}\label{P:tensor_unique}
Let $F$ be global (resp. local). Let $\pi_1,\dots,\pi_k$ and
$\pi'_1,\dots,\pi'_k$ be unitary automorphic
representations (resp. irreducible admissible representations) of
$\GLt_{r_1},\dots,\GLt_{r_k}$. They give rise to isomorphic
metaplectic tensor products with a Hecke character (resp. character) $\omega$, \ie
\[
(\pi_1\,\otimest\cdots\otimest\,\pi_k)_\omega\cong
(\pi'_1\,\otimest\cdots\otimest\,\pi'_k)_\omega,
\]
if and only if for each $i$ there exists a Hecke character (resp. character)
$\omega_i$ of $\GLt_{r_i}$ trivial on $\GLtt_{r_i}$ such that
$\pi_i\cong\omega_i\otimes\pi'_i$.
\end{Prop}

\quad

%%%%%%%%%%%%%%%%%%%%%%%%%%%%%%%%%%%%%%%%%%%%%%%%%%%%%%%%%%%%%%%%%%%

\section{\bf Exceptional representations of  $\GLt_r$}

%%%%%%%%%%%%%%%%%%%%%%%%%%%%%%%%%%%%%%%%%%%%%%%%%%%%%%%%%%%%%%%%%%%

In this section, we review the theory of the exceptional
representation of $\GLt_r$, a special case of which is the Weil
representation on $\GLt_2$. Throughout the section $\chi$ will denote
a unitary character on $F^\times$ when $F$ is local or a unitary Hecke
character on $\A^\times$ when it is global.

%%%%%%%%%%%%%%%%%%%%%%%%%%%%%%%%%%%%%%%%%%%%%%%%%%%%%%%%%%%%%%%%%%%

\subsection{\bf The Weil representation of $\GLt_2$}

%%%%%%%%%%%%%%%%%%%%%%%%%%%%%%%%%%%%%%%%%%%%%%%%%%%%%%%%%%%%%%%%%%%

First let us review the theory of the Weil representation of $\GLt_2$\\

\noindent\textbf{Local case:}

Let us consider the local case, and hence $F$ will be a (not
necessarily non-archimedean) local field of characteristic
$0$. Everything stated below without any
specific reference is found in \cite[\S 2]{GPS} for the
non-archimedean case and in \cite[\S 4]{G} for the archimedean case. Let
$S(F)$ be the space of Schwartz-Bruhat functions on $F$, \ie smooth
functions with compact support if $F$ is non-archimedean, and
functions with all the derivatives rapidly decreasing if $F$ is
archimedean. Let $\rr^\psi$ denote the representation of $\SLt_2(F)$
on $S(F)$ such that
\begin{align}
 \label{E:Weil_action1} \rr^\psi(\s\begin{pmatrix}0&1\\-1&0\end{pmatrix})f(x)&=\gamma(\psi)\hat{f}(x)\\
 \label{E:Weil_action2}\rr^\psi(\s\begin{pmatrix}1&b\\0&1\end{pmatrix})f(x)&
=\psi(bx^2)f(x),\qquad b\in F\\
 \label{E:Weil_action3}\rr^\psi(\s\begin{pmatrix}a&0\\0&a^{-1}\end{pmatrix})f(x)&
=|a|^{1/2}\mu_\psi(a)f(ax),\qquad a\in F^\times\\ 
 \label{E:Weil_action4} \rr^\psi(1,\xi)f(x)&=\xi f(x),
\end{align} 
where $\hat{f}(x)=\int f(y)\psi(2xy)\,dy$ with the Haar
measure $dy$ normalized in such a way that
$\hat{\hat{f}}(x)=f(-x)$. Also $\gamma(\psi)$ is the Weil
index of $\psi$, and $\mu_\psi(a)=\gamma(\psi_a)/\gamma(\psi)$. (See
\cite[Appendix]{Rao} for the notion of Weil index.)
It is well-known that $\rr^\psi$ is reducible and
written as $\rr^\psi=\rr^\psi_+\oplus\rr^\psi_-$, where $\rr^\psi_+$
(resp. $\rr^\psi_-$) is an irreducible representation realized in the
subspace of even functions (resp. odd functions) in $S(F)$. 

If $\chi(-1)=1$
(resp. $\chi(-1)=-1$), one can extend $\rr^\psi_+$
(resp. $\rr^\psi_-$) to a representation $\rr^\psi_\chi$ of
$\GLtt_2(F)$ by letting
\begin{equation}\label{E:Weil_action5}
\rr^\psi_\chi(\s\begin{pmatrix}1&0\\0&a^2\end{pmatrix})f(x)
=\chi(a)|a|^{-1/2}f(a^{-1}x).
\end{equation}
This is indeed a well-defined irreducible
representation of $\GLtt_2(F)$ and call it the Weil representation of
$\GLtt_2(F)$ associated with $\chi$. We denote by $S_\chi(F)$ the subspace of $S(F)$
in which $\rr_\chi^\psi$ is realized, which is the
space of even functions if $\chi(-1)=1$ and odd functions if
$\chi(-1)=-1$. Note that
\begin{equation}\label{E:central_character}
\rr^\psi_\chi(\s(\begin{pmatrix}a&0\\0&a\end{pmatrix})f(x)
=\chi(a)\mu_\psi(a)f(x).
\end{equation}

The Weil representation $\rr_\chi$ of $\GLt_2(F)$ is defined by
\[ 
\rr_\chi:=\Ind_{\GLtt_2(F)}^{\GLt_2(F)}\rr_\chi^\psi.
\] 
Then $\rr_\chi$ is irreducible and
independent of the choice of $\psi$, and hence our notation. 
If $\chi(-1)=1$, one can check that $\rr_\chi$ is the exceptional
representation of Kazhdan-Patterson for $r=2$ with the determinantal
character $\chi^{1/2}$, which will be recalled later.
If $\chi(-1)=-1$, then $\rr_\chi$ is described as follows: For
non-archimedean $F$, it is
supercuspidal (\cite[Proposition 3.3.3]{GPS}), for $F=\R$, it is a discrete
series representation of lowest weight $3/2$ (\cite[\S 6]{GPS}), and finally for $F=\C$, it is
identified with a certain induced representation (\cite[\S 6]{GPS}).\\

\noindent\textbf{Global case:}

We define the global Weil representation
$\rr_\chi$ of $\GLt_2(\A)$ as the restricted tensor product of the
local Weil representations, \ie
\[ 
\rr_\chi=\otimestt_v'\rr_{\chi_v}.
\] 
It is shown in \cite[\S8]{GPS} that $\rr_\chi$ is a square
integrable automorphic representation of $\GLt_2(\A)$, and moreover it
is cuspidal if and only if $\chi^{1/2}$ does not exist. Also one can
see that if $\chi^{1/2}$ exists, then just like the local case,
$\rr_\chi$ is the exceptional representation of Kazhdan-Patterson for
$r=2$, which will be explained later.\\

%%%%%%%%%%%%%%%%%%%%%%%%%%%%%%%%%%%%%%%%%%%%%%%%%%%%%%%%%%%%%%%%%%%

\subsection{\bf The Weil representation of $\MPt$}

%%%%%%%%%%%%%%%%%%%%%%%%%%%%%%%%%%%%%%%%%%%%%%%%%%%%%%%%%%%%%%%%%%%

Let us assume $r=2q$ and $P$ is the
$(2,\dots,2)$-parabolic $P_{2,\dots,2}$, so that
\[
M_P=\underbrace{\GL_2\times\cdots\times\GL_2}_{q \text{ times}}.
\]
Recall from Section \ref{S:metaplectic_cover} that  we write
\[
\MPt=\GLt_2\timest\cdots\timest\GLt_2.
\]
Since each element in the center $Z_{\GLt_{2q}}$ is of the form
$(a^2I_{2q},\xi)$, we have $Z_{\GLt_{2q}}\subseteq\MPtt$. Hence the
metaplectic tensor product of this Levi is unique. (In other words,
there is only once choice for $\omega$.)

We extend the theory of the Weil
representation both locally and globally as discussed in the previous subsection to the group
$\MPt$ by taking the metaplectic tensor product of $q$ copies of the Weil representation of
$\GLt_2$, and write
\begin{equation}\label{E:Weil_M_P}
\Pi_\chi:=(\rr_\chi\otimest\cdots \otimest\rr_\chi)_\omega,
\end{equation}
where $\omega$ is the unique choice for the character $Z_{\GLt_r}$
which is actually given by $(a^2I_{2q},\xi)\mapsto
\xi\chi(a^2)^q$. Also it should be mentioned that locally we have
\[
\Pi_\chi=\Ind_{\MPtt}^{\MPt}\rr_\chi^{\psi}\otimest\cdots
\otimest\rr_\chi^{\psi}.
\]
We call $\Pi_\chi$ the Weil representation of $\MPt$.\\

%%%%%%%%%%%%%%%%%%%%%%%%%%%%%%%%%%%%%%%%%%%%%%%%%%%%%%%%%

\subsection{\bf Non-twisted exceptional representation}

%%%%%%%%%%%%%%%%%%%%%%%%%%%%%%%%%%%%%%%%%%%%%%%%%%%%%%%%%

Let us now consider the non-twisted exceptional
representation of $\GLt_r$ developed by Kazhdan and Patterson in
\cite{KP}. We treat both $r=2q$ and
$2q+1$ at the same time. Also most of the time, we consider the
local and global cases at the same time, and all
the groups are over the local field $F$ (non-archimedean or
archimedean) or the adeles $\A$. 

For our character $\chi$, we let
\[
\Omega_\chi=(\underbrace{\chit\otimest\cdots\otimest\chit}_{r \text{ times}})_\omega,
\]
which is a representation of the metaplectic preimage $\Tt$ of maximal torus $T$. 
Note that $\Omega_\chi$ depends on $\omega$ if $r=2q+1$, but we
suppress it from our notation. For each $\nu\in\Phi_B(\C)$, let us define
\[
\Omega_\chi^\nu:=\Omega_\chi\otimes\exp(\nu, H_B(-))
\]
where $H_B(-)$ is the Harish-Chandra homomorphism as usual. Note that
$\exp(2\rho_B, H_B(-))=\delta_B$. Then it is shown in \cite{KP} that
the induced representation $\Ind_{\Bt}^{\GLt_r}\Omega_\chi^\nu$ has its
greatest singularity at $\nu=\rho_B/2$, and the quotient of
$\Ind_{\Bt}^{\GLt_r}\Omega_\chi^{\rho_B/2}=\Ind_{\Bt}^{\GLt_r}\Omega_\chi\otimes\delta_B^{1/4}$
is called the exceptional representation. Namely, we have

\begin{Prop} 
The induced representation $\Ind_{\Tt
  N_B^\ast}^{\GLt_r}\Omega_\chi\otimes\delta_B^{1/4}$ has a unique
irreducible quotient, which we denote by $\theta_\chi$. For the local
case, it is the image of the intertwining integral
\[ 
\Ind_{\Tt N_B^\ast}^{\GLt_r}\Omega_\chi\otimes\delta_B^{1/4}
\rightarrow \Ind_{\Tt
N_B^\ast}^{\GLt_r}\;^{w_0}(\Omega_\chi\otimes\delta_B^{1/4}),
\] 
where $w_0$ is the longest Weyl group element. For the global case,
it is generated by the residues of the Eisenstein series at
$\nu=\rho_B/2$ for the induced space
$\Ind_{\Bt}^{\GLt_r}\Omega_\chi^\nu$, and $\theta_\chi$ is a
square integrable automorphic
representation of $\GLt_r(\A)$. Moreover for the global $\theta_\chi$,
one has the decomposition
$\theta_\chi=\otimestt'_v\theta_{\chi_v}$. 
\end{Prop}
\begin{proof}
See \cite[Theorem I.2.9]{KP} for the local statement and \cite[Theorem
II.2.1]{KP} for the global one.
\end{proof}

We call the representation $\theta_\chi$ \emph{the non-twisted exceptional
representation} of $\GLt_r$ with the \emph{determinantal character}
$\chi$. It should be mentioned that if $r=2$,
$\theta_\chi$ is isomorphic to the Weil representation
$\rr_{\chi^2}$. Note that just as $\Omega_\chi$, $\theta_\chi$ depends on
$\omega$, but we suppress it from our notation.\\

Let us mention that a small discrepancy between the exceptional
representation defined above and the one in \cite{Takeda1} which is
defined as follows.
First for the maximal torus $T\subseteq B$, we let 
\begin{equation}\label{E:T^e}
T^{\e}=\{\begin{pmatrix}t_1&&\\&\ddots&\\&&t_r\end{pmatrix}\in T:
      t_1t_2^{-1},
      t_3t_4^{-1},\dots,t_{2q-1}t_{2q}^{-1}\text{ are squares}
\}.
\end{equation}
The metaplectic preimage $\Tte$ of $T^{\e}$ is a maximal
abelian subgroup of $\Tt$. Then in \cite{Takeda1} the non-twisted exceptional
representation of $\GLt_r$ was defined to be the unique
irreducible quotient of the induced representation $\Ind_{\Tte
N_B^\ast}^{\GLt_r}\omega_\chi^\psi\otimes\delta_B^{1/4}$, where
$\omega_\chi^\psi$ is the character on $\Tte$ defined by
\begin{equation}\label{E:exceptional_character}
    {\omega_{\chi}^{\psi^a}}((1,\xi)\s(t))=\xi\chi(\det
    t)\mu_{\psi}(t_2)\mu_\psi(t_4)\mu_\psi(t_6)\cdots \mu_\psi(t_{2q}),
\end{equation}
where $\mu_\psi$ is the ratio of the Weil indices.
(Note that even when
$F$ is global, the section $\s$ is defined on $T_\A$ and the
expression $\s(t)$ makes sense.) However the exceptional
representation defined this way coincides with the above $\theta_\chi$
with a certain choice of $\omega$. To see this, let us first assume
that $F$ is local, and define
\[
\Omega_\chi^{\psi}:=\Ind_{\Tte}^{\Tt}\omega_\chi^{\psi}.
\]
This is irreducible (\cite[p.55]{KP}). Indeed
\[
\Omega_\chi^{\psi}=(\chit\otimest\cdots\otimest\chit)_\omega,
\]
where each $\chit$
is the non-genuine character on $\GLt_1$ defined by $(a,\xi)\mapsto
\xi\chi(a)$ and the character $\omega$ on the center $Z_{\GLt_r}$ is
given by 
\[
\omega(aI_r,\xi)=\xi\chi^r(a)\mu_\psi(a)^q.
\] 
By inducing in stages, one can see that
\[
    \Ind_{\Tte N_B^\ast}^{\GLt_r}{\omega_\chi^{\psi}}\otimes\delta_B^{1/4}
    =\Ind_{\Bt}^{\GLt_r}{\Omega_\chi^{\psi}}\otimes\delta_B^{1/4},
\]
which implies that the non-twisted exceptional representation in
\cite{Takeda1} is precisely our $\theta_\chi$ with the above chosen
$\omega$. Now if $F$ is global, we can define
$\Omega_\chi^{\psi}$ to be the global metaplectic tensor product
$(\chit\otimest\cdots\otimest\chit)_\omega$ with $\omega$ chosen in
the same way as the local case, and hence the global exceptional
representation $\theta_{\chi}$ is obtained as the quotient of the
global induced representation
$\Ind_{\Bt(\A)}^{\GLt_r(\A)}{\Omega_\chi^{\psi}}\otimes\delta_B^{1/4}$,
and we have
$\theta_\chi=\otimestt'_v\theta_{\chi_v}$, which again coincides with
the global non-twisted exceptional representation in \cite{Takeda1}.

\begin{Rmk}\label{R:dependence_on_psi}
It is important to note that the above discussion shows that in
\cite{Takeda1} only one particular central character $\omega$ was used, which depends on
the additive character $\psi$ chosen. (But it is shown in
\cite{Takeda1} that after all it depends on $\psi$ only when both $r$
and $q$ are odd.) In this paper, however, we always assume $\omega$ is
arbitrary. Indeed, it is crucial to do so when we compute the poles of
our Eisenstein series as we will see later. Nonetheless, it should be also
mentioned that to obtain the Rankin-Selberg integral of the
$L$-function, it is necessary to choose the particular $\omega$ as
above.
\end{Rmk}

%%%%%%%%%%%%%%%%%%%%%%%%%%%%%%%%%%%%%%%%%%%%%%%%%%%%%%%%%

\subsection{\bf Twisted exceptional representation}

%%%%%%%%%%%%%%%%%%%%%%%%%%%%%%%%%%%%%%%%%%%%%%%%%%%%%%%%%

Next we consider the twisted version of the exceptional
representation of $\GLt_r$ when $r=2q$. The local case was originally
constructed by the Ph.D thesis by Banks \cite{Banks} when the residue
characteristic is odd, and the other cases are taken care of in
\cite{Takeda1}. 

Let $P$ be the $(2,\dots,2)$-parabolic whose Levi $M_P$ is $\GL_2\times\cdots\times\GL_2$
($q$-times), and $\Pi_\chi$ the Weil representation of $\MPt$ as in
\eqref{E:Weil_M_P}. For each $\nu\in\Phi_P(\C)$, let us define
\[
\Pi_\chi^\nu:=\Pi_\chi\otimes\exp(\nu, H_P(-))
\]
where $H_P(-)$ is the Harish-Chandra homomorphism.  Analogously to the non-twisted
exceptional representation of \cite{KP}, 
the induced representation $\Ind_{\Pt}^{\GLt_r}\Pi_\chi^\nu$ has its
greatest singularity at $\nu=\rho_P/2$, and the quotient of
$\Ind_{\Pt}^{\GLt_r}\Pi_\chi^{\rho_P/2}=\Ind_{\Pt}^{\GLt_r}\Pi_\chi\otimes\delta_P^{1/4}$
is called the twisted exceptional representation. Namely, we have
\begin{Prop}
The induced representation $\Ind_{\Pt}^{\GLt_{2q}}\Pi_\chi\otimes\delta_P^{1/4}$ has a
unique irreducible quotient, which we denote by $\vartheta_\chi$. For
the local case, it is the image of the intertwining integral
\[
\Ind_{\Pt}^{\GLt_{2q}}\Pi_\chi\otimes\delta_P^{1/4}\rightarrow
\Ind_{\Pt}^{\GLt_{2q}}{^{w_0}(\Pi_\chi\otimes\delta_P^{1/4})},
\]
where $w_0$ is the longest Weyl group element relative to $P$. For the
global case, it is generated by the residues of the Eisenstein series
at $\nu=\rho_P/2$ for the induced space
$\Ind_{\Pt}^{\GLt_r}\Pi_\chi^\nu$, and $\vartheta_\chi$ is a square integrable
automorphic representation of $\GLt_{2q}(\A)$. Moreover for the global
$\vartheta_\chi$, one has the decomposition
$\vartheta_\chi=\otimestt_v'\vartheta_{\chi_v}$.
\end{Prop}
\begin{proof}
See \cite[Proposition 2.35]{Takeda1} for the local statement and
\cite[Theorem 2.33]{Takeda1} for the global statement.
\end{proof}

\quad

We call $\vartheta_\chi$ \emph{the twisted exceptional representation}
of $\GLt_{2q}$. Both locally and globally, if $\chi^{1/2}$ exists, one can show that
\[
\vartheta_\chi=\theta_{\chi^{1/2}}.
\]
This is because the Weil representation
$\rr_\chi$ is the non-twisted exceptional representation of $\GLt_2$
with the determinantal character $\chi^{1/2}$.

\begin{Rmk}\label{R:dependene_on_psi2}
Let us note that unlike the case $r=2q+1$, there is no choice for the
central character $\omega$ for constructing the metaplectic tensor
product $\Pi_\chi$ and hence $\vartheta_\chi$ depends only on
$\chi$. Accordingly there is no discrepancy between $\vartheta_\chi$
here and the one in \cite{Takeda1}. 
\end{Rmk}

\quad

%%%%%%%%%%%%%%%%%%%%%%%%%%%%%%%%%%%%%%%%%%%%%%%%%%%%%%%%%%%%%%%%%%%

\section{\bf Induced representations and intertwining
  operators}\label{S:normalized_intertwining}

%%%%%%%%%%%%%%%%%%%%%%%%%%%%%%%%%%%%%%%%%%%%%%%%%%%%%%%%%%%%%%%%%%%

Let 
\[
Q=P_{r-1,1}=(\GL_{r-1}\times\GL_1)N_Q
\]
be the standard $(r-1,1)$-parabolic of $\GL_r$, so the Levi
part is $\GL_{r-1}\times\GL_1$. The inducing data for the Eisenstein
series we consider in this paper is a residual representation on the
parabolic $\Qt$. In this section, we first define the inducing
representation, which we called the exceptional representation
of $\GLt_{r-1}\timest\GLt_1$ in \cite{Takeda1}. This representation is
the metaplectic tensor product of the exceptional representation
$\theta_\chi$ or $\vartheta_\chi$ of $\GLt_{r-1}$ and  a character on
$\GLt_1$. (The precise construction differs, depending on the parity
of $r$.) Then we will examine the
analytic behavior of the intertwining operators on this induced representation. The main object
of this section is to prove Theorem \ref{T:normalized_intertwining}.

\quad

%%%%%%%%%%%%%%%%%%%%%%%%%%%%%%%%%%%%%%%%%%%%%%%%%%%%%%%%%

\subsection{\bf The inducing representation for $r=2q$}

%%%%%%%%%%%%%%%%%%%%%%%%%%%%%%%%%%%%%%%%%%%%%%%%%%%%%%%%%

In this subsection we assume $r=2q$ and $F$ can be both local and global,
and for example the group $\GLt_r$ denotes
both $\GLt_r(F)$ ($F$ local) and $\GLt_r(\A)$ ($F$ global). 
Let $\theta_\chi$ be the
non-twisted exceptional representation of $\GLt_{r-1}$ with the
determinantal character $\chi$. For a character $\eta$ on $\GL_1$,
define $\etat:\GLt_1\rightarrow\{\pm 1\}$ to
be the character defined by $\etat(a,\xi)\mapsto\xi\eta(a)$ for
$(a,\xi)\in\GLt_1$. We let
\[
\theta_{\chi,\eta}:=(\theta_\chi\;\otimest\;\etat)_\omega
\]
\ie the metaplectic tensor product of $\theta_\chi$ and
$\etat$. Note that since $Z_{\GLt_{2q}}\subseteq\MQtt$, there is no
actual choice for the character $\omega$. 

It should be mentioned that even when $r=2q+1$, one can define
$\theta_\chi$, (which is equal to $\vartheta_{\chi^2}$) and hence can
define $\theta_{\chi,\eta}$, though most of the time we use the
representation $\theta_{\chi,\eta}$ for the case $r=2q$.

Let us mention that what we denoted by $\theta_{\chi,\eta}$ in
\cite{Takeda1} corresponds to what we mean by $\theta_{\chi,\chi\eta}$
in this paper. The reason is because at the time we wrote
\cite{Takeda1} we did not know how to formulate the global metaplectic
tensor product and as a result we constructed the representation
$\theta_{\chi,\eta}$ more directly as the unique irreducible quotient
of an induced representation. But now that we have developed in
\cite{Takeda2} the
theory of global metaplectic tesnor products, which includes the compatibility with
parabolic inductions (Proposition \ref{P:tensor_parabolic}), one can see that the
construction in \cite{Takeda1} is indeed the same as the one
above. Namely the representation $\theta_{\chi,\eta}$ is, locally or
globally, a unique irreducible quotient of
\[
\Ind_{\Bt^{r-1,1}}^{\MQt}(\chit\otimest\cdots\otimest\chit\otimest\etat)_\omega\otimes
\delta_{B^{r-1,1}}^{1/4},
\]
where $B^{r-1,1}$ is the Borel subgroup of $M_Q=\GL_{r-1}\times\GL_1$,
namely $B^{r-1,1}=M_Q\cap B$.

%%%%%%%%%%%%%%%%%%%%%%%%%%%%%%%%%%%%%%%%%%%%%%%%%%%%%%%%%

\subsection{\bf The inducing representation for $r=2q+1$}

%%%%%%%%%%%%%%%%%%%%%%%%%%%%%%%%%%%%%%%%%%%%%%%%%%%%%%%%%

Next we will consider the case $r=2q+1$. Also keep the notation for $F$
from the previous subsection, namely $F$ is either local or
global. Let $\etat$ be as before and
$\vartheta_\chi$ the twisted exceptional representation of
$\GLt_{2q}$, where we include the case $\chi^{1/2}$ exists. Then we define
\[
\vartheta_{\chi,\eta}:=(\vartheta_\chi\;\otimest\;\etat)_\omega.
\]
Note that if $\chi^{1/2}$
exists, we have $\vartheta_{\chi, \eta}=\theta_{\chi^{1/2},\eta}$.

\begin{Rmk}
Let us mention again that in \cite{Takeda1} a particular central
character $\omega$ is chosen. Indeed, we used
\[
\omega: (1,\xi)\s_Q(aI_r)\mapsto\xi\chi(a)^q\eta(a)\mu_\psi(a)^q,
\]
which depends on $\psi$ if (and only if) $q$ is odd. However in this paper, $\omega$ is
always arbitrary.
\end{Rmk}

Just like the case for $r=2q$, the compatibility with parabolic induction
for metaplectic tensor products (Proposition \ref{P:tensor_parabolic})
implies that $\vartheta_{\chi, \eta}$
is a unique irreducible quotient of 
\[
\Ind_{\Pt^{r-1,1}_{2,\dots,2,1}}^{\MQt}(\rr_\chi\otimest\cdots\otimest\rr_\chi\otimest\etat)_\omega
\otimes\delta_{P^{r-1,1}_{2,\dots,2,1}}^{1/4},
\]
where $P^{r-1,1}_{2,\dots,2,1}$ is the $(2,\dots,2,1)$-parabolic
subgroup of $M_Q$, so the Levi part is
$\GL_2\times\cdots\times\GL_2\times\GL_1$.

%%%%%%%%%%%%%%%%%%%%%%%%%%%%%%%%%%%%%%%%%%%%%%%%%%%%%%%%%%%%%%%%%%%

\subsection{\bf The intertwining operator and its analytic behavior}

%%%%%%%%%%%%%%%%%%%%%%%%%%%%%%%%%%%%%%%%%%%%%%%%%%%%%%%%%%%%%%%%%%%

Let $\theta=\theta_{\chi,\eta}$ or $\vartheta_{\chi,\eta}$ depending
on the parity of $r$ and assume $F$ is global. Define
\begin{equation}\label{E:w_1}
w_1=\begin{cases}
\begin{pmatrix}&&1\\&I_{r-2}&\\1&&\end{pmatrix},
&\text{if $r=2q$};\\
\begin{pmatrix}&&&&1\\&&&1&\\&&I_{r-4}&&\\&1&&&\\1&&&&\end{pmatrix},&\text{if $r=2q+1$}.
\end{cases}
\end{equation}
In the rest of the section, we will consider the analytic behavior of
the global intertwining operator
\[
A(s,\theta,w_1):\Ind_{\Qt(\A)}^{\GLt_r(\A)}\theta\otimes\delta_Q^s\rightarrow
\Ind_{^{w_1}(\Mt_Q(\A))N_{1,r-1}(\A)}^{\GLt_r(\A)}{^{w_1}\theta}\otimes\delta_Q^{-s},
\]
and will show
\begin{Thm}\label{T:normalized_intertwining}
Let us exclude the case that $r=2$  and $\chi^2\eta^{-2}=1$. Then
for $\Re(s)\geq 0$, the above intertwining operator $A(s,\theta,w_1)$ is
holomorphic except when the complete $L$-function
$L(r(2s+\frac{1}{2})-r+1,\chi^2\eta^{-2})$  (if
$r=2q$) or $L(r(2s+\frac{1}{2})-r+1,\chi\eta^{-2})$ (if
$r=2q+1$) has a pole; In other words,  if $r=2q$,
it has a possible pole if and only if $\chi^2\eta^{-2}=1$ and
$s\in\{\frac{1}{4}, \frac{1}{4}-\frac{1}{2r}\}$, and if $r=2q+1$,
it has a possible pole if and only if $\chi\eta^{-2}=1$ and
$s\in\{\frac{1}{4}, \frac{1}{4}-\frac{1}{2r}\}$.

Further if $f^s=\otimes'f^s_v$ is a factorizable section and $S$ is a finite
set of places which contains all the archimedean places and all the
non-archimedean places $v$ at which $f_v^s$ is not spherical. Then the
normalized intertwining operator
\[
A^\ast(s,\theta,w_1)f^s:=\begin{cases}
L^S(r(2s+\frac{1}{2}),\chi^2\eta^{-2})A(s,\theta, w_1)f^s,\quad\text{if
  $r=2q$}\\
L^S(r(2s+\frac{1}{2}),\chi\eta^{-2})A(s,\theta, w_1)f^s,\quad\text{if
  $r=2q+1$},
\end{cases}
\]
is holomorphic for all $s\in\C$ except when the complete
$L$-function $L(r(2s+\frac{1}{2})-r+1,\chi^2\eta^{-2})$
(resp. $L(r(2s+\frac{1}{2})-r+1,\chi\eta^{-2})$) has a pole.
\end{Thm}

The rest of the section is devoted to the proof of this theorem, which,
as we will see, boils down to determining the possible
poles of the local intertwining operator.

%%%%%%%%%%%%%%%%%%%%%%%%%%%%%%%%%%%%%%%%%%%%%%%%%%%%%%%%%

\subsection{\bf Unramified place}

%%%%%%%%%%%%%%%%%%%%%%%%%%%%%%%%%%%%%%%%%%%%%%%%%%%%%%%%%

To prove the above theorem, we first need the following result on the
unramified place.
\begin{Lem}\label{L:spherical_section}
Let $r=2q$ or $2q+1$. Also assume $F$ is a non-archimedean local
field of odd residue characteristic. Further assume that $\chi$,
$\eta$ and $\omega$ are all unramified. Consider the intertwining operators
\begin{align*}
    &A(s,\theta_{\chi,\eta},w_1):
    \Ind_{\Qt}^{\GLt_{2q}}\theta_{\chi,\eta}\otimes\delta_Q^s\rightarrow 
\Ind_{^{w_1}\MQt N_{1,r-1}^\ast}^{\GLt_{2q}}\;^{w_1}(\theta_{\chi,\eta})\otimes\delta_Q^{-s},
\quad (r=2q)\\
&A(s,\vartheta_{\chi,\eta},w_1):
    \Ind_{\Qt}^{\GLt_{2q+1}}{\vartheta_{\chi,\eta}}\otimes\delta_Q^s\rightarrow 
\Ind_{^{w_1}\MQt N_{1,r-1}^\ast}^{\GLt_{2q+1}}\;^{w_1}({\vartheta_{\chi,\eta}})\otimes\delta_Q^{-s},
\quad (r=2q+1),
\end{align*}
where $w_1$ is as in (\ref{E:w_1}).

If $f_0^s\in\Ind_{\Qt}^{\GLt_{2q}}{\theta_{\chi,\eta}\otimes\delta_Q^s}$ (or
$\Ind_{\Qt}^{\GLt_{2q+1}}{\vartheta_{\chi,\eta}}\otimes\delta_Q^s $) is the
spherical section such that $f_0^s(1)=1$, then
\begin{align}
 \label{E:normalizing_factor_even}   A(s,\theta_{\chi, \eta},w_1)f_0^s(1)&=\frac{L(r(2s+\frac{1}{2})-r+1,
      \chi^2\eta^{-2})}{L(r(2s+\frac{1}{2}),
      \chi^2\eta^{-2})},\quad(r=2q);\\
\label{E:normalizing_factor_odd}    A(s,\vartheta_{\chi, \eta},w_1)f_0^s(1)&=\frac{L(r(2s+\frac{1}{2})-r+1,
      \chi\eta^{-2})}{L(r(2s+\frac{1}{2}),
      \chi\eta^{-2})},\quad(r=2q+1).
\end{align}
\end{Lem}
\begin{proof}
This is \cite[Lemma 2.58]{Takeda1}. Note that in \cite{Takeda1} we used
$w_0=\left(\begin{smallmatrix}&1\\ I_{r-1}&\end{smallmatrix}\right)$
instead of the $w_1$ of the lemma, but one can verify that the results
are the same because we have $w_1=\left(\begin{smallmatrix}&1\\
    I_{r-1}&\end{smallmatrix}\right)
\left(\begin{smallmatrix}w'&\\&1\end{smallmatrix}\right)$, where 
\[
w'=\begin{cases}
\begin{pmatrix}&I_{r-2}\\1&\end{pmatrix},
&\text{if $r=2q$};\\
\begin{pmatrix}&&&1\\&&I_{r-4}&\\ &1&&\\
    1&&&\end{pmatrix},&\text{if $r=2q+1$},
\end{cases}
\]
and $\left(\begin{smallmatrix}w'&\\&1\end{smallmatrix}\right)\in
M_Q(\OF)$. Also note that in \cite{Takeda1} a specific $\omega$ was
used but the proof there applies to any $\omega$.
\end{proof}

\begin{Rmk}
Note that for the case $r=2q+1$, if $\chi$ is unramified, $\chi^{1/2}$
exists, and hence one has $\vartheta_{\chi, \eta}=\theta_{\chi^{1/2},
  \eta}$. Then one can
see that the formula for this case is actually subsumed under the formula for
$A(s,\theta_{\chi^{1/2}, \eta},w_1)f_0^s(1)$ as in the $r=2q$ case. 
\end{Rmk}

%%%%%%%%%%%%%%%%%%%%%%%%%%%%%%%%%%%%%%%%%%%%%%%%%%%%%%%%%

\subsection{\bf Proof of Theorem \ref{T:normalized_intertwining} ($r=2q$)} 

%%%%%%%%%%%%%%%%%%%%%%%%%%%%%%%%%%%%%%%%%%%%%%%%%%%%%%%%%

Let us consider the case $r=2q$, so $\theta=\theta_{\chi,\eta}$. This case is
essentially the case treated by \cite{BG}. However, as we pointed out in
\cite{Takeda1}, the argument in \cite{BG} does not seem to work when
they use an asymptotic formula on matrix coefficients at the
archimedean place, and hence we will give an alternate argument, which
follows the idea given by Jiang \cite[84-86]{Jiang} though we use many
of the ideas from \cite{BG}.

First note that for a factorizable
$f^s=\otimes'f_v^s\in \Ind_{\Qt(\A)}^{\GLt_r(\A)}\theta\otimes\delta_Q^s$,
one can, by Lemma \ref{L:spherical_section}, write
\[
A(s,\theta,w_1)f^s
=\frac{L(r(2s+\frac{1}{2})-r+1,\chi^2\eta^{-2})}{L(r(2s+\frac{1}{2}),\chi^2\eta^{-2})}
\left(\underset{v}{\otimes}'
\frac {L_v(r(2s+\frac{1}{2}),\chi_v^2\eta_v^{-2})}{L_v(r(2s+\frac{1}{2})-r+1,\chi_v^2\eta_v^{-2})}
A_v(s,\theta_v,w_1)f_v^s\right),
\]
which gives
\begin{align}\label{E:normalized_intertwining}
&A^\ast(s,\theta,w_1)f^s\\
\notag=&L^S(r(2s+\frac{1}{2}),\chi^2\eta^{-2})A(s,\theta,w_1)f^s\\
\notag=&L(r(2s+\frac{1}{2})-r+1,\chi^2\eta^{-2})
\Bigg(\underset{v\in S}{\otimes'}
\frac {1}{L_v(r(2s+\frac{1}{2})-r+1,\chi_v^2\eta_v^{-2})}
A_v(s,\theta_v,w_1)f_v^s\Bigg)\\
\notag&\quad\Bigg(\underset{v\notin S}{\otimes'}
\frac {L_v(r(2s+\frac{1}{2}),\chi_v^2\eta_v^{-2})}{L_v(r(2s+\frac{1}{2})-r+1,\chi_v^2\eta_v^{-2})}
A_v(s,\theta_v,w_1)f_v^s\Bigg),
\end{align}
where $S$, which depends on $f^s$, is as in Theorem \ref{T:normalized_intertwining}.
By Lemma \ref{L:spherical_section} and our choice of $S$, the product
\[
\underset{v\notin S}{\otimes'}
\frac {L_v(r(2s+\frac{1}{2}),\chi_v^2\eta_v^{-2})}{L_v(r(2s+\frac{1}{2})-r+1,\chi_v^2\eta_v^{-2})}
A_v(s,\theta_v,w_1)f_v^s
\]
is holomorphic. Also for $\Re(s)\geq 0$, the normalizing factor
$L^S(r(2s+\frac{1}{2}),\chi^2\eta^{-2})$ is non-zero holomorphic, and
hence in this region the poles of $A(s,\theta,w_1)f^s$ coincide with
those of $A^\ast(s,\theta_v,w_1)f^s$. 

Hence to prove Theorem
\ref{T:normalized_intertwining}, it suffices to show that the local
``modified intertwining operator''
\begin{align*}
\frac {1}{L_v(r(2s+\frac{1}{2})-r+1,\chi_v^2\eta_v^{-2})}
A_v(s,\theta_v,w_1):&\Ind_{\Qt(F_v)}^{\GLt_r(F_v)}\theta_v\otimes\delta_Q^s\\
&\quad\rightarrow
\Ind_{^{w_1}(\Mt_Q(F_v))N_{1,r-1}(F_v)}^{\GLt_r(F_v)}{^{w_1}\theta_v}\otimes\delta_Q^{-s}
\end{align*}
is holomorphic for all $s\in\C$. Thus the question is now completely
local, and hence in what follows, we will omit the subscript $v$ and
assume that everything is over the local field.

Recall that the representation $\theta_{\chi, \eta}$ is the
metaplectic tensor product
$\theta_{\chi,\eta}=(\theta_\chi\otimest\etat)_\omega$ for an
appropriate $\omega$, and further
recall that the representation $\theta_\chi$ is the exceptional
representation with the determinantal character $\chi$ which is an
irreducible subrepresentation of the induced representation
$\ind_{\Bt^{r-1}}^{\GLt_{r-1}}(\chit\otimest\cdots\otimest\chit)_\omega\otimes\delta_{B^{r-1}}^{1/4}$
(unnormalized induction) for an appropriate $\omega$, where $\Bt^{r-1}$
is the Borel subgroup of $\GLt_{r-1}$. Hence by Proposition \ref{P:tensor_parabolic}, we
have 
\[
\theta_{\chi,\eta}=(\theta_\chi\otimest\etat)_\omega\subseteq
\ind_{\Bt^{r-1,1}}^{\GLt_{r-1}\timest\GLt_1}(\chit\otimest\cdots\otimest\chit\otimest\etat)_\omega
\otimes\delta_{B^{r-1,1}}^{1/4}
\]
for an appropriate $\omega$, where $B^{r-1,1}$ is the Borel subgroup
of $\GL_{r-1}\times\GL_1$. By inducing in stages we have
\[
\ind_{\Qt}^{\GLt_r}\theta_{\chi,\eta}\otimes\delta_Q^{s+\frac{1}{2}}\subseteq \ind_{\Bt}^{\GLt_r}
(\chit\otimest\cdots\otimest\chit\otimest\etat)_\omega
\otimes\delta_{B^{r-1,1}}^{1/4}\delta_Q^{s+\frac{1}{2}}.
\]
By using the normalized induction, we have
\[
\Ind_{\Qt}^{\GLt_r}\theta_{\chi,\eta}\otimes\delta_Q^{s}\subseteq \Ind_{\Bt}^{\GLt_r}
(\chit\otimest\cdots\otimest\chit\otimest\etat)_\omega
\otimes\delta_{B^{r-1,1}}^{-1/4}\delta_Q^{s}.
\]

Furthermore the metaplectic tensor product
$(\chit\otimest\cdots\otimest\chit\otimest\etat)_\omega$ is a
representation of the Heisenberg group $\Tt$, and hence it is induced
from a representation of the maximal abelian subgroup
$\Tt^{\e}$, where $T^{\e}$ is as in (\ref{E:T^e}). Indeed, we have
\[
(\chit\otimest\cdots\otimest\chit\otimest\etat)_\omega
=\Ind_{\Tt^{\e}}^{\Tt}\omega_{\chi,\eta}
\]
for a character $\omega_{\chi,\eta}:\Tt^{\e}\rightarrow\C^1$ with the
property that the restriction $\omega_{\chi,\eta}|_{\Ttt}$ to $\Ttt$
is $\chit\otimest\cdots\otimest\chit\otimest\etat$, namely
\[
\omega_{\chi,\eta}(\s(\begin{pmatrix}t_1^2&&\\ &\ddots&\\
  &&t_r^2\end{pmatrix}))=\chi(t_2^1)\cdots\chi(t_{r-1}^2)\eta(t_r^2).
\]
(One can write down $\omega_{\chi,\eta}$ more explicitly but we will
not need it for our purposes.) Therefore we have
\begin{equation}\label{E:inclusion}
\Ind_{\Qt}^{\GLt_r}\theta_{\chi,\eta}\otimes\delta_Q^{s}\subseteq
\Ind_{\Tt^{\e}N_B^\ast}^{\GLt_r}
\omega_{\chi,\eta}\otimes\delta_{B^{r-1,1}}^{-1/4}\delta_Q^{s},
\end{equation}
and accordingly we can view each section $f^s\in
\Ind_{\Qt}^{\GLt_r}\theta_{\chi,\eta}\otimes\delta_Q^{s}$ as an
element in the induced representation $\Ind_{\Tt^{\e}N_B^\ast}^{\GLt_r}
\omega_{\chi,\eta}\otimes\delta_{B^{r-1,1}}^{-1/4}\delta_Q^{s}$.

Now we would like to study the analytic property of the integral
\[
A(s,\theta,w_1)f^s(g)
=\int_{N_{1,r-1}}f^s(\s(w_1n)g)\,dn.
\]
For this purpose, let
\begin{equation}\label{E:w_0}
w_0=J_{r}=\begin{pmatrix}&&1\\ &J_{r-2}&\\ 1&&\end{pmatrix}
=\begin{pmatrix}&&1\\ &\iddots&\\1&&\end{pmatrix}
\end{equation}
be the longest element in the Weyl group. We use the following by-now
well-known lemma, which seems to be sometimes known as the Rallis
lemma.
\begin{Lem}\label{L:Rallis}
The highest pole of the intertwining operator $A(s,\theta,w_1)$ is
achieved by $A(s,\theta,w_1)f^s(\s(w_0))$ as the sections $f^s$ run
through those sections with $\supp(f^s)\subseteq \Qt w_0\Qt$.
\end{Lem}
\begin{proof}
Several versions of this lemma can be found in various places such as 
\cite[Lemma 4.1]{PSR} and \cite[Lemma 4.1]{Sh92}, and our case is the
metaplectic analogue of \cite[Lemma 2.1.1]{Jiang}.
\end{proof}

We should also mention
\begin{Lem}\label{L:Rallis2}
Let $f^s$ be as in the above lemma, so that $\supp(f^s)\subseteq \Qt
w_0\Qt$. Let $N\subseteq\Qt$ be a subset of $\Qt$ such that $w_0
Nw_0^{-1}\cap\Qt=\{1\}$. Then for each fixed $q\in\Qt$, the map on $N$
defined by $n\mapsto f^s(qw_0n)$ is compactly supported.
\end{Lem}
\begin{proof}
Note that since $f^s$ is in the induced space, it is compactly
supported modulo $\Qt$. Now since $w_0
Nw_0^{-1}\cap\Qt=\{1\}$, the natural map $N\rightarrow
\Qt\backslash\Qt w_0\Qt$ given by $n\mapsto \Qt w_0n$ is 1-1. Hence
the lemma follows.
\end{proof}

Now by Lemma \ref{L:Rallis}, we have only to show
\begin{equation}\label{E:normalized_intertwining_w_0}
\frac {1}{L(r(2s+\frac{1}{2})-r+1,\chi^2\eta^{-2})}A(s,\theta,w_1)f^s(\s(w_0))
\end{equation}
is holomorphic with $f^s$ as in the lemma, where
\[
A(s,\theta,w_1)f^s(\s(w_0))= \int_{N_{1,r-1}}f^s(\s(w_1n)\s(w_0))\,dn.
\]
Let us write each $n\in N_{1,r-1}$ as
\[
n=\begin{pmatrix}1&Z&y\\ &I_{r-2}&\\ &&1\end{pmatrix}.
\]
By direct computation one can verify
\begin{align*}
w_1nw_0&=\begin{pmatrix}1&&\\ & J_{r-2}&\\ y&zJ_{r-2}&1\end{pmatrix}\\
&=\begin{pmatrix}-y^{-1}&-Zy^{-1}&1\\ &I_{r-2}&\\ && y\end{pmatrix}w_0
\begin{pmatrix}1&ZJ_{r-2}y^{-1}&y^{-1}\\ &I_{r-2}&\\ &&1\end{pmatrix}
\end{align*}
provided $y\neq 0$. Hence
\[
\s(w_1n)\s(w_0)=(1,\epsilon)
\s(\begin{pmatrix}-y^{-1}&-Zy^{-1}&1\\ &I_{r-2}&\\ && y\end{pmatrix})
\s(w_0\begin{pmatrix}1&ZJ_{r-2}y^{-1}&y^{-1}\\ &I_{r-2}&\\ &&1\end{pmatrix})
\]
for some $\epsilon=\epsilon(y, Z)\in\{\pm 1\}$, which {\it a priori}
depends on $y$ and $Z$. (One can compute $\epsilon$ by using
the algorithm for computing the cocycle $\sigma_r$ developed in 
\cite{BLS}, and can actually verify that $\epsilon=1$ for any $y$ and
$Z$. But since this computation is extremely tedious, though not so deep,
and for our purposes we will not need the precise information on
$\epsilon$, we will leave $\epsilon$ as above.)

With this computation for $\s(w_1n)\s(w_0)$ one can write
\begin{align*}
&\int_{N_{1,r-1}}f^s(\s(w_1n)\s(w_0))\,dn\\
=&\int_{F^\times}\int_{F^{r-2}}\epsilon f^s\Big(
\s(\begin{pmatrix}-y^{-1}&-Zy^{-1}&1\\ &I_{r-2}&\\ && y\end{pmatrix})
\s(w_0\begin{pmatrix}1&ZJ_{r-2}y^{-1}&y^{-1}\\ &I_{r-2}&\\
  &&1\end{pmatrix})
\Big)\,dZ dy\\
=&\int_{F^\times}\int_{F^{r-2}}\epsilon|y^{-1}|^{\frac{1}{4}(r-2)+s+\frac{1}{2}}|y|^{(1-r)(s+\frac{1}{2})}
(\chit\otimest\cdots\otimest\chit\otimest\etat)_{\omega}
\s(\begin{pmatrix}-y^{-1}&&\\ &I_{r-2}&\\ && y\end{pmatrix})\\
&\qquad\qquad f^s\Big(
\s(w_0\begin{pmatrix}1&ZJ_{r-2}y^{-1}&y^{-1}\\ &I_{r-2}&\\
  &&1\end{pmatrix})
\Big)\,dZ dy\\
=&\int_{F^\times}\int_{F^{r-2}}\epsilon|y|^{-rs-\frac{3}{4}r+\frac{1}{2}}
(\chit\otimest\cdots\otimest\chit\otimest\etat)_{\omega}
\s(\begin{pmatrix}-y^{-1}&&\\ &I_{r-2}&\\ && y\end{pmatrix})\\
&\qquad\qquad f^s\Big(
\s(w_0\begin{pmatrix}1&ZJ_{r-2}y^{-1}&y^{-1}\\ &I_{r-2}&\\
  &&1\end{pmatrix})
\Big)\,dZ dy,
\end{align*}
where we should note that $dy$ is the additive measure and the integral over $F^\times$ is
the same as the integral over $F$ because those two sets are equal
almost everywhere. By changing the variable $ZJ_{r-2}y^{-1}\mapsto Z$,
then changing the additive measure $dy$ to the multiplicative measure
$d^\times y$, and then changing the variable $y^{-1}\mapsto y$, one
can see the above integral is written as
\[
\int_{F^\times}\int_{F^{r-2}}\epsilon'|y|^{rs-\frac{1}{4}r+\frac{1}{2}}
(\chit\otimest\cdots\otimest\chi\otimest\etat)_{\omega}
(\s(\begin{pmatrix}-y&&\\ &I_{r-2}&\\ && y^{-1}\end{pmatrix}))
f^s\Big(\s(w_0\begin{pmatrix}1&Z&y\\ &I_{r-2}&\\
  &&1\end{pmatrix})\Big)\,dZ d^\times y
\]
for some $\epsilon'=\epsilon'(y,Z)\in\{\pm 1\}$. 

Now let $a_1,\dots, a_l$ be a complete set of representatives of
$F^{\times 2}\backslash F^{\times }$. Then the above integral is written as
\begin{align*}
&\sum_{i=1}^l\int_{F^{\times 2}}\int_{F^{r-2}}\epsilon_i'|xa_i|^{rs-\frac{1}{4}r+\frac{1}{2}}
(\chit\otimest\cdots\otimest\chit\otimest\etat)_{\omega}
(\s(\begin{pmatrix}-xa_i&&\\ &I_{r-2}&\\ && (xa_i)^{-1}\end{pmatrix}))\\
&\qquad\qquad\qquad f^s\Big(\s(w_0\begin{pmatrix}1&Z&xa_i\\ &I_{r-2}&\\
  &&1\end{pmatrix})\Big)\,dZ d^\times x
\end{align*}
for a certain choice of the measure $d^\times x$ on $F^{\times 2}$ and
some $\epsilon_i'=\epsilon_i'(x,Z)\in\{\pm 1\}$,
which is further written as
\begin{align*}
&\sum_{i=1}^l\int_{F^{\times 2}}\int_{F^{r-2}}\epsilon_i'|xa_i|^{rs-\frac{1}{4}r+\frac{1}{2}}
(\chit\otimest\cdots\otimest\chit\otimest\etat)_{\omega}
(\s(\begin{pmatrix}x&&\\ &I_{r-2}&\\ && x^{-1}\end{pmatrix}))\\
&\qquad\qquad\qquad 
f^s\Big(\s(\begin{pmatrix}-a_i&&\\ &I_{r-2}&\\ && a_i^{-1}\end{pmatrix})
\s(w_0\begin{pmatrix}1&Z&xa_i\\ &I_{r-2}&\\
  &&1\end{pmatrix})\Big)\,dZ d^\times x.
\end{align*}
Recall from (\ref{E:inclusion}) that we can view the section $f^s$ as an
element in the induced space $\Ind_{\Tt^{\e}N_B^\ast}^{\GLt_r}
\omega_{\chi,\eta}\otimes\delta_{B^{r-1,1}}^{-1/4}\delta_Q^{s}$, and
hence the above integral is written as
\begin{align}
\label{E:integral1}
&\sum_{i=1}^l\int_{F^{\times 2}}\int_{F^{r-2}}\epsilon_i'|xa_i|^{rs-\frac{1}{4}r+\frac{1}{2}}
\chi(x)\eta(x)^{-1}\\
\notag&\qquad\qquad\qquad  
f^s\Big(\s(\begin{pmatrix}-a_i&&\\ &I_{r-2}&\\ && a_i^{-1}\end{pmatrix})
\s(w_0\begin{pmatrix}1&Z&xa_i\\ &I_{r-2}&\\
  &&1\end{pmatrix})\Big)\,dZ d^\times x.
\end{align}

Note that
\[
w_0\begin{pmatrix}1&Z&xa_i\\ &I_{r-2}&\\
  &&1\end{pmatrix}w_0^{-1}=\begin{pmatrix}1&&\\ &I_{r-2}&\\
 xa_i &Z&1\end{pmatrix}
\]
and hence we can apply Lemma \ref{L:Rallis2} to the map
\[
(x, Z)\mapsto f^s\Big(\s(\begin{pmatrix}-a_i&&\\ &I_{r-2}&\\ && a_i^{-1}\end{pmatrix})
\s(w_0\begin{pmatrix}1&Z&xa_i\\ &I_{r-2}&\\
  &&1\end{pmatrix})\Big),
\]
which implies that the one can write
\[
 \epsilon_i'f^s\Big(\s(\begin{pmatrix}-a_i&&\\ &I_{r-2}&\\ &&
  a_i^{-1}\end{pmatrix})\s(w_0\begin{pmatrix}1&Z&xa_i\\ &I_{r-2}&\\
  &&1\end{pmatrix})\Big)=\sum_{\lambda, \phi, \phi'}\lambda(s)\phi(x)\phi'(Z)
\]
for some holomorphic functions $\lambda$ and smooth compactly supported
functions $\phi$ and $\phi'$ on $F$ and $F^{r-2}$,
respectively. Hence to study the analytic behavior of
\eqref{E:integral1} we have only to study that of
\[
\int_{F^{\times 2}}\int_{F^{r-2}}|x|^{rs-\frac{1}{4}r+\frac{1}{2}}
\chi(x)\eta(x)^{-1}\phi(x)\phi'(Z)\,dZ d^\times x,
\]
which is written as
\[
\int_{F^{\times 2}}|x|^{rs-\frac{1}{4}r+\frac{1}{2}}
\chi(x)\eta(x)^{-1}\phi(x)\,d^\times x\cdot\int_{F^{r-2}}
\phi'(Z)\,dZ.
\]
The integral over $Z$ is independent of $s$, and hence we have only to
consider the first integral. But one can see
\[
\int_{F^{\times 2}}|x|^{rs-\frac{1}{4}r+\frac{1}{2}}
\chi(x)\eta(x)^{-1}\phi(x)\,d^\times x
=c\int_{F^{\times}}|y^2|^{rs-\frac{1}{4}r+\frac{1}{2}}
\chi(y^2)\eta(y^2)^{-1}\phi(y^2)\,d^\times y
\]
for an appropriate non-zero constant $c$. Then by Tate's thesis, one
knowns that this integral is $L(2(rs-\frac{1}{4}r+\frac{1}{2}),
\chi^2\eta^{-2})$ times an entire function on $s$, where this
$L$-factor is precisely the one appearing in
(\ref{E:normalized_intertwining_w_0}). Therefore
(\ref{E:normalized_intertwining_w_0}) is an entire function on $s$.

%%%%%%%%%%%%%%%%%%%%%%%%%%%%%%%%%%%%%%%%%%%%%%%%%%%%%%%%%

\subsection{\bf Proof of Theorem \ref{T:normalized_intertwining} ($r=2q+1$)}

%%%%%%%%%%%%%%%%%%%%%%%%%%%%%%%%%%%%%%%%%%%%%%%%%%%%%%%%%

Next we consider the case $r=2q$, so $\theta=\vartheta_{\chi,\eta}$.
The basic idea is the same as the case $\theta=\theta_{\chi,\eta}$, in which we reduce
the problem to the local one and use the Rallis lemma and Tate's
thesis. Namely by arguing as above, we can see that we have only to show that the local
``modified intertwining operator''
\begin{align*}
\frac {1}{L_v(r(2s+\frac{1}{2})-r+1,\chi_v\eta_v^{-2})}
A_v(s,\theta_v,w_1):&\Ind_{\Qt(F_v)}^{\GLt_r(F_v)}\theta_v\otimes\delta_Q^s\\
&\quad\rightarrow
\Ind_{^{w_1}(\Mt_Q(F_v))N_{1,r-1}(F_v)}^{\GLt_r(F_v)}{^{w_1}\theta_v}\otimes\delta_Q^{-s}
\end{align*}
is holomorphic for all $s\in\C$.  Again, we will omit the subscript $v$ and
assume that everything is over the local field. 

Recall that the representation $\vartheta_{\chi, \eta}$ is the
metaplectic tensor product
$\vartheta_{\chi,\eta}=(\vartheta_\chi\otimest\etat)_\omega$ for an
appropriate $\omega$, and further
recall that the representation $\vartheta_\chi$ is the twisted exceptional
representation on $\GLt_{2q}$, which is an
irreducible subrepresentation of the induced representation
$\ind_{\Pt_{2,\dots,2}^{r-1}}^{\GLt_{r-1}}(\rr_{\chi}\otimest\cdots\otimest\rr_{\chi})_\omega
\otimes\delta_{P_{2,\dots,2}^{r-1}}^{1/4}$
(unnormalized induction) for an appropriate $\omega$, where $P_{2,\dots,2}^{r-1}$
is the $(2,\dots,2)$-parabolic subgroup of $\GL_{r-1}$. Hence by
Proposition \ref{P:tensor_parabolic}, we have 
\[
\vartheta_{\chi,\eta}=(\vartheta_\chi\otimest\etat)_\omega\subseteq
\ind_{\Pt_{2,\dots,2,1}^{r-1,1}}^{\GLt_{r-1}\timest\GLt_1
}(\rr_{\chi}\otimest\cdots\otimest\rr_{\chi}\otimest\etat)_\omega
\otimes\delta_{P_{2,\dots,21}^{r-1,1}}^{1/4}
\]
for an appropriate $\omega$. By inducing in stages we have
\[
\ind_{\Qt}^{\GLt_r}\vartheta_{\chi,\eta}\otimes\delta_Q^{s+\frac{1}{2}}
\subseteq \ind_{\Pt_{2,\dots,2,1}^{r-1,1}}^{\GLt_r}
(\rr_{\chi}\otimest\cdots\otimest\rr_{\chit}\otimest\etat)_\omega
\otimes\delta_{P_{2,\dots,2,1}^{r-1,1}}^{1/4}\delta_Q^{s+\frac{1}{2}}.
\]
By using the normalized induction, we have
\[
\Ind_{\Qt}^{\GLt_r}\vartheta_{\chi,\eta}\otimes\delta_Q^{s}\subseteq 
\Ind_{\Pt_{2,\dots,2,1}^{r-1.1}}^{\GLt_r}
(\rr_{\chi}\otimest\cdots\otimest\rr_{\chi}\otimest\etat)_\omega
\otimes\delta_{P_{2,\dots,2,1}^{r-1,1}}^{-1/4}\delta_Q^{s}.
\]

Furthermore by definition of metaplectic tensor product, we have
\[
(\rr_{\chi}\otimest\cdots\otimest\rr_{\chi}\otimest\etat)_\omega
\subseteq
\Ind_{Z_{\GLt_r}\Mtt_{P_{2,\dots,2,1}}}^{\Mt_{P_{2,\dots,2,1}}}\;\omega\;
(\rr_\chi^\psi\otimest\cdots\otimest\rr_\chi^\psi\otimest\etat),
\]
where $P_{2,\dots,2,1}$ is the $(2,\dots,2,1)$-parabolic of
$\GL_r$. (One can check that this inclusion is actually equality, but since we
will not need this fact, we will leave the verification to the
reader.) Therefore we have
\begin{equation}\label{E:inclusion2}
\Ind_{\Qt}^{\GLt_r}\vartheta_{\chi,\eta}\otimes\delta_Q^{s}\subseteq
\Ind_{Z_{\GLt_r}\Mtt_{P_{2,\dots,2,1}}N_{2,\dots,2,1}^\ast}^{\GLt_r}\;\omega\;
(\rr_\chi^\psi\otimest\cdots\otimest\rr_\chi^\psi\otimest\etat)
\otimes\delta_{P^{r-1,1}}^{-1/4}\delta_Q^{s},
\end{equation}
and accordingly we can view each section $f^s$ as an
element in the latter induced representation. Also we should recall that the space of the
representation
$\rr_\chi^\psi\otimest\cdots\otimest\rr_\chi^\psi\otimest\etat$ is the
(usual) tensor product
\[
S_{\chi}(F)\otimes\cdots \otimes S_{\chi}(F)\otimes\C,
\]
where $S_{\chi}(F)$ realizes the Weil representation
$\rr_\chi^\psi$. Hence for each $g\in\GLt_r$, we can view $f^s(g)$ as
an element in this space.

Now we would like to study the analytic property of the integral
\[
A(s,\theta,w_1)f^s(g)
=\int_{N_{1,r-1}}f^s(\s(w_1n)g)\,dn.
\]
But by Lemma \ref{L:Rallis}, we have only to show
\begin{equation}\label{E:normalized_intertwining_w_02}
\frac {1}{L(r(2s+\frac{1}{2})-r+1,\chi\eta^{-2})}A(s,\theta,w_1)f^s(\s(w_0))
\end{equation}
is holomorphic with $f^s$ as in Lemma \ref{L:Rallis}. If we
write each $n\in N_{1,r-1}$ as
\[
n=\begin{pmatrix}1&Z&y\\ &I_{r-2}&\\ &&1\end{pmatrix},
\]
then, if $y\neq0$, we have
\begin{align*}
w_1nw_0&=\begin{pmatrix}1&&\\ & J'_{r-2}&\\ y&zJ_{r-2}&1\end{pmatrix}\\
&=\begin{pmatrix}-y^{-1}&-Zy^{-1}J_{r-2}J'_{r-2}&1\\ &I_{r-2}&\\ && y\end{pmatrix}w_0'
\begin{pmatrix}1&ZJ'_{r-2}y^{-1}&y^{-1}\\ &I_{r-2}&\\ &&1\end{pmatrix},
\end{align*}
where
\[
w_0'=\begin{pmatrix}&&1\\ &J'_{r-2}&\\1&&\end{pmatrix}
\quad\text{and}\quad 
J'_{r-2}=\begin{pmatrix}1&&\\ &J_{r-4}&\\ &&1\end{pmatrix}.
\]
Hence
\[
\s(w_1n)\s(w_0)=(1,\epsilon)
\s(\begin{pmatrix}-y^{-1}&-Zy^{-1}J_{r-2}J'_{r-2}&1\\ &I_{r-2}&\\ && y\end{pmatrix})
\s(w'_0\begin{pmatrix}1&Zy^{-1}J'_{r-2}&y^{-1}\\ &I_{r-2}&\\ &&1\end{pmatrix})
\]
for some $\epsilon=\epsilon(y, Z)\in\{\pm 1\}$, which {\it a priori}
depends on $y$ and $Z$. Here $Z$ is a $1\times (r-2)$ matrix. If we write
\[
Z=(Z', z)
\]
where $Z'$ is $1\times (r-3)$ and $z\in F$, then 
\[
-Zy^{-1}J_{r-2}J'_{r-2}=-y^{-1}(z, Z'
\left(\begin{smallmatrix}&1\\ I_{r-4}&\end{smallmatrix}\right))\quad\text{and}\quad
Zy^{-1}J'_{r-2}=y^{-1}(z, Z'\left(\begin{smallmatrix}1&\\ &J_{r-4}\end{smallmatrix}\right)),
\]
and hence
\begin{align*}\allowdisplaybreaks
&\int_{N_{1,r-1}}f^s(\s(w_1n)\s(w_0))\,dn\\
=&\int_{F^\times}\int_{F^{r-3}}\int_{F}\epsilon f^s\Big(
\s(\begin{pmatrix}-y^{-1}&-zy^{-1}&-y^{-1}Z'
\left(\begin{smallmatrix}&1\\ I_{r-4}&\end{smallmatrix}\right)
&1\\ &1 &&\\ &&I_{r-3}&\\ &&& y\end{pmatrix})\\
&\qquad\qquad\s(w'_0\begin{pmatrix}1&y^{-1}Z'\left(\begin{smallmatrix}1&\\
      &J_{r-4}\end{smallmatrix}\right)&zy^{-1}&y^{-1}\\ &I_{r-3}&&\\ &&1&\\
  &&&1\end{pmatrix})
\Big)\,dz\, dZ'\, dy\\
=&\int_{F^\times}\int_{F^{r-3}}\int_{F}\epsilon|y^{-1}|^{\frac{1}{4}(r-3)+s+\frac{1}{2}}
|y|^{(1-r)(s+\frac{1}{2})}
(\rr_{\chi}\otimest\cdots\otimest\rr_{\chi}\otimest\etat)_{\omega}
(\s(\begin{pmatrix}-y^{-1}&-zy^{-1}&&\\ &1 &&\\ &&I_{r-3}&\\ &&& y\end{pmatrix}))\\
&\qquad\qquad f^s\Big(
\s(w'_0\begin{pmatrix}1&y^{-1}Z'\left(\begin{smallmatrix}1&\\
      &J_{r-4}\end{smallmatrix}\right)&zy^{-1}&y^{-1}\\ &I_{r-3}&&\\
  &&1&\\
  &&&1\end{pmatrix})
\Big)\,dz\, dZ'\, dy\\
=&\int_{F^\times}\int_{F^{r-3}}\int_{F}\epsilon|y|^{-rs-\frac{3}{4}r+\frac{3}{4}}
(\rr_{\chi}\otimest\cdots\otimest\rr_{\chi}\otimest\etat)_{\omega}
(\s(\begin{pmatrix}-y^{-1}&-zy^{-1}&&\\ &1 &&\\ &&I_{r-3}&\\ &&& y\end{pmatrix}))\\
&\qquad\qquad f^s\Big(
\s(w'_0\begin{pmatrix}1&y^{-1}Z'\left(\begin{smallmatrix}1&\\
      &J_{r-4}\end{smallmatrix}\right)&zy^{-1}&y^{-1}\\ &I_{r-3}&&\\
  &&1&\\
  &&&1\end{pmatrix})
\Big)\,dz\, dZ'\, dy.
\end{align*}
By changing the variables $zy^{-1}\mapsto z$ and 
$y^{-1}Z'\left(\begin{smallmatrix}1&\\ &J_{r-4}\end{smallmatrix}\right)\mapsto Z'$,
then changing the additive measure $dy$ to the multiplicative measure
$d^\times y$, and then changing the variable $y^{-1}\mapsto y$, one
can see the above integral is written as
\begin{align*}\allowdisplaybreaks
&\int_{F^\times}\int_{F^{r-3}}\int_{F}\epsilon'|y|^{rs-\frac{1}{4}r+\frac{1}{4}}
(\rr_{\chi}\otimest\cdots\otimest\rr_{\chi}\otimest\etat)_{\omega}
(\s(\begin{pmatrix}-y&-z&&\\ &1 &&\\ &&I_{r-3}&\\ &&& y^{-1}\end{pmatrix}))\\
&\qquad\qquad\qquad\qquad
 f^s\Big(\s(w'_0\begin{pmatrix}1&Z'&z&y\\ &I_{r-3}&&\\ &&1&\\
  &&&1\end{pmatrix})
\Big)\,dz\, dZ'\, d^\times y.
\end{align*}
for some $\epsilon'=\epsilon'(y,Z)\in\{\pm 1\}$. 

Now let $a_1,\dots, a_l$ be a complete set of representatives of
$F^{\times 2}\backslash F^{\times }$. Then the above integral is written as
\begin{align*}\allowdisplaybreaks
&\sum_{i=1}^l\int_{F^{\times 2}}\int_{F^{r-2}}\epsilon_i'|xa_i|^{rs-\frac{1}{4}r+\frac{1}{4}}
(\rr_{\chi}\otimest\cdots\otimest\rr_{\chi}\otimest\etat)_{\omega}
(\s(\begin{pmatrix}-xa_i&-z&&\\ &1 &&\\ &&I_{r-3}&\\ &&& (xa_i)^{-1}\end{pmatrix}))\\
&\qquad\qquad\qquad\qquad
 f^s\Big(\s(w'_0\begin{pmatrix}1&Z'&z&xa_i\\ &I_{r-3}&&\\ &&1&\\
  &&&1\end{pmatrix})
\Big)\,dz\, dZ'\, d^\times x
\end{align*}
for a certain choice of the measure $d^\times x$ on $F^{\times 2}$ and
some $\epsilon_i'=\epsilon_i'(x,Z)\in\{\pm 1\}$,
which is further written as
\begin{align*}\allowdisplaybreaks
&\sum_{i=1}^l\int_{F^{\times 2}}\int_{F^{r-2}}\epsilon_i'|xa_i|^{rs-\frac{1}{4}r+\frac{1}{4}}
(\rr_{\chi}\otimest\cdots\otimest\rr_{\chi}\otimest\etat)_{\omega}
(\s(\begin{pmatrix}-x&-z&&\\ &1 &&\\ &&I_{r-3}&\\ &&& x^{-1}\end{pmatrix}))\\
&\qquad\qquad\qquad\qquad
 f^s\Big(\s(\begin{pmatrix}-a_i&&&\\ &1 &&\\ &&I_{r-3}&\\ &&& a_i^{-1}\end{pmatrix})
\s(w'_0\begin{pmatrix}1&Z'&z&xa_i\\ &I_{r-3}&&\\ &&1&\\
  &&&1\end{pmatrix})
\Big)\,dz\, dZ'\, d^\times x.
\end{align*}
Recall that we can view the section $f^s$ as an
element in the second induced space in  (\ref{E:inclusion2}), and in
particular we can and do view the expression $f^s(\cdots)$ in the
above integral as an element in the Schwartz space
$S_{\chi}(F)\otimes\cdots\otimes S_{\chi}(F)\otimes\C$. Therefore to show
the desired holomorphy, we may consider the integral
\begin{align*}\allowdisplaybreaks
&\sum_{i=1}^l\int_{F^{\times 2}}\int_{F^{r-2}}\epsilon_i'|xa_i|^{rs-\frac{1}{4}r+\frac{1}{4}}
(\rr_{\chi}^\psi\otimest\cdots\otimest\rr_{\chi}^\psi\otimest\etat)_{\omega}
(\s(\begin{pmatrix}-x&-z&&\\ &1 &&\\ &&I_{r-3}&\\ &&& x^{-1}\end{pmatrix}))\\
&\qquad
 f^s\Big(\s(\begin{pmatrix}-a_i&&&\\ &1 &&\\ &&I_{r-3}&\\ &&& a_i^{-1}\end{pmatrix})
\s(w'_0\begin{pmatrix}1&Z'&z&xa_i\\ &I_{r-3}&&\\ &&1&\\
  &&&1\end{pmatrix})
\Big)(t_1,\cdots,t_q)\,dz\, dZ'\, d^\times x,
\end{align*}
where $(t_1,\dots,t_q)\in F^q$ is fixed. By (\ref{E:Weil_action2}),
(\ref{E:Weil_action5}) and (\ref{E:central_character}), one can see
that the above integral is written as
\begin{align}\allowdisplaybreaks
\label{E:integral2}
&\sum_{i=1}^l\int_{F^{\times 2}}\int_{F^{r-2}}\epsilon_i'|xa_i|^{rs-\frac{1}{4}r+\frac{1}{2}}
\chi(x^{1/2})\eta(x)^{-1}\psi(-zt_1^2)\\
\notag&\qquad
 f^s\Big(\s(\begin{pmatrix}-a_i&&&\\ &1 &&\\ &&I_{r-3}&\\ &&& a_i^{-1}\end{pmatrix})
\s(w'_0\begin{pmatrix}1&Z'&z&xa_i\\ &I_{r-3}&&\\ &&1&\\
  &&&1\end{pmatrix})
\Big)(x^{1/2}t_1,\cdots,t_q)\,dz\, dZ'\, d^\times x,
\end{align}
which is independent of the choice of $x^{1/2}$.

Recall that the support of $f^s$ is in $\Qt w_0\Qt=
\Qt w'_0\Qt$. Also note that
\[
w'_0\begin{pmatrix}1&Z'&z&xa_i\\ &I_{r-3}&&\\ &&1&\\
  &&&1\end{pmatrix}{w'_0}^{-1}
=\begin{pmatrix}1&&&\\ &I_{r-3}&&\\ &&1&\\
 xa_i &Z''&z&1\end{pmatrix},
\]
where $Z''=Z'\begin{pmatrix}1&\\&J_{r-4}\end{pmatrix}$. Hence by Lemma
\ref{L:Rallis2} (with $w_0'$ in place of $w_0$), the map
\[
(x, Z', z)\mapsto f^s\Big(\s(\begin{pmatrix}-a_i&&&\\ &1 &&\\
  &&I_{r-3}&\\ &&& a_i^{-1}\end{pmatrix})
\s(w'_0\begin{pmatrix}1&Z'&z&xa_i\\ &I_{r-3}&&\\ &&1&\\
  &&&1\end{pmatrix})
\Big)
\]
is smooth and compactly supported and hence so is the map
\[
(x, Z', z)\mapsto f^s\Big(\s(\begin{pmatrix}-a_i&&&\\ &1 &&\\
  &&I_{r-3}&\\ &&& a_i^{-1}\end{pmatrix})
\s(w'_0\begin{pmatrix}1&Z'&z&xa_i\\ &I_{r-3}&&\\ &&1&\\
  &&&1\end{pmatrix})
\Big) (x^{1/2}t_1,\cdots,t_q).
\]
Therefore one can write 
\begin{align*}
 &\epsilon_i'f^s\Big(\s(\begin{pmatrix}-a_i&&&\\ &1 &&\\ &&I_{r-3}&\\
   &&& a_i^{-1}\end{pmatrix})
\s(w'_0\begin{pmatrix}1&Z'&z&xa_i\\ &I_{r-3}&&\\ &&1&\\
  &&&1\end{pmatrix})
\Big)(x^{1/2}t_1,\cdots,t_q)\\
&=\sum_{\lambda, \phi, \phi'}\lambda(s)\phi(x)\phi'(Z', z)
\end{align*}
for some holomorphic functions $\lambda$ and smooth compactly supported
functions $\phi$ and $\phi'$ on $F$ and $F^{r-3}\times F$,
respectively. Hence to study the analytic behavior of
\eqref{E:integral2} we have only to study that of
\[
\int_{F^{\times 2}}\int_{F^{r-2}}|x|^{rs-\frac{1}{4}r+\frac{1}{2}}
\chi(x^{1/2})\eta(x)^{-1}\phi(x)\phi'(Z)\,dZ d^\times x,
\]
which is written as
\[
\int_{F^{\times 2}}|x|^{rs-\frac{1}{4}r+\frac{1}{2}}
\chi(x^{1/2})\eta(x)^{-1}\phi(x)\,d^\times x\cdot\int_{F^{r-2}}
\phi'(Z)\,dZ,
\]
where recall we have put $Z=(Z',z)$. The integral over $Z$ is
independent of $s$, and hence we have only to
consider the first integral. But one can see
\[
\int_{F^{\times 2}}|x|^{rs-\frac{1}{4}r+\frac{1}{2}}
\chi(x^{1/2})\eta(x)^{-1}\phi(x)\,d^\times x
=c\int_{F^{\times}}|y^2|^{rs-\frac{1}{4}r+\frac{1}{2}}
\chi(y)\eta(y^2)^{-1}\phi(y^2)\,d^\times y
\]
for an appropriate non-zero constant $c$. By Tate's thesis, one
knows that this integral is $L(2(rs-\frac{1}{4}r+\frac{1}{2}),
\chi\eta^{-2})$ times an entire function on $s$, where this
$L$-factor is precisely the one appearing in
(\ref{E:normalized_intertwining_w_02}). Therefore
(\ref{E:normalized_intertwining_w_02}) is an entire function on $s$.

%%%%%%%%%%%%%%%%%%%%%%%%%%%%%%%%%%%%%%%%%%%%%%%%%%%%%%%%%%%%%%%%%%%

\subsection{\bf The case $r=2$}

%%%%%%%%%%%%%%%%%%%%%%%%%%%%%%%%%%%%%%%%%%%%%%%%%%%%%%%%%%%%%%%%%%%

In Theorem \ref{T:normalized_intertwining}, we excluded the case
that $r=2$ and $\chi^2\eta^{-2}=1$. However, the argument above works
even in this case except at $s=0$; Namely Theorem
\ref{T:normalized_intertwining} holds even when $r=2$ and
$\chi^2\eta^{-2}=1$ except at $s=0$. Now at $s=0$, since for $r=2$
the inducing representation is cuspidal, the general theory
of Eisenstein series (\cite[Proposition IV. 1.11.(b)]{MW}) implies
that it is actually holomorphic at $s=0$. One can see
$(\chit\,\otimest\,\etat)_\omega=(\chit\,\otimest\,\chit)_\omega$ by
Proposition \ref{P:tensor_unique}, and
$^w(\chit\,\otimest\,\chit)_\omega=(\chit\,\otimest\,\chit)_\omega$ by
Proposition \ref{P:tensor_Weyl}. Hence
the map $A(0,\theta_{\chi, \eta}, w_1)$ is an endomorphism on
$\Ind_{\Bt}^{\GLt_2}=(\chit\,\otimest\,\chit)_\omega$. But by the
functional equation of the intertwining operator (\cite[Theorem
IV.1.10(b), p.141]{MW}) we must have $A(0,\theta_{\chi, \eta},
w_1)^2=\Id$. This implies that
on each irreducible submodule of
$\Ind_{\Bt}^{\GLt_2}(\chit\,\otimest\,\chit)_\omega$, the operator
$A(0,\theta_{\chi, \eta}, w_1)$ acts
as $\pm \Id$. Indeed, it is shown in \cite[Proposition 7.3
(ii)]{BG} that it acts as $-1$ on all of the induced space. Hence we have
\begin{Prop}\label{P:action_at_s=0}
For $r=2$ and $\chi^2\eta^{-2}=1$, the (global) intertwining operator $A(s,
\theta_{\chi, \eta}, w_1)$ is holomorphic for $\Re(s)\geq 0$ except a
possible simple pole at $s=\frac{1}{4}$. Moreover $A(0,\theta_{\chi,
  \eta}, w_1)$ acts as $-\Id$.
\end{Prop}

\quad

%%%%%%%%%%%%%%%%%%%%%%%%%%%%%%%%%%%%%%%%%%%%%%%%%%%%%%%%%%%%%%%%%%%

\section{\bf The unnormalized Eisenstein series}

%%%%%%%%%%%%%%%%%%%%%%%%%%%%%%%%%%%%%%%%%%%%%%%%%%%%%%%%%%%%%%%%%%%

Now we are ready to state the main theorem on the analytic behavior of
the (unnormalized) Eisenstein series. Let $\theta=\theta_{\chi, \eta}$
or $\vartheta_{\chi, \eta}$, depending
on the parity of $r$. In this paper, we consider the
Eisenstein series associated to the induced representation
\begin{equation}\label{E:induced_space}
\Ind_{\Qt(\A)}^{\GLt_r(\A)}\theta\otimes\delta_Q^s.
\end{equation}
Namely for each $f^s\in
\Ind_{\Qt(\A)}^{\GLt_r(\A)}\theta\otimes\delta_Q^s$, we let
\begin{equation}\label{E:Eisenstein_series}
E(g, s; f^s)=\sum_{\gamma\in Q(F)\backslash\GL_r(F)}f^s(\s(\gamma) g;1)
\end{equation}
for $g\in\GLt_r(\A)$, where we view each section $f^s$ as a function
\[
f^s:\GLt_r(\A)\rightarrow\text{ space of } \theta\otimes\delta_Q^{s+\frac{1}{2}},
\]
and by $f^s(\s(\gamma) g;1)$ we mean the automorphic form
$f^s(\s(\gamma) g)$ on $\MQt(\A)$ evaluated at the identity $1$. Also
for fixed $\s(\gamma) g$,  we often write $f^s(\s(\gamma) g;-)$, which
is viewed as an automorphic form in the space of
$\theta\otimes\delta_Q^s$, namely the function $\tilde{m}\mapsto
f^s(\s(\gamma) g;\tilde{m})$ is an automorphic form on $\MQt(\A)$.

The main theme of this section is to prove 
\begin{Thm}\label{T:main1}
The Eisenstein series $E(g, s;f^s)$ is
holomorphic for $\Re(s)\geq 0$ except that it possibly has a simple
pole at $s=\frac{1}{4}$, when $\chi^2\eta^{-2}=1$ and $r=2q$, or
$\chi\eta^{-2}=1$ and $r=2q+1$.
\end{Thm}

In what follows we will give a proof of the theorem. The basic
idea seems to be standard in that we will compute a constant term of
the Eisenstein series and argue inductively on $r$. Indeed, the most
of the ideas
(at least for the case $\theta=\theta_{\chi, \eta}$) are already present in \cite{BG} and we
will borrow many of the ideas from there. Let us note, however, that
the cuspidal support of our Eisenstein series differs for the two
cases $\theta=\theta_{\chi, \eta}$ and $\theta=\vartheta_{\chi,
  \eta}$, and hence we will compute different constant terms for
those cases.

%%%%%%%%%%%%%%%%%%%%%%%%%%%%%%%%%%%%%%%%%%%%%%%%%%%%%%%%%%%%%%%%%%%

\subsection{\bf The base step of induction}

%%%%%%%%%%%%%%%%%%%%%%%%%%%%%%%%%%%%%%%%%%%%%%%%%%%%%%%%%%%%%%%%%%%
The base step is $r=2$ for the
case $\theta=\theta_{\chi, \eta}$ and $r=3$ for the case
$\theta=\vartheta_{\chi, \eta}$ and $\chi^{1/2}$ does not exist. (If
$r$ is odd and $\chi^{1/2}$ exists, then we have
$\theta=\theta_{\chi^{1/2},\eta}$ and hence the base step will be the case
$r=2$.)

Consider the case $r=2$. Then $\theta=\theta_{\chi,
  \eta}=(\chit\,\otimest\etat)_\omega$. The analytic property of the
Eisenstein series $E(g, s; f^s)$ is determined by the contant term $E_{N_{B}}(g, s; f^s)$
along the unipotent radical $N_{B}$ of $B$. By the standard calculation, one has
\begin{equation}\label{E:no_constant_term}
E_B(g, s; f^s)
=f^s(g)+\int_{N_B}f^s(\s(w_1^{-1}n)g)\,dn
=f^s(g)+A(s,\theta_{\chi,\eta},w_1)f^s(g),
\end{equation}
where $w_1$ is as in \eqref{E:w_1}. By Theorem
\ref{T:normalized_intertwining} (and Proposition
\ref{P:action_at_s=0})
above we know that $A(s,\theta, w_1)$ is
holomorphic for $\Re(s)\geq 0$ except that if $\chi^2\eta^{-1}=1$, it
has a possible pole at $\frac{1}{4}$. Hence Theorem
\ref{T:main1} holds for $r=2$. (Though we do not need this fact, let us
mention that at $s=\frac{1}{4}$, the
intertwining operator does have a pole and the residues generate the
non-twisted exceptional representation for $r=2$ of determinantal
character $\chi$, namely the Weil representation $\rr_{\chi^2}$.)

Next consider the case $r=3$ (so
$\vartheta_{\chi,\eta}=(\rr_\chi\otimest\,\etat)_\omega $) and $\chi^{1/2}$ does not
exist. Since $\chi^{1/2}$ does not exist, necessarily
$\chi\eta^{-2}\neq 1$. Then the inducing representation
$\vartheta_{\chi,\eta}$ is cuspidal, since $\rr_\chi$ is
cuspidal, and the metaplectic tensor product preserves
cuspidality (Proposition \ref{P:tensor_cuspidal}). Moreover the Levi
$\GLt_2\timest\GLt_1$ is maximal and non-self-conjugate, and hence the
Eisenstein series $E(g, s;f^s)$ is entire as desired.

%%%%%%%%%%%%%%%%%%%%%%%%%%%%%%%%%%%%%%%%%%%%%%%%%%%%%%%%%%%%%%%%%%%

\subsection{\bf The induction step for $r=2q$}

%%%%%%%%%%%%%%%%%%%%%%%%%%%%%%%%%%%%%%%%%%%%%%%%%%%%%%%%%%%%%%%%%%%

Now we will prove the induction step. Let us first consider the case
$r=2q>2$, namely $\theta=\theta_{\chi, \eta}$. This basically coincides with
the case treated by Bump and Ginzburg in \cite{BG}. It seems to the
author, however, that some of their arguments
cannot be justified without the theory of metaplectic tensor products
developed in \cite{Takeda2}. Also in \cite{BG}, they use the induction
argument for the normalized Eisenstein series. However, as we will point out
later, their induction argument does not seem to work because the finite set $S$
used to normalize the Eisenstein series depends not only on $\chi$ and
$\eta$, but also on the choice of the section $f^s$, which makes the
induction hypothesis not applicable. (This is another error in
\cite{BG} which does not seem to have been pointed out anywhere else.)
Indeed, to obtain the holomorphy for the normalized Eisenstein series,
one needs to use the functional equation of the Eisenstein series as
we will do in a later section.
Moreover there are a quite few places in
\cite{BG} where important computations are omitted. For those reasons,
we will write out the computations in detail, though our
computations are quite parallel to those in \cite{BG}.

Now for the case at hand, the cuspidal support of our Eisenstein series $E(g,s;f^s)$ is
the Borel subgroup and hence the poles of the Eisenstein series are
precisely the poles of the constant term along any parabolic. So in
particular in this subsection we let 
\[
P=P_{1,r-1}=M_PN_P\subseteq\GL_r
\]
be the $(1,r-1)$-parabolic, and will consider the constant term along
the unipotent radical $N_P$ of this parabolic. The constant term along
$N_P$ is computed as
\begin{align}\label{E:constant_term_even1}\allowdisplaybreaks
E_{P}(g,s; f^s)&=\int_{N_P(F)\backslash N_P(\A)}
E(\s(n)g,s;f^s)\,dn\notag\\
&=\int_{N_P(F)\backslash
  N_P(\A)} \sum_{\gamma\in Q(F)\backslash\GL_r(F)}
f^s(\s(\gamma)\s(n)g;1)\,dn\notag\\
&=\int_{N_P(F)\backslash
  N_P(\A)} \sum_{\gamma\in Q(F)\backslash\GL_r(F)\slash
  N_P(F)}\sum_{n'\in N_P(F)^{\gamma^{-1}}\backslash N_P(F)}
f^s(\s(\gamma n')\s(n)g;1)\,dn\notag\\
&=\sum_{\gamma\in Q(F)\backslash\GL_r(F)\slash
  N_P(F)}\int_{N_P(F)\backslash
  N_P(\A)} \sum_{n'\in N_P(F)^{\gamma^{-1}}\backslash N_P(F)}
f^s(\s(\gamma n'n)g;1)\,dn\notag\\
&=\sum_{\gamma\in Q(F)\backslash\GL_r(F)\slash
  N_P(F)}\int_{N_P(F)^{\gamma^{-1}}\backslash
  N_P(\A)}f^s(\s(\gamma n)g;1)\,dn,
\end{align}
where $N_P(F)^{\gamma^{-1}}=N_P(F)\cap\gamma^{-1}Q(F)\gamma$ and also
for the fourth equality we used Lemma \ref{L:section_s}.

By the Bruhat decomposition, we have
\[
\GL_r(F)=
\bigcup_{w\in\{1, w_1\}} Q(F)w^{-1} P(F),
\]
where $w_1$ is as in \eqref{E:w_1}. Accordingly, we have
\[
Q\backslash\GL_r\slash N_P
=\bigcup_{w\in\{1,w_1\}}Q\backslash Qw^{-1}P\slash N_P
=\bigcup_{w\in\{1,w_1\}}Q\backslash Qw^{-1} M_P
=\bigcup_{w\in\{1,w_1\}}M_P\cap wQw^{-1}\backslash M_P,
\]
where the last equality is given by the map $\gamma=w^{-1}m\mapsto m$ for
$m\in M_P$. Notice that
\[
M_P\cap wQw^{-1}\backslash M_P=
\begin{cases}
P_{r-2, 1}^{r-1}\backslash \GL_{r-1},&\text{if $w=1$};\\
1,&\text{if $w=w_1$},
\end{cases}
\]
where $\GL_{r-1}$ is viewed as a subgroup of $\GL_r$ embedded in the
lower right corner, and $P_{r-2, 1}^{r-1}$ is the $(r-2, 1)$-parabolic
of $\GL_{r-1}$. (Recall the notation from the notation section.)

Using this decomposition, one can write (\ref{E:constant_term_even1})
as
\begin{align*}
&\sum_{\gamma\in Q(F)\backslash\GL_r(F)\slash
  N_P(F)}\int_{N_P(F)^{\gamma^{-1}}\backslash
  N_P(\A)}f^s(\s(\gamma n)g;1)\,dn\notag\\
=&\sum_{w\in\{1,w_1\}}\sum_{m\in M_P(F)\cap wQ(F)w^{-1}\backslash
  M_P(F)}
\int_{N_P(F)^{m^{-1}w}\backslash
  N_P(\A)}f^s(\s(w^{-1}m n)g;1)\,dn\notag\\
=&\left(\sum_{m\in P_{r-2, 1}^{r-1}(F)\backslash\GL_{r-1}(F)}
\int_{N_P(F)\backslash
  N_P(\A)}f^s(\s(m n)g;1)\,dn\right)
+\int_{N_P(F)^{w_1}\backslash
  N_P(\A)}f^s(\s(w_1^{-1}n)g;1)\,dn\\
=&E_P(g,s;f^s)_{\Id}+E_P(g,s;f^s)_{w_1},
\end{align*}
where we have set
\begin{align*}
E_P(g,s;f^s)_{\Id}&:=\sum_{m\in P_{r-2, 1}^{r-1}(F)\backslash\GL_{r-1}(F)}
\int_{N_P(F)\backslash
  N_P(\A)}f^s(\s(m n)g;1)\,dn;\\
E_P(g,s;f^s)_{w_1}&:=\int_{N_P(\A)}f^s(\s(w_1^{-1}n)g;1)\,dn.
\end{align*}
Note that we used $N_P(F)^{w_1}=1$. To sum up, we have obtained
\begin{equation}\label{E:constant_term_even2}
E_P(g,s;f^s)=E_P(g,s;f^s)_{\Id}+E_P(g,s;f^s)_{w_1}.
\end{equation}

In what follows, we will show that the non-identity term
$E_P(g,s;f^s)_{w_1}$ is holomorphic for $\Re(s)\geq 0$ except
when $\chi^{2}\eta^{-2}=1$, in which case it has possible poles at  
$s=\frac{1}{4}$ and $s=\frac{1}{4}-\frac{1}{2r}$, and then the
identity term $E_P(g,s;f^s)_{\Id}$ is holomorphic for $\Re(s)\geq 0$ except
when $\chi^{2}\eta^{-2}=1$, in which case it has a possible pole at  
$s=\frac{1}{4}-\frac{1}{2r}$. Then we will show that the
possible poles at $s=\frac{1}{4}-\frac{1}{2r}$ (if exist at all)
coming from both terms cancel each other. This will complete the induction.

\quad

\noindent\underline{\bf The non-identity term $E_P(g,s;f^s)_{w_1}$}:

First consider the non-identity term. Note that the non-identity term
is written as
\[
E_P(g,s;f^s)_{w_1}=\int_{N_P(\A)}f^s(\s(w_1^{-1}n)g;1)\,dn
=A(s, \theta_{\chi,\eta}, w_1)f^s,
\]
where $A(s, \theta_{\chi,\eta}, w_1)$ is the intertwining operator
studied in Section \ref{S:normalized_intertwining}. From Theorem
\ref{T:normalized_intertwining}, we know that $A(s,
\theta_{\chi,\eta}, w_1)$  is holomorphic for $\Re(s)\geq 0$ except
when $\chi^{2}\eta^{-2}=1$, in which case it has possible poles at  
$s=\frac{1}{4}$ and $s=\frac{1}{4}-\frac{1}{2r}$.

\quad

\noindent\underline{\bf The identity term $E_P(g,s;f^s)_{\Id}$}:

One can write
\[
E_P(g,s;f^s)_{\Id}
=\sum_{m\in P_{r-2, 1}^{r-1}(F)\backslash\GL_{r-1}(F)}
\int_{N_{P'}(F)\backslash
  N_{P'}(\A)}f^s(\s(m)g;\s(n'))\,dn',
\]
where 
\[
P'=P^{r-1,1}_{1,r-2,1}=P\cap M_Q=(\GL_1\times\GL_{r-2}\times\GL_1)N_{P'},
\]
so $P'$ is the parabolic subgroup of
$\GL_{r-1}\times\GL_1$ whose Levi is $\GL_1\times\GL_{r-2}\times\GL_1$.
If one views $f^s(\s(m)g;-)$ as an automorphic
form in $\theta\otimes\delta_Q^{s+1/2}$ on $\MQt(\A)$ as explained at
the beginning of this section, the integral in the above
sum is just the constant term
along $N_{P'}$. But since $f^s(\s(m)g;-)\in\theta\otimes\delta_Q^{s+1/2}$,
we need to compute the constant term of the residual representation
$\theta\otimes\delta_Q^{s+1/2}$ along $N_{P'}$. Recall that the exceptional
representation $\theta$ is constructed as the residue of the
Eisenstein series associated to the
induced representation
$\Ind_{\Bt^{r-1,1}}^{\MQt}(\chit\otimest\cdots\otimest\chit\otimest\etat)_\omega^\nu$
at $\nu=\rho_{B^{r-1,1}}/2$, where $B^{r-1,1}$ is the Borel subgroup of
$\GL_{r-1}\times\GL_1$. Namely it is generated by
\[
\underset{\nu=\frac{1}{2}\rho_{B^{r-1,1}}}{\Res}E^{\MQt}(-, \varphi^\nu;\nu)
\]
for $\varphi^\nu\in
\Ind_{\Bt^{r-1,1}}^{\MQt}(\chit\otimest\cdots\otimest\chit\otimest\etat)_\omega^\nu$. (Here
the superscript for the Eisenstein series is the group on which the
Eisenstein series is defined. We will use this convention in what
follows as well.) But the constant
term of the residue is the same as the residue of the constant term,
and hence one first needs to compute the constant term $E^{\MQt}_{N_{P'}}(-,
\varphi^\nu;\nu)$ of $E^{\MQt}(-,
\varphi^\nu;\nu)$ along $N_{P'}$. For this, one can use
\cite[Proposition (ii), p.92]{MW}, and obtain
\[
E^{\MQt}_{P'}(m',\varphi^\nu;\nu)
=\sum_w
E^{\Mt_{p'}}(m',M(w,\nu)\varphi^\nu),
\]
where $m'\in\Mt_{p'}$ and $w$ runs through all the Weyl
group elements of $\GL_{r-1}\times\GL_1$
such that $w^{-1}(\alpha)>0$ for all the positive roots $\alpha$ that are
in $M_{P'}$. Note that 
\[
M(w,\nu)\varphi^\nu\in
\Ind_{\Bt^{r-1,1}}^{\MQt}\,^w(\chit\otimest\cdots\otimest\chit\otimest\etat)_\omega^{w\nu},
\]
but it is viewed as a map on $\Mt_{P'}$ by
restriction. Hence we actually have
\[
M(w,\nu)\varphi^\nu\in
\Ind_{\Bt^{1,r-1,1}}^{\Mt_{P'}}\,^w(\chit\otimest\cdots\otimest\chit\otimest\etat)_\omega^{w\nu}
\delta_{B^{r-1,1}}^{1/2}\delta_{B^{1,r-1,1}}^{-1/2}.
\]
(Note that since our induction is normalized, we need the
modulus characters $\delta_{B^{r-1,1}}^{1/2}\delta_{B^{1,r-1,1}}^{-1/2}$.)

One can see that by using the language of permutations, $w$
runs through all the elements of the form
\[
w=(12\cdots k),\quad \text{for $k=1,\dots,r-1$}.
\]
We need to compute the residue at $\nu=\rho_{B^{r-1,1}}/2$ of
each
$E^{\Mt_{P'}}(m',M(w,\nu)\varphi^\nu)$
for such $w$. But one can see that this Eisensteins series has a
residue at $\nu=\rho_{B^{r-1,1}}/2$ only for $w=(12\cdots r-1)$ for
the following reason: Since the cuspidal support of the Eisenstein
series $E^{\Mt_{P'}}$ is the Borel, the analytic behavior is determined by the
constant term $E^{\Mt_{P'}}_{B^{1,r-1,1}}(m', M(w,\nu)\varphi^\nu)$
along the Borel $B^{1,r-1,1}$ of $M_{P'}$. By using \cite[Proposition
(i), p.92]{MW}, one can see
\[
E^{\Mt_{P'}}_{B^{1,r-1,1}}(m',M(w,\nu)\varphi^\nu)
=\sum_{w'}M(w',{w\nu})\circ M(w,\nu)\varphi^\nu,
\]
where $w'$ runs through all the elements in the Weyl group of
$M_{P'}$. We have
\[
M(w',{w\nu})\circ M(w,\nu)=M(w'w,\nu).
\]
But we know from \cite[Theorem II.1.3]{KP} that the intertwining operator $M(w'w,\nu)
$ for the exceptional representation
$(\chit\otimest\cdots\otimest\chit\otimest\etat)_\omega^\nu$ has a
residue at $\nu=\rho_{B^{r-1,1}}/2$ only when $w'w$ is the longest
element, which implies $w$ must be of the form $(12\cdots
r-1)$. Therefore we have
\[
\underset{\nu=\frac{1}{2}\rho_{B^{r-1,1}}}{\Res}E^{\MQt}_{P'}(m',\varphi^\nu;\nu)
=\underset{\nu=\frac{1}{2}\rho_{B^{r-1,1}}}{\Res}E^{\Mt_{P'}}(m',M(w,\nu)\varphi^\nu),
\]
where $w=(12\cdots r-1)$.

Now as in the notation section let
\[
\nu=s_1e_1+\cdots+s_{r-1}e_{r-1}\in\Phi_{B^{r-1,1}}(\C)\cong\C^{r-2}
\]
for $s_i\in\C$ with $s_1+\cdots+s_{r-1}=0$. With this notation, for
$w=(12\cdots r-1)$ we have
\begin{align*}
w\nu=&s_{r-1}e_1+s_1e_2+s_2e_3+\cdots+s_{r-2}e_{r-1}\\
=&s_{r-1}e_1+(s_1+\frac{s_{r-1}}{r-2})e_2+(s_2+\frac{s_{r-1}}{r-2})e_3+\cdots+(s_{r-2}+\frac{s_{r-1}}{r-2})e_{r-1}\\
&-\frac{s_{r-1}}{r-2}(e_2+\cdots+e_{r-1})\\
=&\frac{s_{r-1}}{r-2}\cdot(2\rho_{P'})+(s_1+\frac{s_{r-1}}{r-2})e_2+(s_2+\frac{s_{r-1}}{r-2})e_3+\cdots+(s_{r-2}+\frac{s_{r-1}}{r-2})e_{r-1}\\
=&\frac{s_{r-1}}{r-2}\cdot(2\rho_{P'})+\nu'
\end{align*}
where we have set
\[
\nu'=(s_1+\frac{s_{r-1}}{r-2})e_2+(s_2+\frac{s_{r-1}}{r-2})e_3+\cdots+
(s_{r-2}+\frac{s_{r-1}}{r-2})e_{r-1}.
\]
Note that $\nu'\in\Phi_{B^{1,r-2,1}}(\C)\cong\C^{r-3}$. Now we have
\[
\frac{1}{2}\rho_{B^{r-1,1}}
=\frac{1}{4}\big((r-2)e_1+(r-4)e_2+\cdots+(2-r)e_{r-1}\big)
\]
and hence
\begin{align*}
\frac{1}{2}w\rho_{B^{r-1,1}}
&=\frac{1}{4}\Big(-2\rho_{P^{r-1,1}_{1,r-2,1}}+(r-3)e_2+(r-5)e_3
+\cdots+(3-r)e_{r-1}\Big)\\
&=-\frac{1}{2}\rho_{P^{r-1,1}_{1,r-2,1}}+\frac{1}{2}\rho_{B^{1,r-2,1}}.
\end{align*}
But here $-\frac{1}{2}\rho_{P^{r-1,1}_{1,r-2,1}}$  is just a character on the Levi
$M_{P^{r-1,1}_{1,r-2,1}}=\GL_1\times\GL_{r-2}\times\GL_1$ which acts as
\[
(a, g_{r-2}, b)\mapsto 
|a|^{-\frac{r-2}{4}} |{\det}(g_{r-2})|^{\frac{1}{4}}
\]
for $(a, g_{r-2}, b)\in \GL_1\times\GL_{r-2}\times\GL_1$. Hence for
$m'\in M_{P'}=\GL_1\times\GL_{r-2}\times\GL_1$, we have
\[
\underset{\nu=\frac{1}{2}\rho_{B^{r-1,1}}}{\Res}E^{\Mt_{P'}}(m',M(w,\nu)\varphi^\nu)
=\underset{\nu'=\frac{1}{2}\rho_{B^{1,r-2,1}}}{\Res}E^{\Mt_{P'}}(m',{\varphi'}^{\nu'})
\]
where 
\[
{\varphi'}^{\nu'}\in\Ind_{\Bt^{1,r-2,1}}^{\Mt_{P'}}(\chit|-|^{-\frac{r-2}{4}}\,\otimest\,
(\underbrace{\chit\,\otimest\cdots\otimest\,\chit}
_{r-2 \text{ times }})|{\det}_{r-2}|^{\frac{1}{4}}\,\otimest\,\etat)_\omega^{\nu'}
\delta_{B^{r-1,1}}^{1/2}\delta_{B^{1,r-2,1}}^{-1/2}.
\]
By computing $\delta_{B^{r-1,1}}^{1/2}\delta_{B^{1,r-2,1}}^{-1/2}$,
this induced representation is written as
\[
\Ind_{\Bt^{1,r-2,1}}^{\Mt_{P'}}(\chit|-|^{\frac{r-2}{4}}\,\otimest\,
(\chit\,\otimest\cdots\otimest\,\chit)|{\det}_{r-2}|^{-\frac{1}{4}}\,
\otimest\,\etat)_\omega^{\nu'}.
\]

The inducing representation in this induced representation is
the exceptional character (with some character twits) and hence the
residue at $\nu'=\frac{1}{2}\rho_{B^{1,r-2,1}} $ of the Eisenstein
series for this induced representation gives rise to the exceptional
representation of $\Mt_{P'}$. But by the compatibility of parabolic
inductions for metaplectic tensor products (Proposition \ref{P:tensor_parabolic}),
one can see that
\[
\underset{\nu'=\frac{1}{2}\rho_{B^{1,r-2,1}}}{\Res}E^{\Mt_{P'}}(-,{\varphi'}^{\nu'})
\in(\chit|-|^{\frac{r-2}{4}}\,\otimest\,\theta'|{\det}_{r-2}|^{-\frac{1}{4}}\,\otimest\,\etat)_\omega,
\]
where $\theta'$ is the exceptional representation of $\GLt_{r-2}$
which is the unique irreducible quotient of the induced representation
\[
\Ind_{\Bt^{r-2}}^{\GLt_{r-2}}(\chit\,\otimest\cdots\otimest\,\chit)_{\omega'}
\otimes\delta_{B^{r-2}}^{1/4}
\]
for an appropriate choice of $\omega'$.

Hence for each fixed $m\in\GL_{r-1}(F)$, the function on
$\Mt_{P'}(\A)$ given by
\[
m'\mapsto \int_{N_{P'}(F)\backslash
  N_{P'}(\A)}f^s(\s(m)g;\s(n')m')\,dn'
\]
is an element in
$(\chit|-|^{\frac{r-2}{4}}\,\otimest\,\theta'|{\det}_{r-2}|^{-\frac{1}{4}}\,\otimest\,\etat)_\omega
\otimes\delta_Q^{s+\frac{1}{2}}$, where $\delta_Q^{s+\frac{1}{2}}$ is actually the restriction
$\delta_Q^{s+\frac{1}{2}}|_{P'}$ to $P'$. One can compute
\[
\delta_Q^{s+\frac{1}{2}}(a, g_{r-2},
b)=|a|^{s+\frac{1}{2}}|{\det}(g_{r-2})|^{s+\frac{1}{2}}
|b|^{-(r-1)(s+\frac{1}{2})}.
\]

Accordingly the function $f_{N_{P'}}^s$ on $\MPt(A)$ defined by
\[
f_{N_{P'}}^s(g)= \int_{N_{P'}(F)\backslash
  N_{P'}(\A)}f^s(g;\s(n'))\,dn'
\]
is in
\begin{align*}
&\Ind_{\widetilde{P'}(\A)}^{\MPt(\A)}(\chit|-|^{\frac{r-2}{4}}\,\otimest\,\theta'|{\det}_{r-2}|^{-\frac{1}{4}}
\,\otimest\,\etat)_\omega\otimes\delta_Q^{s+\frac{1}{2}}\delta_{P'}^{-1/2}\\
=&\Ind_{\widetilde{P'}(\A)}^{\MPt(\A)}(\chit|-|^{-\frac{r-2}{4}}\,\otimest\,\theta'
|{\det}_{r-2}|^{\frac{1}{4}}\,\otimest\,\etat)_\omega\otimes\delta_Q^{s+\frac{1}{2}}.
\end{align*}

Recall we are trying to figure out the analytic behavior of the
Eisenstein series
\begin{equation}\label{E:identity_term1}
E_P(g,s;f^s)_{\Id}
=\sum_{m\in P_{r-2, 1}^{r-1}(F)\backslash\GL_{r-1}(F)}f_{N_{P'}}^s(\s(m)
g;1)
\end{equation}
as $g$ runs through all elements in $\MPt(\A)$ and $f^s$ runs through all the
sections. But for this purpose we may assume that
$g=((1,h),\xi)\in\MPt(\A)$ is such that
$(1,h)\in\GLt_1(\A)\times\GLt_{r-1}(\A)$, because if $g$
is not of this form, one can
always translate $f^s$ by an appropriate element in $\MPt(\A)$. Namely
we consider the section $f_{N_{P'}}^s|_{\GLt_{r-1}(\A)}$. Let
\[
F^s:=f_{N_{P'}}^s|_{\GLt_{r-1}(\A)}.
\]
Then one can see, by using Proposition \ref{P:tensor_restriction} on restriction of metaplectic
tensor product on smaller Levi, that 
\[
F^s\in \bigoplus_{\delta}\Ind_{\Pt^{r-1}_{r-2,1}(\A)}^{\GLt_{r-1}(\A)}(\theta'
|{\det}_{r-2}|^{\frac{1}{4}}\,\otimest\,\etat)_{\omega_\delta}\otimes\delta_Q^{s+\frac{1}{2}}
\delta_{P'}^{1/2}\delta_{P_{r-2,1}^{r-1}}^{-1/2},
\]
where $\delta$ runs through a subset of $\GL_1(F)$ and $\omega_\delta$
is an appropriate character on the center of $\GLt_{r-1}$. Hence after all, the analytic
behavior of (\ref{E:identity_term1}) is determined by the analytic
behavior of the Eisenstein series on $\GLt_{r-1}(\A)$ given by
\[
\sum_{m\in P_{r-2, 1}^{r-1}(F)\backslash GL_{r-1}(F)}F^s(\s(m)g;1)
\]
where $F^s\in\Ind_{\Pt^{r-1}_{r-2,1}(\A)}^{\GLt_{r-1}(\A)}(\theta'
|{\det}_{r-2}|^{\frac{1}{4}}\,\otimest\,\etat)_{\omega}\otimes\delta_Q^{s+\frac{1}{2}}
\delta_{P'}^{1/2}\delta_{P_{r-2,1}^{r-1}}^{-1/2}$ and $\omega$ is some
appropriately chosen character and 
$\delta_Q^{s+\frac{1}{2}}\delta_{P'}^{1/2}\delta_{P_{r-2,1}^{r-1}}^{-1/2}$
is restricted to $P_{r-2,1}^{r-1}$. 
As a character on $\GL_{r-2}\times\GL_1\subseteq P_{r-2,1}^{r-1}$, one
can compute
\begin{align*}
|\det(g_{r-2})|^{\frac{1}{4}}
\delta_Q^{s+\frac{1}{2}}\delta_{P'}^{1/2}\delta_{P_{r-2,1}^{r-1}}^{-1/2} (g_{r-2},b)&=
|\det(g_{r-2})|^{s-\frac{1}{4}}|b|^{-(s+\frac{1}{2})(r-1)+\frac{1}{2}(r-2)}\\
&=\delta_{P_{r-2,1}^{r-1}}(g_{r-1}, b)^{\frac{rs+\frac{1}{4}}{r-1}}
|\det(g_{r-1})b|^{\frac{-\frac{1}{4}r-s}{r-1}}.
\end{align*}
Thus the section $F^s$ belongs to
\begin{equation}\label{E:induced_space_r-1}
\Ind_{\Pt^{r-1}_{r-2,1}(\A)}^{\GLt_{r-1}(\A)}(\theta'\otimest\,\etat)_{\omega}\otimes
\delta_{P_{r-2,1}^{r-1}}^b|{\det}_{r-1}|^{a},
\end{equation}
where
\begin{equation}\label{E:exponent_b}
a=\frac{-\frac{1}{4}r-s}{r-1},\qquad b=\frac{rs+\frac{1}{4}}{r-1}.
\end{equation}
(Note that those two exponents
$a$ and $b$ are precisely the ones in the middle of p. 196 of
\cite{BG}, though in \cite{BG} induction is not normalized and hence
their exponents look different for ours.)

Therefore the analytic behavior of the Eisenstein series in
(\ref{E:identity_term1}) is determined by that of the Eisenstein
series associated to the induced representation
(\ref{E:induced_space_r-1}). But the twist by $|{\det}_{r-1}|^{a}$
does not have any affect on the analytic behavior, and hence we have
to consider the Eisenstein series on $\GLt_{r-1}(\A)$ associated with
the induced space
\[
\Ind_{\Pt^{r-1}_{r-2,1}(\A)}^{\GLt_{r-1}(\A)}(\theta'\otimest\,\etat)_{\omega}\otimes
\delta_{P_{r-2,1}^{r-1}}^{b},
\]
where $b$ is as in (\ref{E:exponent_b}).

Now by the induction hypothesis, this Eisenstein series on
$\GLt_{r-1}(\A)$ is holomorphic for $\Re(b)\geq 0$,
except when $\chi^2\eta^{-2}=1$ in which case it has a possible simple
pole at $b=\frac{1}{4}$. From (\ref{E:exponent_b}), $b=\frac{1}{4}$
amounts to
\begin{equation}\label{E:exponent_s}
s=\frac{1}{4}-\frac{1}{2r}.
\end{equation}

\begin{Rmk}
The above argument is essentially the detail of the argument outlined
in \cite[p.195-196]{BG}. As we pointed out at the beginning of the
section, however, in \cite{BG} it seems the induction argument is used for the
normalized Eisenstein series rather than the unnormalized one. But the
normalization of the Eisenstein series depends on the choice of the
set $S$ of ``bad places'', which include the bad places for the
section $f^s$. As one can see from the above computation, one has to apply the induction
hypothesis to the new section $F^s$ on the lower rank group $\GLt_{r-1}(\A)$. But there
is no guarantee that the same set $S$ works for $F^s$. This is why we
cannot use the induction argument for the normalized Eisenstein series.
\end{Rmk}

\quad

\noindent\underline{\bf Cancellation of the poles at $s=\frac{1}{4}-\frac{1}{2r}$}:

Finally, to complete the induction, we need to show  the (possible)
poles at $s=\frac{1}{4}-\frac{1}{2r}$ (if exist at all) of both the identity term
$E_P(g, s; f^s)_{\Id}$ and the non-identity term $E_P(g, s;
f^s)_{w_1}$ cancel out.

Set
\[
s_0:=\frac{1}{4}-\frac{1}{2r}.
\]
Let us note that the possible pole at
$s=s_0$ will happen only when $\chi^2\eta^{-2}=1$, namely
$\chi^2=\eta^2$. Moreover by the uniqueness of the metaplectic tensor
product (Proposition \ref{P:tensor_unique}), if $\chi^2\eta^{-2}=1$, then
\[
(\chit\otimest\cdots\otimest\chit\otimest\etat)_\omega
=(\chit\otimest\cdots\otimest\chit\otimest\chit)_\omega,
\]
\ie we may (and do) assume $\eta=\chi$, although most of the time we
use the notation
$(\chit\,\otimest\cdots\otimest\,\chit\,\otimest\,\etat)_\omega$.

As we have seen above, the two terms $E_P(g,s;f^s)_{\Id}$ and $E_P(g,s;f^s)_{w_1}$
both have a possible pole at $s=\frac{1}{4}-\frac{1}{2r}$.  But in what follows,
we will show the Eisenstein series $E(g,s;f^s)$ does not have a pole at
this point. Namely those two possible poles cancel each other or they
just do not exist to begin with. This is essentially shown in
\cite[p.201-203]{BG}. However many of the computations are omitted
there, and hence we will give a complete proof in detail here. The
basic idea is the following: First one computes the constant
term of our Eisenstein series along the Borel
subgroup $B$ instead of $P$. Then one can see that all the terms in the
constant term is holomorphic at $s=\frac{1}{4}-\frac{1}{2r}$ except
two terms. One can then see that the treatment of those two terms can be
reduced to the ``$\GLt_2$-case'', and invoke Proposition
\ref{P:action_at_s=0} to show the cancellation of the poles.

So let us compute the constant term $E_{B}(g, s; f^s) $ of the
Eisenstein series along the Borel. Analogously to (\ref{E:constant_term_even1})
\begin{align*}
E_{B}(g, s; f^s)=\sum_{\gamma\in Q(F)\backslash\GL_r(F)\slash
  N_B(F)}\int_{N_B(F)^{\gamma^{-1}}\backslash N_B(\A)}f^s(\s(\gamma n)g;1)\,dn,
\end{align*}
where $N_B(F)^{\gamma^{-1}}=N_B(F)\cap\gamma^{-1}Q(F)\gamma$. For
$i=0,\dots,r-1$, let
\[
w_i=\begin{pmatrix}I_{r-i-1}&&\\ &&1\\ &I_i&\end{pmatrix}.
\]
(Let us note that $w_1$ here differs from the one in
\eqref{E:w_1}, but this should not cause any confusion.)
By the Bruhat decomposition, one has
\[
Q(F)\backslash\GL_r(F)\slash N_B(F)=\bigcup_{i=0}^{r-1}Q\backslash
Qw_i^{-1}T_B=\bigcup_{i=0}^{r-1}T_B\cap w_iQ{w_i^{-1}}\backslash T_B,
\]
where the last equality is given by the map $\gamma=w_i^{-1}t\mapsto
t$ for $t\in T_B$. But $T_B\cap w_iQ{w_i^{-1}}\backslash T_B=1$ for
any $w_i$, and so each double coset in
$Q(F)\backslash\GL_r(F)\slash N_B(F)$ is represented by
$w_i^{-1}$. Hence
\begin{equation}\label{E:constant_term_even3}
E_{B}(g, s; f^s)=\sum_{i=0}^{r-1}\int_{N_B(F)^{w_i}\backslash N_B(\A)}f^s(\s(w_i^{-1} n)g;1)\,dn.
\end{equation}
Let us put
\begin{equation}\label{E:c_i}
c_i(g, s;f^s)=
\int_{N_B(F)^{w_i}\backslash N_B(\A)}f^s(\s(w_i^{-1} n)g;1)\,dn,
\end{equation}
so that
\begin{equation}\label{E:constant_term_even4}
E_{B}(g, s; f^s)=\sum_{i=0}^{r-1}c_i(g, s;f^s).
\end{equation}

Let us compute the term for $i=0$, so $w_0=1$ and
$N_B(F)^{w_0}=N_B(F)$. For $m\in\MQt$, define
\[
\hat{f}^s(g;m):=\int_{N_{B^{r-1,1}}(F)\backslash N_{B^{r-1,1}}(\A)}f^s(g;\s(n)m)\,dn.
\]
Namely $\hat{f}^s(g)$ is the constant term of the
automorphic form $f^s(g)\in\theta_{\chi,\eta}\otimes\delta_Q^{s+1/2}$
along the Borel subgroup $B^{r-1,1}$ of $M_Q=\GL_{r-1}\times\GL_1$. 
Then 
\begin{align*}
c_0(g,s;f^s)
&=\int_{N_B(F)\backslash N_B(\A)}f^s(\s(w_0^{-1} n)g;1)\,dn\\
&=\int_{N_{B^{r-1,1}}(F)\backslash
  N_{B^{r-1,1}}(\A)}f^s(\s(n)g;1)\,dn\\
&=\int_{N_{B^{r-1,1}}(F)\backslash
  N_{B^{r-1,1}}(\A)}f^s(g;\s(n))\,dn\\
&=\hat{f}^s(g;1)
\end{align*}
where we used that $N_B=N_QN_{B^{r-1,1}}$ and $N_Q(\A)$ acts trivially
on $f^s$.
But $\theta_{\chi,\eta}$ is
generated by the residues of Eisenstein series associated with the
induced representation
$\Ind_{B^{r-1,1}(\A)}^{\MQt(\A)}(\chit\,\otimest\cdots\otimest\,\chit\,\otimest\,\etat)^\nu_\omega$
at $\nu=\frac{1}{2}\rho_{B^{r-1,1}}$. Namely
\[
f^s(g;-)=\underset{\nu=\frac{1}{2}\rho_{B^{r-1,1}}}{\Res}E(-, \varphi^\nu)\otimes\delta_Q^{s+\frac{1}{2}}(-)
\]
for $\varphi^\nu\in\Ind_{B^{r-1,1}(\A)}^{\MQt(\A)}
(\chit\,\otimest\cdots\otimest\,\chit\,\otimest\,\etat)^\nu_\omega$. Now
a constant term of a residue is the residue of the constant
term. Hence we need to compute the constant term
$E_{B^{r-1,1}}(-,\varphi^\nu)$. By \cite[Proposition II.1.7 (ii), p.92]{MW}, we
can write
\[
E_{B^{r-1,1}}(-,\varphi^\nu)=\sum_{w\in W_{\GL_{r-1}}}M(w,\nu)\varphi_\nu(-)
\]
where $W_{\GL_{r-1}}$ is the Weyl group of $\GL_{r-1}$ embedded into
the left upper corner of $\GL_r$, and $M(w,\nu)$
is the intertwining operator 
\[
M(w,\nu):\Ind_{B^{r-1,1}(\A)}^{\MQt(\A)}(\chit\,\otimest\cdots\otimest\,\chit\,\otimest\,\etat)^\nu_\omega
\rightarrow \Ind_{B^{r-1,1}(\A)}^{\MQt(\A)}\,
^w(\chit\,\otimest\cdots\otimest\,\chit\,\otimest\,\etat)^{w\nu}_\omega.
\]
But we know from \cite[Theorem II.1.3]{KP} that this intertwining operator has a
residue at  $\nu=\frac{1}{2}\rho_{B^{r-1,1}}$ only for the longest
element $u=\begin{pmatrix} J_{r-1}&\\ &1\end{pmatrix}\in
W_{\GL_{r-1}}$, where $J_{r-1}$ is the $(r-1)\times (r-1)$ matrix with
$1$'s on the anti-diagonal entries and all the other entries
0. Moreover the residue is in the exceptional representation
$\theta_{\chi,\eta}$. Hence the function on $\Bt(\A)$ defined by
$b\mapsto\hat{f}^s(g;b)$ for $b\in\Bt(\A)$ is an automorphic form in
the space generated by
the constant terms of the exceptional representation
$\theta_{\chi,\eta}\otimes\delta_Q^{s+\frac{1}{2}}$, which is equal to 
\[
(\chit\,\otimest\cdots\otimest\,\chit\,\otimest\,\etat)_\omega\otimes\delta_{B^{r-1,1}}^{1/4}.
\]
Also for each $b, b'\in\Bt(\A)$, 
\begin{align*}
\hat{f}^s(b'g;b)&=\int_{N_{B^{r-1,1}}(F)\backslash
  N_{B^{r-1,1}}(\A)}f^s(b'g;\s(n)b)\,dn\\
&=\int_{N_{B^{r-1,1}}(F)\backslash
  N_{B^{r-1,1}}(\A)}f^s(g;\s(n)bb')\,dn\\
&=\hat{f}^s(g;bb').
\end{align*}
Hence we have 
\begin{equation}\label{E:Heisenberg1}
\hat{f}^s\in \ind_{\Bt(\A)}^{\GLt_r(\A)}
(\chit\,\otimest\cdots\otimest\,\chit\,\otimest\,\etat)_\omega
\otimes\delta_{B^{r-1,1}}^{1/4}\delta_Q^{s+\frac{1}{2}},
\end{equation}
where the induction is not normalized. 

Next consider $i>0$. Let
\[
N_i:=\{\begin{pmatrix}I_{r-i-1}&&\\ &1& X_i\\ &&I_i\end{pmatrix}:
X_i \text{ is a $1\times i$ matrix}\},
\]
\ie $N_i$ is the set of the ``$(r-i)^{\text{th}}$-rows'' of $N_B$. Also let
\[
U_{i}=\{\begin{pmatrix}I_{r-i-1}&0&Y_{r-i-1}\\0&I_i&0\\0&0&1\end{pmatrix}:
Y_{r-i-1} \text{ is a $(r-i-1)\times 1$ matrix}\},
\]
\ie $U_{i}$ is the set of the first $r-i-1$ entries in the last
column. Both $N_i$ and $U_{i}$ are subgroups of $N_B$. One can see
\[
w_i U_{i}N_{B^{r-1,1}} w_i^{-1} N_i=N_B,
\]
and
\[
N_B(F)^{w_i}=N_B(F)\cap w_iQ(F)w_i^{-1}=w_iU_{i}(F)N_{B^{r-1,1}}(F)w_i^{-1}.
\]
Therefore we have
\begin{align}\label{E:integral_c_i}
\notag c_i(g,s;f^s)
=&\int_{N_B(F)^{w_i}\backslash N_B(\A)}f^s(\s(w_i^{-1} n)g;1)\,dn\\
\notag=&\underset{w_iU_{i}(F)N_{B^{r-1,1}}(F)w_i^{-1}\backslash w_i
  U_{i}(\A)N_{B^{r-1,1}}(\A) w_i^{-1} N_i (\A)}{\int}
f^s(\s(w_i^{-1} n)g;1)\,dn\\
\notag=&\underset{N_i(\A)}{\int}\;\;
\underset{U_{i}(F)N_{B^{r-1,1}}(F)\backslash U_{i}(\A)N_{B^{r-1,1}}(\A)}{\int}
f^s(\s(n'w_i^{-1}n_i)g;1)\,dn'dn_i\\
=&\underset{N_i(\A)}{\int}\;\;
\underset{N_{B^{r-1,1}}(F)\backslash N_{B^{r-1,1}}(\A)}{\int}\;\;
\underset{U_i(F)\backslash U_i(\A)}{\int}
f^s(\s(u_i)\s(n_{r-1})\s(w_i^{-1}n_i)g;1)\,du_idn_{r-1}dn_i,
\end{align}
where for the last equality we used Lemma \ref{L:section_s} and that $w_i\in\GL_r(F)$.
Note that $U_i(\A)\subseteq N_Q(\A)$ acts trivially on $f^s$, and
$f^s(\s(u_i)\s(n_{r-1})\s(w_i^{-1}n_i)g;1)=f^s(\s(n_{r-1})\s(w_i^{-1}n_i)g;1)$. Hence the
inner most integral simply goes away. Furthermore,
$f^s(\s(n_{r-1})\s(w_i^{-1}n_i)g;1)=f^s(\s(w_i^{-1}n_i)g;\s(n_{r-1}))$. Therefore
the integral (\ref{E:integral_c_i}) is written as
\begin{align*}
\underset{N_i(\A)}{\int}\;\;
\underset{N_{B^{r-1,1}}(F)\backslash N_{B^{r-1,1}}(\A)}{\int}\;\;
f^s(\s(w_i^{-1}n_i)g;\s(n_{r-1}))dn_{r-1}dn_i
=\underset{N_i(\A)}{\int}\
\hat{f}^s(\s(w_i^{-1}n_i)g;1)dn_i.
\end{align*}
Therefore
\begin{align}\label{E:integral_c_i_2}
c_i(g,s;f^s)&=\int_{N_i(\A)}\hat{f}^s(\s(w_i^{-1})\s(n_i)g;1)\,dn_i.\\
\notag&=\underset{\A^i}{\int}\;\hat{f}^s\Big(\s(w_i^{-1})
\s(\begin{pmatrix}I_{r-i-1}&&\\ &1& X_i\\
  &&I_i\end{pmatrix})g;1\Big)\,dX_i.
\end{align}
This is precisely the formula stated (without a proof) at the
end of p.202 of \cite{BG}, though our $w_i^{-1}$ is their $w_i$.

We will show that $c_i(g, s;f^s)$ is holomorphic at
$s=s_0$ for all $i<r-2$, and for $i=r$ and $r-1$ it does have a pole
but they cancel out, and hence the constant term
(\ref{E:constant_term_even4}) has no pole at $s=s_0$, which implies
that the Eisenstein series $E(g,s,f^s)$ does not have a pole at $s=s_0$.

For this purpose, we need to reduce our situation to the
``$\GLt_2$-situation''. For this, just as is done in \cite{BG}, we need to
interpret the metaplectic tensor product
$(\chit\,\otimest\cdots\otimest\,\chit\,\otimest\,\etat)_\omega$ as follows.
\begin{Lem}\label{L:Heisenberg}
Let $\Tm(\A)$ be the subgroup of the maximum torus $T(\A)$ of the form
\[
\Tm(\A)=\{(t_1,\dots,t_r)\in T: t_i\in F^\times\A^{\times2}\}p(Z_{\GLt_r}(\A)),
\]
where $p$ is the canonical projection. Also let $\Tmt(\A)$ be the
metaplectic preimage of $\Tm(\A)$. Then $\Tmt(\A)$ is a maximal
abelian subgroup of $\Tt(\A)$. Moreover let $\omega^{\m}$ be any
character on $\Tmt(\A)$ extending the character on $Z_{\GLt_r}(\A)\Ttt(\A)$ given by
\[
\omega(\chit^{(2)}\otimest\cdots\otimest\chit^{(2)}),
\]
where $\omega$ is chosen so that it agrees with the metaplectic tensor
product $\chit^{(2)}\otimest\cdots\otimest\chit^{(2)} $ on the overlap. Then we have
\[
(\chit\otimest\cdots\otimest\chit)_\omega\cong\ind_{\Tmt(\A)}^{\Tt(\A)}\omega^\m,
\]
where $\ind$ is as in \cite[p. 54]{KP}.
\end{Lem}
\begin{proof}
The group $\Tt(\A)$ is a Heisenberg group and the group $\Tmt(\A)$ is
a maximal abelian subgroup by \cite[II.1.1]{KP}. Hence the lemma
follows from a general theory on the Heisenberg group as described in
\cite[p.52-56]{KP}.
\end{proof}

Let us note that in \cite{BG}, our $\Tm$ is denoted by $T^i$, and
our $\omega^\m$ by $\omega^i$, but we avoid to use $i$ in order not to
confuse it with the index $i$ we have been using.

With this lemma, one may assume
\begin{equation}\label{E:Heisenberg2}
\hat{f}^s\in \ind_{\Tmt(\A){N_B(\A)}^\ast}^{\GLt_r(\A)}
\omega^\m \otimes\delta_{B^{r-1,1}}^{1/4}\delta_Q^{s+\frac{1}{2}},
\end{equation}
where ${N_B(\A)}^\ast$ acts trivially as usual. 

Another important property of the group $T^\m(\A)$ is
\begin{Lem}\label{L:Tm}
The partial section $\s:\GL_r(\A)\rightarrow\GLt_r(\A)$ is not only
defined but a homomorphism on $T^\m(\A)N_B(\A)$. Also for $t\in
T^\m(\A)$ and $g\in\GL_r(F)$, both $\s(tg)$ and $\s(gt)$ are defined
and $\s(tg)=\s(t)\s(g)$ and $\s(gt)=\s(g)\s(t)$.
\end{Lem}
\begin{proof}
It is known that the partial section $\s$ is defined on the Borel
subgroup $B(\A)$ and the block-compatible cocycle is
globally defined on $B(\A)\times B(\A)$. Now if $t, t'\in T^\m(\A)$
and $n, n'\in N_B(\A)$, one can compute
\[
\sigma_r(tn, t'n')=\sigma_r(tnt^{-1}t, t'n')=\sigma_r(t,t')=1,
\]
where the last equality follows because the Hilbert symbol is trivial
on $F^\times\A^{\times 2}\times F^\times\A^{\times 2}$.

To show the second part of the lemma note that at every place $v$,
\[
s_r(t_vg_v)=\frac{s_r(t_v)s_r(g_v)}{\sigma_r(t_v,g_v)}\tau_r(t_v,g_v)
\]
and for almost all $v$ one can see that the right hand side is 1, and
hence the product $\prod_v s_r(t_vg_v)$ is defined, \ie $\s(tg)$ is
defined. Also this implies that globally we have
\[
s_r(tg)=\frac{s_r(t)s_r(g)}{\sigma_r(t,g)}\tau_r(t,g)=s_r(t)s_r(g)\tau_r(t,g)
\]
because $\sigma_r(t,g)=1$, which implies $\s(tg)=\s(t)\s(g)$. The same
argument works for $\s(gt)$.
\end{proof}

Now for each $i$, let us define the inclusion
\[
\iota_i:\GL_2\rightarrow\GL_r,\quad
g_2\mapsto
\begin{pmatrix}I_{r-i-1}&&&\\ &g_2&\\ &&I_{i-1}\end{pmatrix},
\]
so the first entry of $g_2\in\GL_2$ is in the $(r-i,
r-i)$-entry. This lifts to
\[
\iotat_i:\GLt_2(\A)\rightarrow\GLt_r(\A).
\]

With this notation, define
\begin{align*}
F_i^s(g_2;g)&:=c_i(\iotat_i(g_2)g,s;f^s)=\underset{\A^i}{\int}\;\hat{f}^s\Big(\s(w_i^{-1})
\s(\begin{pmatrix}I_{r-i-1}&&\\ &1& X_i\\
  &&I_i\end{pmatrix})\iotat_i(g_2)g;1\Big)\,dX_i\\
G_i^s(g_2;g)&:=c_{i-1}(\iotat_i(g_2)g,s;f^s)=\underset{\A^{i-1}}{\int}\;\hat{f}^s\Big(\s(w_{i-1}^{-1})
\s(\begin{pmatrix}I_{r-i}&&\\ &1& X_{i-1}\\
  &&I_{i-1}\end{pmatrix})\iotat_{i}(g_2)g;1\Big)\,dX_{i-1}.
\end{align*}
It should be noted that
\begin{equation}\label{E:F=G}
G_i^s(1;g)=F_{i-1}^s(1;g).
\end{equation}
Also we have
\begin{Lem}\label{L:F_and_G}
Assuming the integrals of $F_i^s(g_2;g)$ and $G_i^s(g_2;g)$ both converge, we
have
\[
F_i^s(g_2;g)=\int_{\A}G_i^s(\s_2\begin{pmatrix} &1\\
  1&\end{pmatrix} \s_2\begin{pmatrix} 1&x\\
  &1\end{pmatrix}g_2;g)\;dx,
\]
where $\s_2:\GL_2(\A)\rightarrow\GLt_2(\A)$ is the partial section for $\GLt_2(\A)$.
\end{Lem}
\begin{proof}
First note that $\iotat_i\circ\s_2=\s\circ\iota_i$, and hence 
\[
\iotat_i(\s_2\begin{pmatrix} &1\\
  1&\end{pmatrix} \s_2\begin{pmatrix} 1&x\\
  &1\end{pmatrix}g_2)=\s(\iota_i\begin{pmatrix}&1\\
  1&\end{pmatrix})
\s(\iota_i\begin{pmatrix} 1&x\\ &1\end{pmatrix})\iotat_i(g_2).
\]
Second note that
\begin{align}\label{E:w_i}
w_i=\begin{pmatrix}I_{r-i-1}&&&\\ &&&1\\ &1&&\\ &&I_{i-1}&\end{pmatrix}
&=\begin{pmatrix}I_{r-i-1}&&&\\ &&1&\\ &1&&\\ &&&I_{i-1}\end{pmatrix}
\begin{pmatrix}I_{r-i-1}&&&\\ &1&&\\ &&&1\\ &&I_{i-1}&\end{pmatrix}\\
\notag&=\iota_i\begin{pmatrix}&1\\
  1&\end{pmatrix} w_{i-1}.
\end{align}

Then one can check the lemma by a direct computation using Lemma \ref{L:section_s}.
\end{proof}

Now let $B^2=T_{B^2}N_{B^2}$ be the Borel subgroup of $\GL_2$ and $T_2^\m$ be
the analogous subgroup of $T_{B^2}$ as defined above with $r=2$. Let
$\omega_2^\m$ be the above $\omega^\m$ with $r=2$, so
$\ind_{\Tmt_2(\A)}^{\Tt_{B^2}(\A)}\omega_2^\m=\chit\otimest\chit$. (Note
that for $\GLt_2$ there is no choice for the central character
$\omega$ for the metaplectic tensor product because the center
$Z_{\GLt_2}$ is already contained in $\Ttt_{B^2}$, and hence we write
$\chit\otimest\chit$ instead of $(\chit\otimest\chit)_\omega$.)

Let us mention
\begin{Lem}
With the above notation, we have
\[
G_i^{s}(-;g)\in
\ind_{\Tmt_2(\A){N_{B^2}(\A)}^\ast}^{\GLt_2(\A)}\omega_2^\m\otimes
\delta_{B^2}^{t+\frac{1}{2}}
|\det|^u,
\]
where
\[
t=\frac{1}{2}rs+\frac{1}{8}-\frac{1}{4}i\quad\text{and}\quad
u=s-\frac{1}{2}rs-\frac{3}{8}r+\frac{3}{4}i,
\]
and hence in particular at $s=s_0=\frac{1}{4}-\frac{1}{2r}$ we have
\[
G_i^{s_0}(-;g)\in
\ind_{\Tmt_2(\A){N_{B^2}(\A)}^\ast}^{\GLt_2(\A)}\omega_2^\m\otimes\delta_{B^2}^{\frac{1}{4}(r-i+1)}
|\det|^{-\frac{1}{2}r+\frac{3}{4}i+\frac{1}{2}-\frac{1}{2r}},
\]
provided the integral is convergent.
\end{Lem}
\begin{proof}
Let
\[
g_2=\s(\iota_i(\begin{pmatrix}t_1&\\ &t_2\end{pmatrix})),
\]
where $t_i\in F^\times\A^{\times 2}$, so $\begin{pmatrix}t_1&\\
  &t_2\end{pmatrix}\in T_2^\m(\A)$. Then one can compute
\allowdisplaybreaks\begin{align*}
G_i^s(g_2;g)&=\underset{\A^{i-1}}{\int}\;\hat{f}^s\Big(\s(w_{i-1}^{-1})
\s(\begin{pmatrix}I_{r-i-1}&&&\\ &1&&\\ &&1& X_{i-1}\\
 & &&I_{i-1}\end{pmatrix})\s(\begin{pmatrix}I_{r-i-1}&&&\\ &t_1&&
 \\ &&t_2 &\\
  &&&I_{i-1}\end{pmatrix})g;1\Big)\,dX_{i-1}\\
&=\underset{\A^{i-1}}{\int}\;\hat{f}^s\Big(\s(w_{i-1}^{-1})
\s(\begin{pmatrix}I_{r-i-1}&&&\\ &t_1&&\\ &&t_2&\\
 & &&I_{i-1}\end{pmatrix})
\s(\begin{pmatrix}I_{r-i-1}&&&\\ &1&&\\ &&1& t_2^{-1}X_{i-1}\\
  &&&I_{i-1}\end{pmatrix})g;1\Big)\,dX_{i-1}\\
&=\underset{\A^{i-1}}{\int}\;\hat{f}^s\Big(\s(w_{i-1}^{-1})
\s(\begin{pmatrix}I_{r-i-1}&&&\\ &t_1&&\\ &&t_2&\\
 & &&I_{i-1}\end{pmatrix})
\s(w_{i-1})\s(w_{i-1}^{-1})\\
&\qquad\qquad\qquad\qquad\qquad\qquad
\s(\begin{pmatrix}I_{r-i-1}&&&\\ &1&&\\ &&1& t_2^{-1}X_{i-1}\\
  &&&I_{i-1}\end{pmatrix})g;1\Big)\,dX_{i-1},
\end{align*}
where for the second equality we used Lemma \ref{L:Tm} and for the
third equality we used $\s(w_i)^{-1}=\s(w_i^{-1})$ since
$\s$ is a group homomorphism on $\GL_r(F)$. (Note that in the above
computations, the global partial section $\s$ is always defined.)

Next by using Lemma \ref{L:Tm} and conjugating by $w_{i-1}$, we can compute 
\begin{align*}
\s(w_{i-1}^{-1})
\s(\begin{pmatrix}I_{r-i-1}&&&\\ &t_1&&\\ &&t_2&\\
 & &&I_{i-1}\end{pmatrix})
\s(w_{i-1})
=&\s(w_{i-1}^{-1})\s(w_{i-1})\s(\begin{pmatrix}I_{r-i-1}&&&\\ &t_1&&\\ &&I_{i-1}&\\
 & &&t_2\end{pmatrix}) \\
=&\s(\begin{pmatrix}I_{r-i-2}&&&\\ &t_1&&\\ &&I_{i-1}&\\
 & &&t_2\end{pmatrix}),
\end{align*}
where the last equality follows because $\s$ is a group homomorphism
on $\GL_r(F)$.

Therefore we have
\begin{align*}
&G_i^s(g_2;g)\\=&\underset{\A^{i-1}}{\int}\;\hat{f}^s\Big(
\s(\begin{pmatrix}I_{r-i-2}&&&\\ &t_1&&\\ &&I_{i-1}&\\
 & &&t_2\end{pmatrix})
\s(w_{i-1}^{-1})\s(\begin{pmatrix}I_{r-i-1}&&&\\ &1&&\\ &&1& t_2^{-1}X_{i-1}\\
  &&&I_{i-1}\end{pmatrix})g;1\Big)\,dX_{i-1}.
\end{align*}
Now by (\ref{E:Heisenberg2}) and the change of variables
$t_2^{-1}X_{i-1}\mapsto X_{i-1}$, one can see 
\begin{align*}
G_i^s(g_2;g)=&
|t_1|^{\frac{1}{4}(r-2(r-i))+s+\frac{1}{2}}|t_2|^{(1-r)(s+\frac{1}{2})+i-1}
\omega^\m(\s(\begin{pmatrix}I_{r-i-2}&&&\\ &t_1&&\\ &&I_{i-1}&\\
 & &&t_2\end{pmatrix}))\\
&\quad\underset{\A^{i-1}}{\int}\;\hat{f}^s\Big(
\s(w_{i-1}^{-1})\s(\begin{pmatrix}I_{r-i-1}&&&\\ &1&&\\ &&1& X_{i-1}\\
  &&&I_{i-1}\end{pmatrix})g;1\Big)\,dX_{i-1}.
\end{align*}
By direct computation, one can see
\[
\omega^\m(\s(\begin{pmatrix}I_{r-i-2}&&&\\ &t_1&&\\ &&I_{i-1}&\\
 & &&t_2\end{pmatrix}))
=\omega_2^\m(\s_2(\begin{pmatrix}t_1&\\ &t_2\end{pmatrix})).
\]
Then the lemma follows by simplifying the exponents for
$|t_1|$ and $|t_2|$.
\end{proof}

Let 
\begin{equation}\label{E:M_2(s)}
M_2(t):\ind_{\Tt_{B^2}^\m(\A)N_{B^2}(\A)}^{\GLt_2(\A)}\;\omega_2^\m\otimes
\delta_{B^2}^{t+\frac{1}{2}}|\det|^u
\rightarrow
\ind_{\Tt_{B^2}^\m(\A)N_{B^2}(\A)}^{\GLt_2(\A)}\;\omega_2^\m\otimes
\delta_{B^2}^{-t+\frac{1}{2}}|\det|^u
\end{equation}
be the intertwining operator defined on the induced space in the above
lemma. Then Lemma \ref{L:F_and_G} says that 
\begin{equation}\label{E:F_and_G}
F_i^s(g_2;g)=M_2(t)G_i^s(g_2;g),
\end{equation}
with
\[
t=\frac{1}{2}rs+\frac{1}{8}-\frac{1}{4}i,
\]
provided all the integrals are convergent, which is the case if $\Re(s)>>0$.

With this said, one can prove
\begin{Prop}
Let $i<r-2$. Then for each $g$ and $g_2$, the integral for $F_i^{s}(g_2;g)$ converges
absolutely at $s=s_0$. In particular $c_i(g,s;f^s)=F_i^s(1;g)$
converges at $s=s_0$.
\end{Prop}
\begin{proof}
We prove it by induction on $i$. For $i=1$, first note that
$G_1^s(g_2;g)=\hat{f}^s(\iotat_1(g_2)g;1)$, and hence certainly
$G_1^s(g_2;g)$ converges at any $s$. Now at $s=s_0$ by the previous
two lemmas, one can see that 
\[
F_1^s(g_2;g)=M_2(\frac{1}{4}(r-i-1)) G_1^s(g_2;g),
\]
where $M_2$ is as in (\ref{E:F_and_G}). But if $i=1$ (and $r>3$), the intertwining operator
$M_2(\frac{1}{4}(r-i-1))$ converges.

Now assume $F_i^s(g_2;g)$ converges for all $g_2$ and $g$ at $s=s_0$. Notice
that 
\[
G_{i+1}^s(g_2;g)=G_{i+1}^s(1;\iotat_{i+1}(g_2)g)=F_i^s(1;\iotat_{i+1}(g_2)g).
\]
Hence $G_{i+1}^s(g_2;g)$ converges for all $g_2$ and $g$. Again by the
previous two lemmas, one can see that 
\[
F_{i+1}^s(g_2;g)=M_2(\frac{1}{4}(r-i)-\frac{1}{2}) G_{i+1}^s(g_2;g),
\]
and if $i+1<r-2$, the intertwining operator here converges.
\end{proof}

Let us note that the convergence of $c_i(g,s;f^s)$ for $i<r-2$ is
stated without proof at the end of p.202 of \cite{BG} for the
non-twisted case. The author believes that the above proof is the one
they have in mind.

Finally we show the cancellation of the possible poles for
$c_{r-1}(g,s;f^s)$ and $c_{r-2}(g,s;f^s)$ at $s=s_0$, namely
\begin{Prop}
The sum
\[
c_{r-1}(g,s;f^s)+c_{r-2}(g,s;f^s)
\]
is holomorphic at $s=s_0$.
\end{Prop}
\begin{proof}
The proof is essentially described in the first half of p.203 of
\cite{BG}. But we will repeat the argument with our notations.  

First note that
\[
c_{r-2}(g,s;f^s)=F_{r-2}^s(1,g)=M_2(t)G_{r-2}^s(1;g),
\]
and by the above proposition, $G_{r-2}^s(1;g)=F_{r-3}^s(1;g)$ is
convergent. With $i=r-2$ and $s=s_0$, one can see
$t=\frac{1}{4}$. But at $t=\frac{1}{4}$, the intertwining operator
$M_2(t)$ has a simple pole. Then we have
\[
\underset{s=s_0}{\Res}c_{r-2}(g,s;f^s)=\underset{s=s_0}{\Res}F_{r-2}^s(1,g)
=\underset{s=s_0}{\Res}G_{r-1}^s(1,g),
\]
where for the last equality we used (\ref{E:F=G}). Second note that
\[
c_{r-1}(g,s;f^s)=F_{r-1}^s(1;g)=M(t)G_{r-1}^s(1;g),
\]
and with $i=r-1$ and $s=s_0$, one can see $t=0$, and we know by
Proposition \ref{P:action_at_s=0} that the intertwining operator
$M_2(0)$ is holomorphic and acts as $-\Id$.  Therefore we have
\[
\underset{s=s_0}{\Res}c_{r-1}(g,s;f^s)=M(0)\underset{s=s_0}{\Res}G_{r-1}^s(1;g)
=-\underset{s=s_0}{\Res}G_{r-1}^s(1;g)=-\underset{s=s_0}{\Res}
c_{r-2}(g,s;f^s)
\]
Hence the residues get cancelled out. 
\end{proof}

With those two propositions, we have proven that the constant term
(\ref{E:constant_term_even4}) is holomorphic at $s=s_0$, and hence the
Eisenstein series $E(g,s;f^s)$ is holomorphic at $s=s_0$.
\quad

%%%%%%%%%%%%%%%%%%%%%%%%%%%%%%%%%%%%%%%%%%%%%%%%%%%%%%%%%%%%%%%%%%%

\subsection{\bf The induction step for $r=2q+1$}

%%%%%%%%%%%%%%%%%%%%%%%%%%%%%%%%%%%%%%%%%%%%%%%%%%%%%%%%%%%%%%%%%%%

We now consider the case $r=2q+1$, namely $\theta=\vartheta_{\chi,
  \eta}$. First of all, let us note that if $\chi^{1/2}$ exists,
then $\vartheta_{\chi,\eta}=\theta_{\chi^{1/2},\eta}$, to which case
the argument for $r=2q$ applies. Hence we will assume that
$\chi^{1/2}$ does not exist. But even in this case, we argue similarly
to the case $r=2q$.

This time, however, the cuspidal support
of our Eisenstein series $E(g,s;f^s)$ is the $(2,\dots,2,1)$-parabolic,
and hence the poles of the Eisenstein series are
precisely the poles of the constant term along any parabolic
containing the $(2,\dots,2,1)$-parabolic. In particular in this subsection we let 
\[
P=P_{2,r-2}=M_PN_P\subseteq\GL_r
\]
be the $(2,r-2)$-parabolic, and will consider the constant term along
the unipotent radical $N_P$ of this parabolic. Similarly to the
computation for the case $r=2q$, the constant term along
$N_P$ is computed as follows:
\begin{align}\label{E:constant_term_odd1}
E_{P}(g,s; f^s)
&=\sum_{\gamma\in Q(F)\backslash\GL_r(F)\slash
  N_P(F)}\int_{N_P(F)^{\gamma^{-1}}\backslash
  N_P(\A)}f^s(\s(\gamma n)g;1)\,dn,
\end{align}
where $N_P(F)^{\gamma^{-1}}=N_P(F)\cap\gamma^{-1}Q(F)\gamma$. 
Since, by the Bruhat decomposition, we have
\[
\GL_r(F)=
\bigcup_{w\in\{1, w_1\}} Q(F)w^{-1} P(F),
\]
where $w_1$ is as in \eqref{E:w_1}, we have
\[
Q\backslash\GL_r\slash N_P
=\bigcup_{w\in\{1,w_1\}}M_P\cap wQw^{-1}\backslash M_P,
\]
where
\[
M_P\cap wQw^{-1}\backslash M_P=
\begin{cases}
P_{r-3, 1}^{r-2}\backslash \GL_{r-2},&\text{if $w=1$};\\
\overline{B}^2\backslash \GL_2,&\text{if $w=w_1$},
\end{cases}
\]
where $\GL_{r-2}$ is viewed as a subgroup of $\GL_r$ embedded in the
lower right corner, and $P_{r-3, 1}^{r-2}$ is the $(r-3, 1)$-parabolic
of $\GL_{r-2}$, and $\GL_2$ is embedded in the upper left corner
and $\overline{B}^2$ is the opposite of the standard Borel subgroup of
$\GL_2$. Using this decomposition, one can write (\ref{E:constant_term_odd1})
as
\begin{align*}
E_{P}(g,s;f^s)
=&E_P(g,s;f^s)_{\Id}+E_P(g,s;f^s)_{w_1},
\end{align*}
where we have set
\begin{align*}
E_P(g,s;f^s)_{\Id}&:=\sum_{m\in P_{r-3, 1}^{r-2}(F)\backslash\GL_{r-2}(F)}
\int_{N_P(F)\backslash N_P(\A)}f^s(\s(m n)g;1)\,dn;\\
E_P(g,s;f^s)_{w_1}&:=\sum_{m\in \overline{B}^2(F)\backslash \GL_{2}(F)}
\int_{N_P(F)^{m^{-1}w_1}\backslash N_P(\A)}f^s(\s(w_1^{-1}mn)g;1)\,dn.
\end{align*}

In what follows, we will show that the identity term
$E_P(g,s;f^s)_{\Id}$ is holomorphic for $\Re(s)\geq 0$ and the
non-identity term $E_P(g,s;f^s)_{w_1}$ vanishes, which will complete the
proof of Theorem \ref{T:main1}.

\quad\\

\noindent\underline{\bf The identity term $E_P(g,s;f^s)_{\Id}$}:

The argument for the identity term is quite similar to the case
$r=2q$ in that we interpret it as the Eisenstein series on the smaller
group $\GLt_{r-2}$ and use induction. Also let us mention that, as we
will see, unlike the case $r=2q$, we will not have to show the
cancellation of the pole at $s=\frac{1}{4}-\frac{1}{2r}$. 

First one can write
\begin{equation}\label{E:identity_term2}
E_P(g,s;f^s)_{\Id}
=\sum_{m\in P_{r-3, 1}^{r-2}(F)\backslash\GL_{r-2}(F)}
\int_{N_{P'}(F)\backslash
  N_{P'}(\A)}f^s(\s(m)g;\s(n'))\,dn',
\end{equation}
where 
\[
P'=P^{r-1,1}_{2,r-3,1}=P\cap M_Q=(\GL_2\times\GL_{r-3}\times\GL_1)N_{P'},
\]
so $P'$ is the parabolic subgroup of
$\GL_{r-1}\times\GL_1$ whose Levi is $\GL_2\times\GL_{r-3}\times\GL_1$.
If one views $f^s(\s(m)g;-)$ as an automorphic
form in $\theta\otimes\delta_Q^{s+1/2}$ on $\MQt(\A)$, the integral in the above
sum is just the constant term
along $N_{P'}$. But since $f^s(\s(m)g;-)\in\theta\otimes\delta_Q^{s+1/2}$,
we need to compute the constant term of the residual representation
$\theta\otimes\delta_Q^{s+1/2}$ along $N_{P'}$. Recall that the exceptional
representation $\theta$ is constructed as the residue of the
Eisenstein series associated to the
induced representation
$\Ind_{\Pt^{r-1,1}_{2,\dots,2,1}}^{\MQt}
(\rr_{\chi}\otimest\cdots\otimest\rr_{\chi}\otimest\etat)_\omega^\nu$
at $\nu=\rho_{P^{r-1,1}_{2,\dots,2,1}}/2$, where
$P_{2,\dots,2,1}^{r-1,1}$ is the $(2,\dots,2,1)$-parabolic  subgroup of
$\GL_{r-1}\times\GL_1$. Namely it is generated by
\[
\underset{\nu=\frac{1}{2}\rho_{P_{2,\dots,2,1}^{r-1,1}}}{\Res}E(-, \varphi^\nu;\nu)
\]
for $\varphi^\nu\in\Ind_{\Pt^{r-1,1}_{2,\dots,2,1}}^{\MQt}
(\rr_{\chi}\otimest\cdots\otimest\rr_{\chi}\otimest\etat)_\omega^\nu$. But
the constant
term of the residue is the same as the residue of the constant term,
and hence one first needs to compute the constant term $E_{P'}(-,
\varphi^\nu;\nu)$ of $E(-,
\varphi^\nu;\nu)$ along $N_{P'}$. For this, one can use
\cite[Proposition (ii), p.92]{MW}, and obtain
\[
E_{P'}(m',\varphi^\nu;\nu)
=\sum_w
E^{M_{P'}}(m',M(w,\nu)\varphi^\nu),
\]
where $m'\in\widetilde{M_{P'}}$ and $w$ runs through all the Weyl
group elements of $\GL_{r-1}\times\GL_1$
such that $w(\GL_2\times\cdots\times\GL_2\times\GL_1)w^{-1}$ is a standard
Levi of $\GL_2\times\GL_{r-3}\times\GL_1$ and $w^{-1}(\alpha)>0$ for
all the positive roots $\alpha$ that are
in $M_{P'}=\GL_2\times\GL_{r-3}\times\GL_1$. One can see that by using
the language of permutations, $w$
runs through all the elements of the form
\[
w=(12\cdots k),\quad \text{for $k=1,\dots,q$},
\]
where each permutation corresponds to a permutation of $\GL_2$-blocks
in the Levi $\GL_2\times\cdots\times\GL_2\times\GL_1$. (Note that the
last $\GL_1$ is always fixed by $w$.) But by exactly the same
reasoning as the case $r=2q$, one
can see that the Eisenstein series has a residue at
$\nu=\rho_{P_{2,\dots,2,1}^{r-1,1}}/2$ only for $w=(12\cdots q)$ by
using \cite[Proposition 2.42]{Takeda1}.

Now as in the case $r=2q$, we can see
\begin{align*}
\frac{1}{2}w\rho_{P_{2,\dots,2,1}^{r-1,1}}
&=-\frac{1}{2}\rho_{P^{r-1,1}_{2,r-3,1}}+\frac{1}{2}\rho_{P_{2,\dots,2,1}^{2,r-2,1}},
\end{align*}
where we note that $-\frac{1}{2}\rho_{P^{r-1,1}_{2,r-3,1}}$  is a character on the Levi
$M_{P^{r-1,1}_{2,r-3,1}}=\GL_2\times\GL_{r-3}\times\GL_1$ which acts as
\[
(a, g_{r-3}, b)\mapsto 
|\det a|^{-\frac{r-3}{4}} |{\det}(g_{r-3})|^{\frac{1}{2}}
\]
for $(a, g_{r-3}, b)\in \GL_2\times\GL_{r-3}\times\GL_1$. Hence for
$m'\in M_{P'}=\GL_2\times\GL_{r-3}\times\GL_1$, we have
\[
\underset{\nu=\frac{1}{2}\rho_{P_{2,\dots,2,1}^{r-1,1}}}{\Res}E^{M_{P'}}(m',M(w,\nu)\varphi^\nu)
=\underset{\nu'=\frac{1}{2}\rho_{P_{2,\dots,1,1}^{2,r-3,1}}}{\Res}E^{M_{P'}}(m',{\varphi'}^{\nu'})
\]
where 
\[
{\varphi'}^{\nu'}\in\Ind_{\Pt_{2,\dots,2,1}^{2,r-3,1}}^{\Mt_{P'}}(\rr_{\chi}|{\det}_2|^{-\frac{r-3}{4}}
\,\otimest\,(\underbrace{\rr_{\chi}\,\otimest\cdots\otimest\,\rr_{\chi}}
_{q-1 \text{ times}})|{\det}_{r-3}|^{\frac{1}{2}}\,\otimest\,\etat)_\omega^{\nu'}
\delta_{P_{2,\dots,2,1}^{r-1,1}}^{1/2}\delta_{P_{2,\dots 2,1}^{2,r-3,1}}^{-1/2}.
\]
By computing $\delta_{P_{2,\dots,2,1}^{r-1,1}}^{1/2}\delta_{P_{2,\dots 2,1}^{2,r-3,1}}^{-1/2}$,
this induced representation is written as
\[
\Ind_{\Pt_{2,\dots,2,1}^{2,r-3,1}}^{\Mt_{P'}}(\rr_{\chi}|{\det}_2|^{\frac{r-3}{4}}\,\otimest\,
(\rr_{\chi}\,\otimest\cdots\otimest\,\rr_{\chi})|{\det}_{r-3}|^{-\frac{1}{2}}\,
\otimest\,\etat)_\omega^{\nu'}.
\]
Then as in the case of $r=2q$, we see that 
\[
\underset{\nu'=\frac{1}{2}\rho_{P_{2,\dots,2,1}^{2,r-3,1}}}{\Res}E^{M_{P'}}(-,{\varphi'}^{\nu'})
\in(\rr_{\chi}|{\det}_2|^{\frac{r-3}{4}}\,\otimest\,\theta'|{\det}_{r-3}|^{-\frac{1}{2}}\,
\otimest\,\etat)_\omega,
\]
where $\theta'$ is the twisted exceptional representation of $\GLt_{r-3}$
which is the unique irreducible quotient of the induced representation
\[
\Ind_{\Pt_{2,\dots,2}^{r-3}}^{\GLt_{r-3}}
(\underbrace{\rr_{\chi}\,\otimest\cdots\otimest\,\rr_{\chi}}_{q-1\text{ times}})_{\omega'}
\otimes\delta_{P_{2,\dots,2}^{r-3}}^{1/4}
\]
for an appropriate choice of $\omega'$.

Hence for each fixed $m\in\GL_{r-2}(F)$, the function on
$\Mt_{P'}(\A)$ given by
\[
m'\mapsto \int_{N_{P'}(F)\backslash
  N_{P'}(\A)}f^s(\s(m)g;\s(n')m')\,dn'
\]
is an element in
$(\rr_{\chi}|{\det}_2|^{\frac{r-3}{4}}\,\otimest\,\theta'|{\det}_{r-3}|^{\frac{1}{2}}\,
\otimest\,\etat)_\omega\otimes\delta_Q^{s+\frac{1}{2}}$, where
$\delta_Q^{s+\frac{1}{2}}$ is actually the restriction
$\delta_Q^{s+\frac{1}{2}}|_{P'}$ to $P'$. One can compute
\[
\delta_Q^{s+\frac{1}{2}}(a, g_{r-3},
b)=|\det a|^{s+\frac{1}{2}}|{\det}(g_{r-2})|^{s+\frac{1}{2}}
|b|^{-(r-1)(s+\frac{1}{2})}.
\]
Accordingly the
function $f_{N_{P'}}^s$ on $\MPt(A)$ defined by
\[
f_{N_{P'}}^s(g)= \int_{N_{P'}(F)\backslash
  N_{P'}(\A)}f^s(g;\s(n'))\,dn'
\]
is in
\begin{align*}
&\Ind_{\widetilde{P'}(\A)}^{\MPt(\A)}(\rr_{\chi}|{\det}_2|^{\frac{r-3}{4}}\,\otimest\,\theta'
|{\det}_{r-3}|^{-\frac{1}{2}}\,\otimest\,\etat)_\omega\otimes
\delta_Q^{s+\frac{1}{2}}\delta_{P'}^{-1/2}\\
=&\Ind_{\widetilde{P'}(\A)}^{\MPt(\A)}(\rr_{\chi}|{\det}_2|^{-\frac{r-3}{4}}\,\otimest\,\theta'
|{\det}_{r-3}|^{\frac{1}{2}}\,\otimest\,\etat)_\omega\otimes\delta_Q^{s+\frac{1}{2}}.
\end{align*}

Again as in the $r=2q$ case, by restricting the section to
$\GLt_{r-3}$ (Proposition \ref{P:tensor_restriction}) one only has to consider the analytic
behavior of the $\GLt_{r-3}$ Eisenstein series 
\[
\sum_{m\in P_{r-3, 1}^{r-2}(F)\backslash\GL_{r-2}(F)}F^s(\s(m)g)
\]
where 
\[
F^s\in\Ind_{\Pt^{r-2}_{r-3,1}(\A)}^{\GLt_{r-2}(\A)}(\theta'
|{\det}_{r-3}|^{\frac{1}{2}}\,\otimest\,\etat)_{\omega}\otimes\delta_Q^{s+\frac{1}{2}}
\delta_{P'}^{1/2}\delta_{P_{r-3,1}^{r-2}}^{-1/2}
\]
and $\omega$ is some
appropriately chosen character and 
$\delta_Q^{s+\frac{1}{2}}\delta_{P'}^{1/2}\delta_{P_{r-3,1}^{r-2}}^{-1/2}$
is restricted to $P_{r-3,1}^{r-2}$. 
As a character on $\GL_{r-3}\times\GL_1\subseteq P_{r-3,1}^{r-2}$, one
can compute
\begin{align*}
|\det(g_{r-3})|^{\frac{1}{2}}
\delta_Q^{s+\frac{1}{2}}\delta_{P'}^{1/2}\delta_{P_{r-3,1}^{r-2}}^{-1/2} (g_{r-3},b)&=
|\det(g_{r-3})|^{s-\frac{1}{2}}|b|^{-(s+\frac{1}{2})(r-1)+\frac{1}{2}(r-3)}\\
&=\delta_{P_{r-3,1}^{r-2}}(g_{r-3}, b)^{\frac{rs+\frac{1}{4}}{r-1}}
|\det(g_{r-3})b|^{\frac{-\frac{1}{4}r-s}{r-1}}.
\end{align*}
Thus the section $F^s$ belongs to
\begin{equation}\label{E:induced_space_r-2}
\Ind_{\Pt^{r-2}_{r-3,1}(\A)}^{\GLt_{r-2}(\A)}(\theta'\otimest\,\etat)_{\omega}\otimes
\delta_{P_{r-3,1}^{r-2}}^b|{\det}_{r-2}|^{a},
\end{equation}
where
\begin{equation}\label{E:exponent_b2}
a=\frac{-\frac{1}{2}r-2s-\frac{1}{2}}{r-2},\qquad b=\frac{rs+\frac{3}{2}}{r-2}.
\end{equation}

Therefore the analytic behavior of the Eisenstein series in
(\ref{E:identity_term2}) is determined by that of the Eisenstein
series associated to the induced representation
(\ref{E:induced_space_r-2}). But the twist by $|{\det}_{r-2}|^{a}$
does not have any affect on the analytic behavior, and hence we have
to consider the Eisenstein series on $\GLt_{r-2}(\A)$ associated with
the induced space
\[
\Ind_{\Pt^{r-2}_{r-3,1}(\A)}^{\GLt_{r-2}(\A)}(\theta'\otimest\,\etat)_{\omega}\otimes
\delta_{P_{r-3,1}^{r-2}}^{b},
\]
where $b$ is as in (\ref{E:exponent_b2}).

Now by the induction hypothesis, this Eisenstein series on
$\GLt_{r-2}(\A)$ is holomorphic for $\Re(b)\geq 0$, which implies
that it is holomorphic for $\Re(s)\geq-\frac{3}{2r}$. (Note that the
induction is on $q$ and the base step is $q=1$ \ie $r=3$.) Thus
(\ref{E:identity_term2}) is holomorphic for $\Re(s)\geq 0$.

\quad

\noindent\underline{\bf The non-identity term $E_P(g,s;f^s)_{w_1}$}:

We will show that the non-identity term vanishes, \ie $E_P(g,s;f^s)_{w_1}=0$, which will
complete our proof. First note that the non-identity term is written as
\begin{align}\label{E:non-identity_term2}
E_P(g,s;f^s)_{w_1}&=\sum_{m\in \overline{B}^2(F)\backslash \GL_{2}(F)}
\int_{N_P(F)^{m^{-1}w_1}\backslash
  N_P(\A)}f^s(\s(w_1^{-1}mn)g;1)\,dn\\
&=\sum_{m\in \overline{B}^2(F)\backslash \GL_{2}(F)}
\int_{N_P(F)^{m^{-1}w_1}\backslash N_P(\A)}f^s(\s(w_1^{-1})\s(mn)g;1)\,dn,\notag
\end{align}
where we used Lemma \ref{L:section_s} and the fact that $\s$ is a homomorphism
on $\GL_r(F)$.  Then as we will see, each integral in the sum vanishes
due to the ``cuspidality''. But as we
just have done, to move round the section $\s$ we will frequently use Lemma
\ref{L:section_s} and the fact that $\s$ is a homomorphism on both
$\GL_r(F)$ and $N_B(\A)$, and the Weyl group elements are in
$\GL_r(F)$. The reader can verify that each manipulation on $\s$ can
be justified.

By the change of variable $mnm^{-1}\mapsto n$,
(\ref{E:non-identity_term2}) becomes
\begin{equation}\label{E:non-identity_term3}
\sum_{m\in \overline{B}^2(F)\backslash \GL_{2}(F)}
\int_{N_P(F)^{w_1}\backslash N_P(\A)}f^s(\s(w_1^{-1})\s(n)\s(m)g;1)\,dn.
\end{equation}
Let us introduce
\[
N_1=\{\begin{pmatrix}1&0&n_1\\
  &1&0
\\ &&I_{r-2}\end{pmatrix}\}
\quad\text{and}\quad
N_2=\{\begin{pmatrix}1&0&0\\
  &1&n_2
\\ &&I_{r-2}\end{pmatrix}\},
\]
\ie $N_1$ is the first row and $N_2$ is the second row of
$N_P$. Note that $N_P=N_2N_1$. Then by direct computation we see $N_P^{w_1}=N_2$.
Further let us write $N_2=UV$, where
\[
U=\{\begin{pmatrix}1&0&0&0\\ &1& 0& u\\
&&I_{r-3}&0\\ &&&1\end{pmatrix}\},
\quad\text{and}\quad
V=\{\begin{pmatrix}1&0&0&0\\ &1& v& 0\\
&&I_{r-3}&0\\ &&&1\end{pmatrix}\},
\]
where $v$ is $1\times (r-3)$ and $u$ is $1\times 1$, so
$n_2=(v,u)$. Note that $N_2=UV$. With those notations, each integral
in the sum of (\ref{E:non-identity_term3}) is written as
\begin{align}\label{E:non-identity_term4}\allowdisplaybreaks
\notag&\int_{N_1(\A)}\int_{N_2(F)\backslash N_2(\A)}
f^s(\s(w_1^{-1})\s(n_2)\s(n_1)\s(m)g;1)\,dn_2\,dn_1\\
=&\int_{N_1(\A)}\int_{V(F)\backslash V(\A)}\int_{U(F)\backslash U(\A)}
f^s(\s(w_1^{-1}vn_1)\s(m)g;\;\s(w_1^{-1})\s(u)\s(w_1))\,du\,dv\,dn_1
\end{align}
where for the last equality we used
\[
w_1^{-1}uw_1=
\begin{pmatrix}I_{r-3}&&&\\ 0&1&&\\0&u&1&\\ 0&0&0&1\end{pmatrix}
\in Q.
\]
Now let 
\[
w_2=\begin{pmatrix}I_{r-3}&0&0&0\\ 0&0&1&0\\0&1&0&0\\ 0&0&0&1\end{pmatrix},
\]
which is an element in $M_Q(F)$. By the automorphy of the automorphic
form $f^s(\s(w_1^{-1}vn_1)\s(m)g;-)$, the integrand of
(\ref{E:non-identity_term4}) is written as
\begin{align*}
&f^s(\s(w_1^{-1}vn_1)\s(m)g;\;\s(w_1^{-1})\s(u)\s(w_1))\\
=&f^s(\s(w_1^{-1}vn_1)\s(m)g;\;\s(w_2^{-1}w_1^{-1}uw_1w_2w_2^{-1}w_1^{-1})\s(w_1)),
\end{align*}
where we again used  Lemma \ref{L:section_s}. Note that
\[
w_2^{-1}w_1^{-1}uw_1w_2=
\begin{pmatrix}I_{r-3}&&&\\ 0&1&u&\\0&0&1&\\ 0&0&0&1\end{pmatrix}
\in N_B.
\]
Therefore the inner most integral of (\ref{E:non-identity_term4}) is written as
\[
\int_{F\backslash \A}
f^s(\s(w_1^{-1}vn_1)\s(m)g;\;\s(\begin{pmatrix}I_{r-3}&&&\\
  0&1&u&\\0&0&1&\\
  0&0&0&1\end{pmatrix})\s(w_2^{-1}w_1^{-1})\s(w_1))\,du
\]
But the cuspidal suport of the automorphic form
$f^s(\s(w_1^{-1}n_1)\s(m)g;-)$ is the
$(2,\dots,2,1)$-parabolic, which makes this integral vanish.\\

%%%%%%%%%%%%%%%%%%%%%%%%%%%%%%%%%%%%%%%%%%%%%%%%%%%%%%%%%%%%%%%%%%%

\section{\bf The normalized Eisenstein series}

%%%%%%%%%%%%%%%%%%%%%%%%%%%%%%%%%%%%%%%%%%%%%%%%%%%%%%%%%%%%%%%%%%%

Following \cite{BG}, we normalize the Eisenstein series by using the
denominators of (\ref{E:normalizing_factor_even}) and
(\ref{E:normalizing_factor_odd}). Namely we set
\begin{equation}\label{E:normalized_Eisenstein_series}
E^\ast(g, s;f^s)=
\begin{cases}L^S(r(2s+\frac{1}{2}),\chi^2\eta^{-2}) E(g, s;
  f^s),&\text{if   $\theta=\theta_{\chi, \eta}$};\\
L^S(r(2s+\frac{1}{2}),\chi\eta^{-2}) E(g, s; f^s),&\text{if
  $\theta=\vartheta_{\chi, \eta}$},
\end{cases}
\end{equation}
where $S=S(\chi,\eta, \omega, f^s)$ is a finite set of places containing all
the bad places with respect to $\chi, \eta, \omega$ and $f^s$; namely $S$
contains all the archimedean places, the places where $\chi, \eta$ or $\omega$
is ramified, the places dividing 2, and the places where $f^s$ is ramified. (Note that here
$\omega$ is the central character used to define the metaplectic
tensor product.) Let us emphasize that the normalization depends on
the set $S$, which in turn depends on $f^s$. 

Then we have
\begin{Thm}\label{T:main2}
Assume $r>2$. The above normalized Eisenstein series $E^\ast(g, s;f^s)$ is
holomorphic for all $s\in\C$, except that if $\chi^2\eta^{-2}=1$ and $r=2q$, or
$\chi\eta^{-2}=1$ and $r=2q+1$, it has a possible simple pole at 
$s=\frac{1}{4}$ and $-\frac{1}{4}$.
\end{Thm}

The idea of the proof is by now standard in that we use the functional
equation of the Eisenstein series and the analytic behavior of the normalized
intertwining operator $A^*(s,\theta,w_1)$ we obtained in Section
\ref{S:normalized_intertwining}. 

To use the functional
equation, however we need to introduce the
``opposite Eisenstein series'' as follows.
Let $\iQ=P_{1,r-1}^r$ be the standard $(1,r-1)$-parabolic of
$\GL_{r}$. Define a representation $\itheta$ of the Levi part $\iMQt$ by
\[
\itheta:=\begin{cases}(\eta\otimest\theta_\chi)_\omega,\quad{\text{if
    $r=2q$}}\\
(\eta\otimest\vartheta_\chi)_\omega,\quad{\text{if
    $r=2q+1$}},
\end{cases}
\]
where $\omega$ is arbitrary as long as
it satisfies the requirement for the metaplectic tensor product. Then
consider the global induced space
$\Ind_{\iQ(\A)}^{\GLt_r(\A)}\itheta\otimes\delta_{\iQ}^s$. Note that
this induced representation is nothing but the codomain of our
intertwining operator $A(s,\theta,w_1)$. Form the
corresponding ``opposite Eisenstein series''
\[
^\iota E(g,s;f^s)=\sum_{\gamma\in \iQ(F)\backslash\GL_r(F)}f^s(\s(\gamma)g;1)
\]
for $f^s\in \Ind_{\iQ(\A)}^{\GLt_r(\A)}\itheta\otimes\delta_{\iQ}^s$. Then we
have
\begin{Thm}\label{T:main_opposite}
The above Eisenstein series $^\iota
E(g,s;f^s)$ is holomorphic for $\Re(s)\geq 0$ except that it has a possible
simple pole at $s=\frac{1}{4}$ if $\chi^2\eta^{-2}= 1$ and $r=2q$,
or $\chi\eta^{-2}= 1$ and $r=2q$.
\end{Thm}
\begin{proof}
The proof is completely identical to Theorem \ref{T:main1}, except
that we need use the parabolic subgroup $^\iota Q$. However when one
applies the induction argument to compute the possible poles of the
Eisenstein series, one needs to induct ``from the bottom''. Namely, say
if $r=2q$, one
considers the constant term along the $(r-1,1)$-parabolic, and then the
identity term will be interpreted as the Eisenstein series on
$\GLt_{r-1}$ embedded in the upper left corner. To do so, one needs
part (b) of Proposition \ref{P:tensor_restriction}, and hence one needs to
assume that the metaplectic tensor product is realized as such.
\end{proof}

Now we are ready to provide 
\begin{proof}[Proof of Theorem \ref{T:main2}]
This follows from the functional equation together with the holomorphy
of the two Eisenstein series (Theorem
\ref{T:main1} and \ref{T:main_opposite}) and the holomorphy of the
normalized intertwining operator (Theorem
\ref{T:normalized_intertwining}). Though the argument seems to be
standard by now, we will repeat it in what follows. We only treat the
case $r=2q$, and the other case is identical. 

Note that since we know the holomorphy for $\Re(s)\geq 0$ by
Theorem \ref{T:main1} and the partial $L$-function
$L^S(r(2s+\frac{1}{2}),\chi^2\eta^{-2})$ is holomorphic for $\Re(s)\geq
0$, we only have to consider $\Re(s)<0$. Now by the functional
equation of the Eisenstein series, one has
\[
^\iota E(g,-s; A(s,\theta,w_1)f^s)=E(g,s;f^{s}).
\]
By multiplying the normalizing factor
$L^S(r(2s+\frac{1}{2}),\chi^2\eta^{-2})$ to both sides, we have
\begin{equation}\label{E:functional_equation}
^\iota E(g,-s; A^\ast(s,\theta,w_1)f^s)=L^S(r(2s+\frac{1}{2}),\chi^2\eta^{-2}) E(g,s;f^{s}),
\end{equation}
where the right hand side is nothing but the normalized Eisenstein
series $E^\ast(g,s,f^{s})$.

Assume $\chi^2\eta^{-2}\neq 1$. Then $A^\ast(s,\theta,w_1)f^s$
is holomorphic for all $s\in\C$ by Theorem \ref{T:normalized_intertwining},
and hence by Theorem \ref{T:main_opposite}, the left hand side (and
hence the right hand side) is
holomorphic for $\Re(s)\leq 0$. Assume that $\chi^2\eta^{-2}=1$. Then
$A^\ast(s,\theta,w_1)f^s$ is
holomorphic for $\Re(s)<0$ again by Theorem
\ref{T:normalized_intertwining}. But by Theorem \ref{T:main_opposite},
the left hand side (and the right hand side) of (\ref{E:functional_equation}) has a possible
simple pole at $s=-\frac{1}{4}$.
\end{proof}

%%%%%%%%%%%%%%%%%%%%%%%%%%%%%%%%%%%%%%%%%%%%%%%%%%%%%%%%%%%%%%%%%%%

\section{\bf The twisted symmetric square $L$-function}

%%%%%%%%%%%%%%%%%%%%%%%%%%%%%%%%%%%%%%%%%%%%%%%%%%%%%%%%%%%%%%%%%%%

Theorem \ref{T:main2} along with the integral representation of the
symmetric square $L$-function of $\GL_r$ obtained in \cite{Takeda1}
immediately implies
\begin{Thm}\label{T:main3}
Let $\pi$ be a cuspidal automorphic representation of $\GL_r(\A)$ with unitary
central character $\omega_\pi$ and $\chi$ a unitary Hecke
character. For a sufficiently large finite set $S$ of places, the (incomplete)
twisted symmetric square $L$-function $L^S(s,\pi, Sym^2\otimes\chi)$
is holomorphic everywhere except that if $\chi^r\omega_\pi^2=1$ it has a
possible simple pole at $s=0$ and $s=1$. Indeed, the set $S$ can be taken
to be precisely the set of archimedean places, places dividing 2 and non-archimedean
places at which either $\pi$ or $\chi$ is ramified.
\end{Thm}
\begin{proof}
Assume $r=2q+1$. Set
$\theta=\vartheta_{\chi\omega_\pi^2,\chi^{-q}}=
(\vartheta_{\chi\omega_\pi^2}\otimest\chit^{-q})_\omega$,
where $\omega$ is chosen appropriately as in
\cite[(2.57)]{Takeda1}. (As one can see from there, the bad places for
$\omega$ are dependent on the choice of the additive
character $\psi$. However, one can always choose $\psi$ so that the
bad places of $\omega$ are either dividing 2 or contained in those of $\chi$ or $\pi$.)
Then in \cite{Takeda1} we defined the zeta integral
\[
Z(\phi,\theta,f^s)=\underset{Z(\A)\GL_r(F)\backslash\GL_r(\A)}{\int}
\phi(g)\Theta(\kappa(g))E(\kappa(g),s;f^s)\,dg
\]
where $\phi\in\pi$,
$f^s\in\Ind_{\Qt(\A)}^{\GLt_r(\A)}\theta\otimes\delta_Q^s$, and $\Theta$
is an automorphic form in the twisted exceptional representation
$\theta_{\omega^{-1}_\pi}$ on $\GLt_r(\A)$. Now if
$f^s=f_\infty^s\otimes(\otimes'f_v^s)$ is an (almost) factorizable
section, we have shown in \cite{Takeda1} that
\[
L^S(2s+\frac{1}{2}, \pi, Sym^2\otimes \chi)Z_S(s)=L^S(r(2s+\frac{1}{2}),\chi^r\omega^2)Z(\phi,\theta,f^s)
\]
for an appropriately chosen $\phi$ and $\theta$, where $Z_S(s)$ is a
product of local zeta integrals. (Note that in \cite{Takeda1} the
induced representation for the Eisenstein series is not normalized and
hence there is a shift by $\frac{1}{2}$.) Here we may assume that the
local section $f_v^s$ is unramified if $v\notin S$. Moreover one can
choose the section so that $Z_S(s)$ is non-zero holomorphic, and hence the poles of
$L^S(2s+\frac{1}{2}, \pi, Sym^2\otimes \chi)$ are the poles of
$L^S(r(2s+\frac{1}{2}),\chi^r\omega^2)Z(\phi,\theta,f^s)$, which are
among the poles of the normalized Eisenstein series
$E^\ast(\kappa(g),s;f^s)$, where the normalization is with respect to
$S$. Hence by using Theorem \ref{T:main2}, we see that the incomplete
$L$-function $L^S(2s+\frac{1}{2}, \pi,
Sym^2\otimes \chi)$ is entire except that if $\chi^r\omega_\pi^2=1$ it
has a possible pole at $s=-\frac{1}{4}$ and $s=\frac{1}{4}$.

Assume $r=2q$ and $r>2$. (If $r=2$, the theorem is already well-known
by the work of Gelbart and Jacquet (\cite{GJ}).) Set
$\theta=\theta_{\omega_\pi,\chi^{-q}}=(\theta_{\omega_\pi}\otimest\chit^{-q})_\omega
=\theta_{\omega_\pi}\otimest\chit^{-q}$. (There
is no actual choice for $\omega$ here because the center of
$\GLt_{2q}$ is already contained in $\MQtt$, and that is why we simply
write $\theta_{\omega_\pi}\otimest\chit^{-q}$.) Then in \cite{Takeda1}
we have shown that the twisted symmetric square $L$-function is
represented by
\[
Z(\phi,\theta,f^s)=\underset{Z(\A)\GLtwo_r(F)\backslash\GLtwo_r(\A)}{\int}
\phi(g)\Theta(\kappa(g))E(\kappa(g),s;f^s)\,dg
\]
where the Eisenstein series $E(\kappa(g),s;f^s)$ is a restriction to
$\GLtt_r(\A)$ of the Eisenstein series associated with the induced
space $\Ind_{\Qt(\A)}^{\GLt_r(\A)}\theta\otimes\delta_Q^s$, and
$\Theta$ is a restriction of an automorphic form in the exceptional
representation $\vartheta_{\chi\omega_\pi^2}$. (It should be mentioned
that both the restriction of the Eisenstein series and the restriction
of the exceptional representation depends on the choice of $\psi$. See \cite{Takeda1} for
the details. But once again, one can choose $\psi$, so that the bad
places incurred by $\psi$ are either diving 2 or contained in the bad places of $\chi$ or
$\pi$.) Then by arguing as above, we can derive that the poles
of the twisted symmetric square $L$-function are among the poles of
the restriction of the normalized Eisenstein series
$E^\ast(\kappa(g),s;f^s)$, which are among the poles of the normalized
Eisenstein series itself. But by Theorem \ref{T:main2} we see that the
normalized Eisenstein series is entire except that if
$\chi^r\omega_\pi=1$ it has a possible simple pole at $s=-\frac{1}{4}$
and $s=\frac{1}{4}$. The proof is complete.
\end{proof}

\quad\\


\begin{thebibliography}{999999}

\bibitem[Ad]{Adams} J. Adams, {\it Extensions of tori in ${\rm
    SL}(2)$}, Pacific J. Math. 200 (2001), 257–271.

%\bibitem[AS1]{AS} M. Asgari and F. Shahidi, {\it Generic transfer for
%general spin groups}, Duke Math. J. 132, 137-190 (2006).

%\bibitem[AS2]{AS2} M. Asgari and F. Shahidi, {\it Functoriality for
 %   General Spin Groups}, preprint.

%\bibitem[BJ]{BJ} D. Ban and C. Jantzen, {\it The Langlands quotient
%theorem for finite central extensions of p-adic groups}, preprint.

\bibitem[B1]{Banks} W. Banks, {\it Exceptional representations of the
metaplectic group}, Ph.D Thesis, Stanford University (1994).

\bibitem[B2]{Banks2} W. Banks, {\it Twisted symmetric-square
$L$-functions and the nonexistence of Siegel zeros on ${\GL}(3)$},
Duke Math. J. 87 (1997), 343--353.

\bibitem[BLS]{BLS} W. Banks, J. Levy, M. Sepanski, {\it
Block-compatible metaplectic cocycles}, J. Reine Angew. Math. 507
(1999), 131--163.

%\bibitem[BZ]{BZ} I. N. Bernstein and A. V. Zelevinsky, {\it Induced
%representations of reductive $p$-adic groups}, I. Ann. Sci. Ecole
%Norm. Sup. (4) (1977), 441--472.

%\bibitem[BW]{BW} A. Borel and N. Wallach, {\it Continuous cohomology,
%    discrete subgroups, and representations of reductive groups}
%  Second edition. Mathematical Surveys and Monographs, 67. American
%  Mathematical Society, Providence, RI, 2000. xviii+260 pp.

%\bibitem[B]{Bump} D. Bump, {\it Automorphic forms and representations},
%Cambridge Studies in Advanced Mathematics, 55. Cambridge University
%Press, Cambridge, 1997. xiv+574.


\bibitem[BG]{BG} D. Bump and D. Ginzburg, {\it Symmetric square
$L$-functions on ${\rm GL}(r)$}, Ann. of Math. 136 (1992), 137--205.


\bibitem[F]{F} Y. Flicker, {\it Automorphic forms on covering groups
of ${\GL}(2)$}, Invent. Math. 57 (1980), 119--182.

\bibitem[FK]{FK} Y. Flicker and D. Kazhdan, {\it Metaplectic
[correspondence}, Inst. Hautes Etudes Sci. Publ. Math. No. 64 (1986), 53--110

\bibitem[G]{G} S. Gelbart, {\it Weil's representation and the spectrum
    of the metaplectic group}, Lecture Notes in Mathematics,
  Vol. 530. Springer-Verlag, Berlin-New York, 1976.

\bibitem[GJ]{GJ} S. Gelbart, and H. Jacquet, {\it A relation between
automorphic representations of ${\GL}(2)$ and ${\GL}(3)$},
Ann. Sci. Ecole Norm. Sup. (4) 11 (1978), 471--542.

\bibitem[GPS]{GPS} S. Gelbart and I. I. Piatetski-Shapiro, {\it
Distinguished representations and modular forms of half-integral
weight} Invent. Math. 59 (1980), 145--188.

\bibitem[GPSR]{GPSR} S. Gelbart, I. Piatetski-Shapiro, and S. Rallis,
  {\it Explicit constructions of automorphic L-functions}, Lecture
  Notes in Mathematics, 1254. Springer-Verlag, Berlin, 1987.

\bibitem[HT]{HT} M. Harris and R. Taylor, {\it The geometry and
cohomology of some simple Shimura varieties}, With an appendix by
Vladimir G. Berkovich, Annals of Mathematics Studies, 151. Princeton
University Press, Princeton, NJ, 2001.

\bibitem[He]{He} G. Henniart, {\it Une preuve simple des conjectures
de Langlands pour ${\GL}(n)$ sur un corps $p$-adique},
Invent. Math. 139 (2000), 439--455.

\bibitem[Ik]{Ikeda} {\it On the location of poles of the triple
    L-functions}, Compositio Math. 83 (1992), 187--237. 

%\bibitem[JS]{JS} H. Jacquet and J. Shalika, {\it Exterior square
%$L$-functions}, Automorphic forms, Shimura varieties, and
%$L$-functions, Vol. II (Ann Arbor, MI, 1988), 143--226,
%Perspect. Math., 11, Academic Press, Boston, MA, 1990.

\bibitem[Ji]{Jiang} D. Jiang, {\it Degree 16 standard $L$-function of
    $\GSp(2)\times\GSp(2)$} Mem. Amer. Math. Soc. 123 (1996), no. 588,
  viii+196 pp.

%\bibitem[K1]{Kable} A. Kable, {\it Exceptional representations of the
%metaplectic double cover of the general linear group}, PH.D thesis,
%Oklahoma State University (1997).

%\bibitem[K2]{Kable2} A. Kable, {\it The tensor product of exceptional
%    representations on the general linear group}, Ann. Sci. École
%  Norm. Sup. (4) 34 (2001), 741--769.

\bibitem[KP]{KP} D. A. Kazhdan and S. J. Patterson, {\it Metaplectic
forms}, Inst. Hautes Etudes Sci. Publ. Math. No. 59 (1984), 35--142.

%\bibitem[Ki]{Kim} H. Kim, {\it Langlands-Shahidi method and poles of
%automorphic $L$-functions: application to exterior square
%$L$-functions}, Canad. J. Math. 51 (1999), no. 4, 835--849.

\bibitem[Kub]{Kubota} T. Kubota, {\it On automorphic functions and the
    reciprocity law in a number field}, Lectures in Mathematics,
  Department of Mathematics, Kyoto University, No. 2 Kinokuniya
  Book-Store Co., Ltd., Tokyo 1969 iii+65 pp. 


%\bibitem[Kud]{Kudla} S. Kudla, {\it Tate's thesis} in An introduction
% to the Langlands program (Jerusalem, 2001), 109--131, Birkhauser
% Boston, Boston, MA, 2003.


\bibitem[KR]{KR} S. Kudla and S. Rallis, {\it Poles of
    Eisenstein series and L-functions},
Festschrift in honor of I. I. Piatetski-Shapiro on the occasion of his
sixtieth birthday, Part II, 81--110.


%\bibitem[L]{Langlands} R. Langlands, {\it The volume of the
%    fundamental domain for some arithmetical subgroups of Chevalley
%    groups},  Proc. Sympos. Pure Math., (1965) pp. 143--148.



%\bibitem[LS]{LS}
%H. Y. Loke and G. Savin, {\it Modular forms on non-linear double
 %covers of algebraic groups}, preprint.

%\bibitem[Mat]{Matsumoto} H. Matsumoto, {\it Sur les sous-groupes
%  arithmétiques des groupes semi-simples déployés}, Ann. Sci. École
%Norm. Sup.  2 (1969), 1--62.

\bibitem[Me]{Mezo} P. Mezo, {\it Metaplectic tensor products for
irreducible representations}, Pacific J. Math. 215 (2004), 85--96.

\bibitem[MW]{MW} C. Moeglin and J.-L. Waldspurger, {\it Spectral
    decomposition and Eisenstein series}, Cambridge Tracts in
  Mathematics, 113. Cambridge University Press, Cambridge, 1995.

\bibitem[PP]{PP} S. J.  Patterson and I. I.  Piatetski-Shapiro, {\it
The symmetric-square $L$-function attached to a cuspidal automorphic
representation of ${\rm GL}_3$}, Math. Ann. 283 (1989), 551--572.

\bibitem[PSR]{PSR} I.I. Piatetski-Shapiro and S. Rallis, {\it Rankin triple
    $L$ functions},  Compositio Math. 64 (1987), 31--115. 

\bibitem[R]{Rao}
R. Ranga Rao, {\it On some explicit formulas in the theory of Weil representation},
Pacific J. Math. 157 (1993), 335--371. 

%\bibitem[Sh1]{Sh81} F. Shahidi, {\it On certain $L$-functions},
%Amer. J. Math. 103 (1981), no. 2, 297--355.

\bibitem[Sh1]{Sh90} F. Shahidi, {\it A proof of Langlands' conjecture
on Plancherel measures; complementary series for $p$-adic groups}, 
 Ann. of Math. 132 (1990), 273--330.

\bibitem[Sh1]{Sh92} F. Shahidi, {\it Twisted endoscopy and
    reducibility of induced representations for p-adic groups}, Duke
  Math. J. 66 (1992), 1--41.


%\bibitem[Sh2]{Sh97} F. Shahidi, {\it On non-vanishing of twisted
%symmetric and exterior square $L$-functions for ${\GL}(n)$}, Olga
%Taussky-Todd: in memoriam.  Pacific J. Math. 1997, Special Issue,
%311--322.

%\bibitem[Sha]{Sha74} J. Shalika, {\it The multiplicity one theorem for
%${\GL}_{n}$}, Ann. of Math. 100 (1974), 171--193.


\bibitem[Shi]{Shimura} G. Shimura, {\it On the holomorphy of certain
Dirichlet series}, Proc. London Math. Soc. (3) 31 (1975), 79--98.


\bibitem[T1]{Takeda1} S. Takeda, {\it The twisted symmetric square
    $L$-function of $\GL(r)$},  Duke Math. J. 163 (2014), 175--266.

\bibitem[T2]{Takeda2} S. Takeda, {\it Metaplectic tensor products of
    automorphic representation of $\GLt(r)$}, Canad. J. Math. (to appear)

%\bibitem[We]{Weil} A. Weil, {\it Andre Sur certains groupes d'operateurs
%    unitaires}, Acta Math. 111 (1964) 143--211.

\end{thebibliography}
\end{document}